\documentclass[12pt,a4]{amsart}
\title{A decomposition formula for J-stability and its applications}
\author{Masafumi Hattori}
\address{Department of Mathematics, Faculty of Science, Kyoto University, Kyoto, 606-8502, Japan}
\email{hattori.masafumi.47z@st.kyoto-u.ac.jp}
\usepackage{amsthm}
\usepackage{amsmath}
\usepackage{amssymb}
\usepackage{amsfonts}
\usepackage{mathrsfs}
\usepackage{enumerate}
\usepackage{amscd}
\usepackage{graphicx}
\usepackage[all]{xy}
\usepackage[a4paper, top=3cm, bottom=3cm, left=2.5cm, right=2.5cm]{geometry}
\def\coloneq{\mathrel{\mathop:}=}
\newtheorem{thm}{Theorem}[section]
\newtheorem{lem}[thm]{Lemma}
\newtheorem{prop}[thm]{Proposition}
\newtheorem*{claim}{Claim}

\newtheorem{cor}[thm]{Corollary}
\newtheorem{teiri}{Theorem}

\newtheorem{kei}[teiri]{Corollary}
\newtheorem{fact}[teiri]{Fact}

\theoremstyle{definition}

\newtheorem{de}[thm]{Definition}
\newtheorem{ex}[thm]{Example}

\newtheorem{conj}[thm]{Conjecture}
\newtheorem{rem}[thm]{Remark}

\begin{document}

 \maketitle

 \begin{abstract}
For algebro-geometric study of J-stability, a variant of K-stability, we prove a decomposition formula of non-archimedean $\mathcal{J}$-energy of $n$-dimensional varieties into $n$-dimensional intersection numbers rather than $(n+1)$-dimensional ones, and show the equivalence of slope $\mathrm{J}^H$-(semi)stability and $\mathrm{J}^H$-(semi)stability for surfaces when $H$ is pseudoeffective.
Among other applications, we also give a purely algebro-geometric proof of a uniform K-stability of minimal surfaces due to \cite{JSS}, and provides examples which are J-stable (resp., K-stable) but not uniformly J-stable (resp., uniformly K-stable).
\end{abstract}

 \tableofcontents

 \section{Introduction}\label{Intro}
 
 Let $(X,L)$ be a polarized complex manifold. It is conjectured that the K-polystability of $(X,L)$ is equivalent to the existence of a constant scalar curvature K\"{a}hler (cscK for short) metric in $\mathrm{c}_1(L)$ (the Yau-Tian-Donaldson conjecture). $(X,L)$ is K-stable if the non-Archimedean Mabuchi functional $M^\mathrm{NA}>0$ for any non-almost trivial ample test configuration over $(X,L)$. We can decompose $M^\mathrm{NA}$ into the non-Archimedean entropy and the $(\mathcal{J}^{K_X})^\mathrm{NA}$-energy as non-Archimedean functional (cf. \cite{BHJ}). Chi Li \cite{Li} proved that if $(\mathcal{J}^{K_X})^\mathrm{NA}$-energy is non-negative, the Mabuchi functional in differential geometry is coercive and hence $(X,L)$ has a cscK metric. Thus $(\mathcal{J}^H)^\mathrm{NA}$-energy plays an important role in studying K-stability. On the other hand, Simon K. Donaldson \cite{D1} introduced J-flow and so-called J-equation (see Fact \ref{h}). Donaldson proposed the question when there is a solution to J-equation, i.e. there is a stationary solution of J-flow. Xiuxiong Chen \cite{xC} defined independently J-flow, which he called J-equation, and pointed out the solvability of this equation implies the coerciveness of Mabuchi functional of K\"{a}hler surfaces with ample canonical classes. This question is studied in \cite{SW}, \cite{LS}, \cite{CS}, \cite{DK}, \cite{HK}, \cite{G}, \cite{DP}, \cite{S}. Moreover, Mehdi Lejmi and G\'{a}bor Sz\'{e}kelyhidi \cite{LS} conjectured that the solvability of J-equation to $\mathrm{J}^{H}$-stability (i.e. $(\mathcal{J}^{H})^\mathrm{NA}(\mathcal{X},\mathcal{L})>0$ for any non-almost trivial ample test configuration $(\mathcal{X},\mathcal{L})$ over $(X,L)$, the notion of J-stability is introduced in \cite{LS}, \cite{HK}) and to another algebro-geometric condition. The conjecture is proved in \cite{G}, \cite{DP} and \cite{S}. We prove that if $X$ is an integral surface and $H$ is pseudoeffective, $(X,L)$ is $\mathrm{J}^H$-semistable iff $2\frac{H\cdot L}{L^{\cdot 2}}L-H$ is nef.
 
 On the other hand, examples in \cite{A} show that K-stability may not ensure the existence of a cscK metric and G. Sz\'{e}kelyhidi \cite{SP} proposed that uniform notion of K-stability should be equivalent to the existence of a cscK metric. Ruadha\'{i} Dervan \cite{D} and Boucksom-Hisamoto-Jonsson \cite{BHJ} define the uniform K-stability independently. In \cite{BHJ}, $(X,L)$ is uniformly K-stable iff there exists a positive real number $\delta$ such that $M^\mathrm{NA}(\mathcal{X},\mathcal{L})\ge\delta J^\mathrm{NA}(\mathcal{X},\mathcal{L})$ for any semiample test configuration $(\mathcal{X},\mathcal{L})$ over $(X,L)$. This definition is equivalent to the definition of uniform K-stability in \cite{D}. S. K. Donaldson \cite{skD} extend the definition of K-stability to log pairs. There is a folklore conjecture that the uniform K-stability is equivalent to the K-stability for log pairs. There is a partial consequence: a log Fano pair $(X,B,-K_{(X,B)})$ is K-stable if and only if it is uniformly K-stable due to \cite{LXZ} (for smooth Fano manifolds, cf. \cite{CDS}, \cite{T} and \cite{BBJ}).
 
 The notations are as in \S \ref{Notat}. In this paper, our main results on possibly singular surfaces are the followings with a purely algebro-geometric proof:

 \begin{teiri}[Theorem \ref{d2}]\label{c}
 Let $(X,L)$ be a polarized integral surface and $H$ be a pseudoeffective $\mathbb{Q}$-Cartier divisor such that
 \[
 2\frac{L\cdot H}{L^{\cdot 2}}L-H
 \]
 is nef. Then $(X,L)$ is $\mathrm{J}^H$-semistable.
 \end{teiri}
 
 Theorem \ref{c} tells us that the ampleness of $2\frac{L\cdot H}{L^{\cdot 2}}L-H$ is equivalent to uniform $\mathrm{J}^H$-stability if $H$ is big. To be precise, we have the following:
 \begin{kei}[Corollary \ref{jstable}, Corollary \ref{ms}, cf. Theorem 2 of \cite{xC}, Theorem 1.1 of \cite{G}]\label{d}
 For any polarized deminormal surface $(X,L)$ with a big (resp. pseudoeffective) $\mathbb{Q}$-line bundle $H$ such that for any irreducible components $X_i$, $\frac{L|_{X_i}\cdot H|_{X_i}}{(L|_{X_i})^2}=\frac{L\cdot H}{L^2}$, then the following are equivalent. 
\begin{itemize}
\item[(1)] $(X,L)$ is uniformly $\mathrm{J}^H$-stable (resp. $\mathrm{J}^H$-semistable). In other words, there exists $\epsilon>0$ such that for any semiample test configuration $(\mathcal{X},\mathcal{L})$
\[
(\mathcal{J}^H)^\mathrm{NA}(\mathcal{X},\mathcal{L})\ge \epsilon J^\mathrm{NA}(\mathcal{X},\mathcal{L})\quad (\mathrm{resp}.\, \ge0).
\] 
\item[(2)] $(X,L)$ is uniformly slope $\mathrm{J}^H$-stable (resp. slope $\mathrm{J}^H$ semistable). In other words, there exists $\epsilon>0$ such that for any semiample deformation to the normal cone $(\mathcal{X},\mathcal{L})$ along integral curve
\[
(\mathcal{J}^H)^\mathrm{NA}(\mathcal{X},\mathcal{L})\ge \epsilon J^\mathrm{NA}(\mathcal{X},\mathcal{L})\quad (\mathrm{resp}.\, \ge0).
\] 
\item[(3)] There exists $\epsilon>0$ such that for any integral curve $C$, 
\[
\left(2\frac{H\cdot L}{L^{\cdot 2}}L-H\right)\cdot C\ge \epsilon L\cdot C\quad (\mathrm{resp}.\, \ge0).
\]
\end{itemize}
 \end{kei}

 X.X. Chen \cite{xC} proves this for K\"{a}hler surfaces $X$ when $H=K_X$ is ample. We emphasize that we show that Corollary \ref{d} holds not only when $H$ is ample but big for singular surfaces.
 
 The assumption that for any irreducible components $X_i$, $\frac{L|_{X_i}\cdot H|_{X_i}}{(L|_{X_i})^2}=\frac{L\cdot H}{L^2}$ is necessary. Indeed, if the assumption does not hold, $X$ is $\mathrm{J}^H$-unstable as the following generalization of Theorem 7.1 of \cite{LW} shows.
 
 \begin{teiri}[Theorem \ref{demichan}]\label{i}
 Let $(X,\Delta;L)$ be an $n$-dimensional polarized deminormal pair with a $\mathbb{Q}$-line bundle $H$ such that $X=\bigcup _{i=1}^l X_i$ be the irreducible decomposition. Let also $L_i=L|_{X_i}$ and $H_i=H|_{X_i}$. Suppose that 
\[
\frac{H\cdot L^{\cdot n-1}}{L^{\cdot n}}\ne\frac{H_i\cdot L_i^{\cdot n-1}}{L_i^{\cdot n}}
\]
for some $1\le i\le l$. Then $(X,L)$ is $\mathrm{J}^H$-unstable.

Furthermore, let $\nu:\coprod_{i=1}^l X_i^\nu\to X$ be the normalization, $\overline{L_i}=L|_{X_i^\nu}$, $\overline{H_i}=H|_{X^\nu_i}$ and $K_{(X_i^\nu,\Delta_i)}=\nu^*(K_{(X,\Delta)})|_{X^\nu_i}$. Suppose that 
\[
\frac{K_{(X,\Delta)}\cdot L^{\cdot n-1}}{L^{\cdot n}}\ne\frac{K_{(X_i^\nu,\Delta_i)}\cdot \overline{L_i}^{\cdot n-1}}{\overline{L_i}^{\cdot n}}
\]
for some $1\le i\le l$. Then $(X,\Delta;L)$ is K-unstable.
 \end{teiri}
 
 Compare Corollary \ref{d} with the following, which is originally conjectured by Lejmi and Sz{\' e}kelyhidi. Gao Chen \cite{G} proved the uniform version of this conjecture, Datar-Pingali \cite{DP} proved for projective manifolds and Jian Song \cite{S} proved generally:

\begin{fact}[Theorem 1.1 of \cite{G}, Theorem 1.2 of \cite{DP}, Corollary 1.2 of \cite{S}]\label{h}
Notations as in \cite{G}. Given an $n$-dimensional K{\" a}hler manifold $M$. Let $\chi $ and $\omega_0$ be K{\" a}hler metrics on $M$ and $c_0>0$ be the constant such that
\begin{eqnarray*}
\int_M \chi \wedge \frac{\omega_0^{n-1}}{(n-1)!}=c_0\int_M\frac{\omega_0^n}{n!}.
\end{eqnarray*}
Then the followings are equivalent:
\begin{itemize}
\item[(1)] There exists a smooth function $\varphi$ such that
\[
\omega_{\varphi}=\omega_0+\sqrt{-1}\partial\bar{\partial}\varphi>0
\]
 satisfies the J-equation
\[
\mathrm{tr}_{\omega_{\varphi}}\chi= c_0.
\]
Moreover, such $\varphi$ is unique up to a constant;
\item[(2)] There exists a smooth function $\varphi$ such that $\varphi$ is the critical point of the $\mathcal{J}_\chi$ functional. Moreover, such $\varphi$ is unique up to a constant;
\item[(3)] The $\mathcal{J}_\chi$ functional is coercive. In other words, there exist a positive constant $\epsilon$ and
another constant $C$ such that $\mathcal{J}_\chi(\varphi) \ge \epsilon\mathcal{J}_{\omega_0}(\varphi)-C$;
\item[(4)] $(M,[\omega_0],[\chi])$ is uniformly J-stable. In other words, there exists a positive constant $\epsilon$
such that for any K{\" a}hler test configuration $(X,\Omega)$ (cf. \cite[Definition 2.10]{DR}, \cite[Definitions 3.2 and 3.4]{SjP}), the invariant $\mathcal{J}_{[\chi]}(X,\Omega)$ (cf. \cite[Definition 6.3]{DR}) satisfies
$\mathcal{J}_{[\chi]}(X,\Omega)\ge \epsilon \mathcal{J}_{[\omega_0]}(X,\Omega)$;
\item[(5)] $(M,[\omega_0],[\chi])$ is uniformly slope J-stable. In other words, there exists a positive constant $\epsilon$
such that for any subvariety $V$ of $M$, the deformation to normal cone of $V$
$(X,\Omega)$ (cf. \cite[Example 2.11 (ii)]{DR} satisfies $\mathcal{J}_{[\chi]}(X,\Omega)\ge \epsilon \mathcal{J}_{[\omega_0]}(X,\Omega)$;
\item[(6)]There exists a positive constant $\epsilon$ such that
\[
\int_V (c_0-(n-p)\epsilon)\omega^p_0-p\chi\wedge\omega_0^{p-1}\ge 0
\]
for any $p$-dimensional subvariety $V$ with $p=1,2,\cdots,n-1$.
\item[(7)]
\[
\int_V c_0\omega^p_0-p\chi\wedge\omega_0^{p-1}> 0
\]
for any $p$-dimensional subvariety $V$ with $p=1,2,\cdots,n-1$.
\end{itemize}
\end{fact}
  
 Recall that for K-stability case, slope stability was strictly weaker (cf. \cite[Example 7.6]{PR}). On the other hand, the novelty of Fact \ref{h} and Corollary \ref{d} is to show the validity of slope type stability for J-stability case, in the spirit of Ross-Thomas \cite{RT07}. Our purely algebro-geometric proof of Theorem \ref{c} is inspired from Fact \ref{h} and from \cite[Theorems 6.1 and 6.4]{RT07}. We remark that the assumption of Theorem \ref{c} is weaker than one of Fact \ref{h} for smooth projective surfaces.
 
 On the other hand,  $(3)\Rightarrow (2)$ in Corollary \ref{d} is not so trivial since $(2\frac{H\cdot L}{L^{\cdot 2}}L-H)\cdot C$ is not a slope $(\mathcal{J}^H)^\mathrm{NA}$-energy of the deformation to the normal cone of $C$ as we see in Example \ref{33}.
Indeed, we have the following:
 
 \begin{teiri}[Theorem \ref{JJJ}, Proposition \ref{p}]\label{e}
 There exists a smooth polarized surface $(X,L)$ with an ample divisor $H$ such that $(X,L)$ is $\mathrm{J}^H$-stable but {\rm not} uniformly slope $\mathrm{J}^H$-stable. In particular, there does not necessarily exists a solution to J-equation even when $(X,L)$ is $\mathrm{J}^H$-stable.
 \end{teiri}
 
 J. Song \cite{S} calls the conditions (6) and (7) of Fact \ref{h} uniform J-positivity and J-positivity respectively (cf. \cite[Definition 1.1]{S}), and actually proves that J-positivity and uniform J-positivity are equivalent to uniform J-stability. However, we prove J-positivity is not equivalent to J-stability.
 
 Although K-stability and uniform K-stability are equivalent in the case of Fano manifolds, we show that the equivalence does not hold in general. In fact, we prove the following:
 
  \begin{kei}[Corollary \ref{K}, Corollary \ref{DJ}]\label{f}
 There exist a polarized normal pair $(X,\Delta;L)$ and a polarized deminormal surface $(Z,M)$ such that they are K-stable but {\rm not} uniformly K-stable.
 \end{kei}
 
 On the other hand, we extend $(4)\Leftrightarrow (6)$ in Fact \ref{h} to higher dimensional deminormal schemes over $\mathbb{C}$ (see Theorem \ref{SW} and Remark \ref{remstable}). We apply this to obtain the following extension of Jian-Shi-Song \cite[Theorem 1.1]{JSS}:
 
 \begin{teiri}[Theorem \ref{lastcor}]\label{g}
 Let $(X,\Delta;L)$ be an $n$-dimensional klt log minimal model over $\mathbb{C}$, i.e, $K_{(X,\Delta)}$ is nef. Then $(X,\Delta;K_{(X,\Delta)}+\epsilon L)$ is uniformly K-stable for sufficiently small $\epsilon>0$. Furthermore, if $X$ is smooth and $\Delta=0$, $(X,K_X+\epsilon L)$ also has a cscK metric.
 \end{teiri}
 
 The existence of cscK metrics of smooth minimal models was already proved by Zakarias Sj\"{o}str\"{o}m Dyrefelt \cite{Zak} and J. Song \cite{S} but our proof is more algebro-geometric. In particular when $n=2$, we have a purely algebro-geometric proof of Theorem \ref{g}.
 
  The point of the proof of Theorem \ref{c} is the following:
 
  \begin{teiri}[Theorem \ref{slope}]\label{a}
 Given an $n$-dimensional polarized variety $(X,L)$, and an flag ideal $\mathfrak{a}=\sum_{i=0}^r\mathfrak{a}_it^i\subset \mathcal{O}_{X\times \mathbb{A}^1}$ such that $\mathfrak{a}_0\ne0$. Then there exists an alternation $\pi:X'\to X$ (i.e. $\pi$ is a generically finite and proper morphism) such that $X'$ is smooth and irreducible, $D_0$ is an snc divisor corresponding to $\pi^{-1}\mathfrak{a}_0$ and the integral closure $\overline{(\pi\times\mathrm{id}_{\mathbb{A}^1})^{-1}\mathfrak{a}}$ of the inverse image of $\mathfrak{a}$ to $X'\times\mathbb{A}^1$ satisfies the following condition (*) 
 \begin{itemize}
\item[(*)]\[
\overline{(\pi\times\mathrm{id}_{\mathbb{A}^1})^{-1}\mathfrak{a}}=\mathscr{I}_{D_0}+\mathscr{I}_{D_1}t+\cdots +\mathscr{I}_{D_{r-1}}t^{r-1}+t^r,
\]
where each $\mathscr{I}_{D_i}$ is a coherent ideal sheaf corresponding to a Cartier divisor $D_i$ of $X'$. Furthermore, for each $m\in \mathbb{Z}_{>0}$, $$\biggl(\overline{(\pi\times\mathrm{id}_{\mathbb{A}^1})^{-1}\mathfrak{a}}\biggr)^m=\sum_{k=0}^{mr} t^k\mathscr{I}_{m,k},$$ where $\mathscr{I}_{m,k}=\mathscr{I}_{D_j}^{m-i}\cdot \mathscr{I}_{D_{j+1}}^{i}$ for $j=\lfloor \frac{k}{m}\rfloor$ and $i=k-mj$.
\end{itemize}
 Moreover, if $(\mathrm{Bl}_{\mathfrak{a}}(X\times \mathbb{A}^1),L_{\mathbb{A}^1}-E)$ is a semiample test configuration, where $\Pi:\mathrm{Bl}_{\mathfrak{a}}(X\times \mathbb{A}^1)\to X\times \mathbb{A}^1$ is the blow up along $\mathfrak{a}$ with the exceptional divisor $E=\Pi^{-1}(\mathfrak{a})$, $(\pi\times\mathrm{id}_{\mathbb{A}^1})^{-1}\mathfrak{a}\cdot (\pi\times\mathrm{id}_{\mathbb{A}^1})^*L_{\mathbb{A}^1}$ is semiample and $\pi^*L-D_0$ is nef. 
 \end{teiri}
 
This theorem is the technical heart of this note. In the above condition (*), we remark that $$\biggl(\overline{(\pi\times\mathrm{id}_{\mathbb{A}^1})^{-1}\mathfrak{a}}\biggr)^m\subset\sum_{k=0}^\infty t^k\mathscr{I}_{m,k}$$ is not trivial. Due to Theorem \ref{a}, we have the following decomposition formula:
 
 \begin{teiri}[Theorem \ref{sl}, Theorem \ref{impl}]\label{b}
 $(X,L)$, $\mathfrak{a}$ and $\pi$ are as in Theorem \ref{a}. Suppose that $(\mathrm{Bl}_{\mathfrak{a}}(X\times \mathbb{A}^1),L_{\mathbb{A}^1}-E)$ is a semiample test configuration, $n\ge 2$ and $\mathrm{deg}\,\pi=l$. Then we can calculate $(\mathcal{J}^H)^\mathrm{NA}(\mathrm{Bl}_{\mathfrak{a}}(X\times \mathbb{A}^1),L_{\mathbb{A}^1}-E)$ by using the mixed multiplicities of $\mathscr{I}_{D_k}$ (see Definition \ref{mazemaze}):
 \begin{align*}
L^{\cdot n}&(\mathcal{J}^{H})^\mathrm{NA}(\mathrm{Bl}_{\mathfrak{a}}(X\times \mathbb{A}^1),L_{\mathbb{A}^1}-E)\\
=&\frac{1}{l}\left(\frac{n}{n+1}\frac{H\cdot L^{\cdot n-1}}{L^{\cdot n}}\sum _{k=0}^{r-1}\sum_{j=0}^{n}e_{\pi^*L}(\mathscr{I}_{D_k}^{[j]},\mathscr{I}_{D_{k+1}}^{[n-j]})-\sum _{k=0}^{r-1}\sum_{j=0}^{n-1}e_{\pi^*L_{|\pi^*H}}(\mathscr{I}_{D_k|\pi^*H}^{[j]},\mathscr{I}_{D_{k+1}|\pi^*H}^{[n-1-j]})\right)\\
=&\frac{1}{l}\left(\frac{n}{n+1}\frac{H\cdot L^{\cdot n-1}}{L^{\cdot n}}\sum _{k=0}^{r-1}\sum_{j=0}^{n}((\pi^*L^{\cdot n})-(\pi^*L-D_k)^{\cdot j}\cdot (\pi^*L-D_{k+1})^{\cdot n-j})\right.\\
-&\left. \sum _{k=0}^{r-1}\sum_{j=0}^{n-1}\pi^*H\cdot((\pi^*L^{\cdot n-1})-(\pi^*L-D_k)^{\cdot j}\cdot (\pi^*L-D_{k+1})^{\cdot n-1-j})\right).
\end{align*}
 \end{teiri}
 
 Due to Theorem \ref{b}, we can calculate $(\mathcal{J}^H)^\mathrm{NA}$-energy by using the intersection numbers of $X$ itself rather than of its test configuration.
  
 \subsection*{Outline of this paper}
 
 In \S \ref{Notat}, we prepare many terminology and facts about K-stability and J-stability. From \S \ref{Mixed}, we state our original result.
 
 In \S \ref{Mixed}, we state the definition of the mixed multiplicities and prove the modification of the consequences of Mumford \cite{M}. Furthermore, we prove Theorem \ref{b} when $\pi$ is an isomorphism.
 
 In \S \ref{NewTro}, we prepare many terminology about toroidal embeddings in \cite{KKMS} and introduce the notion of Newton polyhedron. Moreover, we explain about our proof of Theorems \ref{a} and \ref{b} briefly in Example \ref{yokuwakaru}.
 
 In \S \ref{ThMain}, we restate Theorems \ref{a} and \ref{b}, and prove them by using the consequences of \S \ref{Mixed} and \S \ref{NewTro}. We remark that Theorem \ref{b} follows from Theorem \ref{a}.
 
 In \S \ref{App1}, we study the J-stability of surfaces. In \S \ref{App11}, we study irreducible surfaces and check that the mixed multiplicities are non-negative when $L$ and $H$ satisfy some conditions. Next, we apply Theorem \ref{b} to obtain Theorem \ref{c} and Corollary \ref{d} for irreducible and normal surfaces.
 
  In \S \ref{Demi}, we discuss about J-stability for reducible surfaces and construct a deminormal surface on which Hodge index theorem does not hold. Furthermore, we prove Theorem \ref{i} and Corollary \ref{d} completely.
  
  In \S\ref{App2}, we construct a smooth polarized surface $(X,L)$ and its ample divisor $H$ such that $(X,L)$ is slope $\mathrm{J}^H$-stable but not uniformly $\mathrm{J}^H$-stable. Then we apply Theorem \ref{b} and prove that $(X,L)$ is also $\mathrm{J}^H$-stable. Finally, Corollary \ref{f} follows immediately from Theorem \ref{e}.
  
  In \S \ref{Fib}, we extend the consequences of Hashimoto-Keller \cite[Theorem 3]{HK} and of Kento Fujita \cite[Theorem 6.5]{Fu}. On the other hand, we extend Lejmi-Sz\'{e}kelyhidi conjecture (see Theorem \ref{SW}) to non-smooth varieties over $\mathbb{C}$ and discuss on the application of J-stability for higher dimensional varieties. We apply the criteria to obtain the extension of the notion of stability threshold of Sj\"{o}str\"{o}m Dyrefelt \cite{Sj} to deminormal schemes and to prove Theorem \ref{g}.
 
 \section*{Acknowledgements}
 
 I can not express enough my sincere and deep gratitude to Associate Professor Yuji Odaka at Kyoto University who is my research advisor for a lot of suggestive advices, productive discussions and reading the draft. I also would like to thank Professors Gao Chen, Ruadha\'{i} Dervan, Kento Fujita, Julius Ross and Richard Thomas for helpful comments and warmful encouragement. I am also grateful to Professor Zakarias Sj\"{o}str\"{o}m Dyrefelt for productive discussions and letting me know his great paper \cite{Sj}, which led me to improve Theorem \ref{JJJ}.
 
 I am partially supported by Iwadare Scholarship Foundation.
 
 \section{Notations}\label{Notat}
Unless otherwise stated, we work over an algebraically closed field $k$ of characteristic $0$. We follow the notations used in \cite{BHJ}. If $V$ is a scheme, we assume $V$ to be an algebraic scheme (i.e. a scheme of finite type over $k$) that might not be proper. $V$ is a variety if $V$ is an irreducible and reduced algebraic scheme. On the other hand, $(X,L)$ be an $n$-dimensional polarized scheme if $X$ is an algebraic scheme and $L$ is an ample $\mathbb{Q}$-line bundle on $X$. Furthermore, we assume that $X$ is proper, reduced and equidimensional. Furthermore, $X$ is deminormal if $X$ satisfies that Serre's condition $S_2$ and codimension $1$ points of $X$ are either regular points or nodes (cf. Definition 5.1 of \cite{Ko}). A flat and proper family $\pi:(\mathcal{X},\mathcal{L})\to \mathbb{A}^1$ is a {\it (semi)ample test configuration over} $(X,L)$ if $\pi$ is $\mathbb{G}_m$-equivariant when $\mathbb{G}_m$ acts on $\mathbb{A}^1$ canonically, $\mathcal{L}$ is a $\mathbb{G}_m$-linearized (semi)ample $\mathbb{Q}$-line bundle, and there exists a $\mathbb{G}_m$-equivariant isomorphism $$(\mathcal{X},\mathcal{L})\times_{\mathbb{A}^1}(\mathbb{A}^1-\{0\})\cong (X\times(\mathbb{A}^1-\{0\}),L\otimes\mathcal{O}_{\mathbb{A}^1-\{0\}})$$ (cf. Definitions 2.1, 2.2 of \cite{BHJ}). Denote the trivial test configuration by $(X_{\mathbb{A}^1},L_{\mathbb{A}^1})$.
 \begin{de}
 Let $B$ be a boundary of $X$ if $X$ is deminormal. In other words, there exists $l\in \mathbb{Z}_{\ge0}$ and a finite number of $B=\sum_{i=1}^l a_iD_i$ for $0\le a_i\le 1$ such that $K_{(X,B)}=K_X+B$ is $\mathbb{Q}$-Cartier, where $D_i$'s are integral closed subschemes of codimension 1. Suppose that $(\mathcal{X},\mathcal{L})$ dominates $(X_{\mathbb{A}^1},L_{\mathbb{A}^1})$. In other words, there exists a $\mathbb{G}_m$-equivariant morphism $\rho:\mathcal{X}\to X_{\mathbb{A}^1}$. Let $(\overline{\mathcal{X}},\overline{\mathcal{L}})$ be the canonical compactification over $\mathbb{P}^1$ (cf. Definition 2.7 of loc.cit). Let also $p:\overline{\mathcal{X}}\to X$ be the canonical projection and $\overline{\mathcal{B}}$ be the closure of $B\times (\mathbb{P}^1\setminus\{0\})$ in $\mathcal{X}$.
 \[
 V(L)=L^{\cdot n}, \quad S(X,B;L)=-nV(L)^{-1}(K_{(X,B)}\cdot L^{\cdot n-1}).
 \]
 If there is no confusion, we denote $L^n=L^{\cdot n}$. Then, we define as follows:
  \begin{itemize}
 \item $E^\mathrm{NA}(\mathcal{X},\mathcal{L})=V(L)^{-1}\frac{\overline{\mathcal{L}}^{n+1}}{(n+1)}$;
 \item $R_{B}^\mathrm{NA}(\mathcal{X},\mathcal{L})=V(L)^{-1}(p^*K_{(X,B)}\cdot \overline{\mathcal{L}}^n)$;
  \item $I^\mathrm{NA}(\mathcal{X},\mathcal{L})=V(L)^{-1}(\overline{\mathcal{L}}\cdot(p^*L)^n-(\overline{\mathcal{L}}-p^*L)\cdot \overline{\mathcal{L}}^n)$;
   \item $J^\mathrm{NA}(\mathcal{X},\mathcal{L})=V(L)^{-1}(\overline{\mathcal{L}}\cdot(p^*L)^n)-E^\mathrm{NA}(\mathcal{X},\mathcal{L})$;
   \item $(\mathcal{J}^H)^\mathrm{NA}(\mathcal{X},\mathcal{L})=V(L)^{-1}(p^*H\cdot \overline{\mathcal{L}}^n-\frac{nH\cdot L^{n-1}}{(n+1)L^n}\overline{\mathcal{L}}^{n+1})$,
 \end{itemize}
 where $H$ is a $\mathbb{Q}$-line bundle on $X$. If there is no confusion, we denote $p^*L=\rho^*L_{\mathbb{P}^1}$ by $L_{\mathbb{P}^1}$.
 
 On the other hand, we can see the functionals above are invariant under taking a {\it pullback}. In other words, if $\mu:(\mathcal{X}',\mathcal{L}')\to(\mathcal{X},\mathcal{L})$ is a $\mathbb{G}_m$-equivariant morphism of test configurations and $\mathcal{L}'=\mu^*\mathcal{L}$, the functionals take the same value on them. For example, we have $\overline{\mathcal{L}}^{n+1}=\overline{\mathcal{L}'}^{n+1}$. Therefore, we can introduce the equivalence relation generated by pulling back and let $\mathcal{H}^\mathrm{NA}(L)$ be the set consists of all the equivalence class of semiample test configurations. We call $\mathcal{H}^\mathrm{NA}(L)$ the set of all {\it non-Archimedean positive metrics on} $L$ (cf. Definition 6.2 of loc.cit). According to the fact we will explain after Definition \ref{fldef} that for any metric $\phi\in\mathcal{H}^\mathrm{NA}(L)$, we can necessarily take a semiample test configuration $(\mathcal{X},\mathcal{L})$ that is a representative of $\phi$ and satisfies that $\mathcal{X}$ dominates $X_{\mathbb{A}^1}$, we can consider $E^\mathrm{NA},R_B^\mathrm{NA},I^\mathrm{NA},J^\mathrm{NA}$ and $(\mathcal{J}^H)^\mathrm{NA}$ as well-defined functionals of $\mathcal{H}^\mathrm{NA}(L)$.
  \end{de}
 
 \begin{de}[J-stability]
 The polarized pair $(X,B;L)$ is $\mathrm{J}^H$-semistable (resp. $\mathrm{J}^H$-stable, uniformly $\mathrm{J}^H$-stable) if $(\mathcal{J}^H)^\mathrm{NA}\ge0$ (resp. $(\mathcal{J}^H)^\mathrm{NA}>0$, $(\mathcal{J}^H)^\mathrm{NA}\ge \delta J^\mathrm{NA}$ for some $\delta>0$) on nontrivial metrics of $\mathcal{H}^\mathrm{NA}(L)$. If there is no confusion, we say that J-semistable (resp. J-stable, uniformly J-stable).
 \end{de}
 
 \begin{de}[K-stability for deminormal pairs]\label{genki}
  If $X$ is deminormal and $K_{(X,B)}$ is $\mathbb{Q}$-Cartier, we can define
  \[
 K^{\mathrm{log}}_{(\overline{\mathcal{X}},\overline{\mathcal{B}})/\mathbb{P}^1}=K_{(\overline{\mathcal{X}},\overline{\mathcal{B}})}-(\mathcal{X}_0-\mathcal{X}_{0,\mathrm{red}})
  \]
   for any deminormal semiample test configuration $(\mathcal{X},\mathcal{L})$ that dominates $X_{\mathbb{A}^1}$ (since $\mathcal{X}$ is Gorenstein of codimension 1 and $S_2$. see also \cite[\S 3]{Fu}). Therefore we define the {\it non-Archimedean entropy functional} as 
  \[
  H_{B}^\mathrm{NA}(\mathcal{X},\mathcal{L})=V(L)^{-1}(K^{\mathrm{log}}_{(\overline{\mathcal{X}},\overline{\mathcal{B}})/\mathbb{P}^1}\cdot \overline{\mathcal{L}}^n)-R_{B}^\mathrm{NA}(\mathcal{X},\mathcal{L})
  \]
   For any non-Archimedean metric $\phi$ and its representative $(\mathcal{X},\mathcal{L})$, take the partial normalization $(\widetilde{\mathcal{X}},\widetilde{\mathcal{L}})$ of $(\mathcal{X},\mathcal{L})$ (cf. \cite[Definition 3.7, Proposition 3.8]{Od}). Note that $\widetilde{\mathcal{X}}$ is a deminormal test configuration by \cite[Proposition 3.3]{Fu} and define
   \[
    H_{B}^\mathrm{NA}(\phi)\coloneq H_{B}^\mathrm{NA}(\widetilde{\mathcal{X}},\widetilde{\mathcal{L}}).
   \]
    Then, we also define the {\it non-Archimedean Mabuchi functional} as $$M_{B}^\mathrm{NA}=H_{B}^\mathrm{NA}+R_{B}^\mathrm{NA}+S(X,B;L)E^\mathrm{NA}.$$ The polarized pair $(X,B;L)$ is K-semistable (resp. K-stable, uniformly K-stable) if $M_{B}^\mathrm{NA}\ge0$ (resp. $M_{B}^\mathrm{NA}>0$, $M_{B}^\mathrm{NA}\ge \delta J^\mathrm{NA}$ for some $\delta>0$) on $\mathcal{H}^\mathrm{NA}(L)$.
\end{de}

 \begin{rem}
As usual, to define K-stability for general polarized schemes, we use the Donaldson-Futaki invariant (cf. \cite[Definition 3.3, Proposition 3.12]{BHJ}). Denote the Donaldson-Futaki invariant for log polarized pairs $(X,B;L)$ by $\mathrm{DF}_B$. Due to \cite[Remark 3.19, \S 7.3, Proposition 8.2]{BHJ}, $\mathrm{DF}_B\ge 0$ (resp. $>0$, $\ge\delta J^\mathrm{NA}$ for some $\delta>0$) iff $M^\mathrm{NA}_B\ge0$ (resp. $>0$, $\ge\delta J^\mathrm{NA}$), and hence we can define K-stability by using non-Archimedean Mabuchi functional. 
\end{rem}

\begin{de}[Log discrepancy cf. \S 1.5 \cite{BHJ}]\label{dfn}
For any divisorial valuation $v$ on a normal pair $(X,B)$, we can take a proper birational morphism $\mu:Y\to X$ of normal varieties such that there exists a prime divisor $F$ of $Y$ such that $v=c\,\mathrm{ord}_F$. We define the {\it log discrepancy} of $v$ as 
\[
A_{(X,B)}(v)=c(1+\mathrm{ord}_F(K_Y-K_{(X,B)})).
\] 
This is independent of the choice of $\mu:Y\to X$. Moreover, by Theorem 4.6 and Corollary 7.18 of \cite{BHJ}, 
\[
  H_{B}^\mathrm{NA}(\mathcal{X},\mathcal{L})=V(L)^{-1}\sum_Eb_EA_{(X,B)}(v_E)(E\cdot \mathcal{L}^n),
  \]
  where $E$ runs over the irreducible components of $\mathcal{X}_0$, $b_E=\mathrm{ord}_E(\mathcal{X}_0)>0$ and $v_E$ is the restriction of $b_E^{-1}\mathrm{ord}_E$ to $X$. If $E_0$ is the trivial divisor i.e. $E_0$ is the strict transformation of $X\times\{0\}$, $v_{E_0}=0$ and we define $A_{(X,B)}(v_{E_0})=0$.
\end{de}

\begin{de}[Conductor subscheme]
Let $(X,B;L)$ be a deminormal polarized pair and $\nu:(\overline{X},\overline{L})\to (X,L)$ be the normalization. The {\it conductor ideal} of $X$ is defined as 
\[
\mathfrak{cond}_{X}\coloneq \mathcal{H}om_{\mathcal{O}_X}(\nu_*\mathcal{O}_{\overline{X}},\mathcal{O}_X).
\]
This is a coherent ideal sheaf of both of $\mathcal{O}_{\overline{X}}$ and $\mathcal{O}_X$ and defines the closed subschemes
\[
D_X\subset X,\quad D_{\overline{X}}\subset\overline{X}.
\]
They are reduced of codimension 1 and supported on nodes. 

Let $\overline{B}$ be the divisorial part of $\nu^{-1}B$ and then we have
\[
K_{(\overline{X},\overline{B}+D_{\overline{X}})}\sim_{\mathbb{Q}} \nu^*(K_{(X,B)}).
\]
Here, we remark that if $(\overline{\mathcal{X}},\overline{\mathcal{L}})\to(\mathcal{X},\mathcal{L})$ is the normalization, then we have
\[
H_{\overline{B}+D_{\overline{X}}}^\mathrm{NA}(\overline{\mathcal{X}},\overline{\mathcal{L}})=H_{B}^\mathrm{NA}(\mathcal{X},\mathcal{L}).
\]
\end{de}

 Thanks to \cite{OX} and \cite{God}, if a polarized normal (resp. deminormal) pair $(X,B;L)$ is K-semistable, then $(X,B)$ has only lc (resp. slc) singularities (see \cite{Ko}). Recall that if $(X,B;L)$ is an lc polarized pair, $H_{B}^\mathrm{NA}\ge 0$ on $\mathcal{H}^\mathrm{NA}(L)$ (cf. \cite{BHJ}). Furthermore, if $(X,B;L)$ is klt, $H_{B}^\mathrm{NA}\ge \alpha(X,B;L)I^\mathrm{NA}$. Here, $$\alpha(X,B;L)=\inf_{D,v}\frac{A_{(X,B)}(v)}{v(D)}>0$$ is the alpha invariant of $(X,B;L)$ where the infimum is taken over all the effective $\mathbb{Q}$-Cartier divisor $D$ on $X$ that is $\mathbb{Q}$-linearly equivalent to $L$ and all the divisorial valuation $v$ on $X$ such that $v(D)>0$ (cf. Theorem 9.14 of loc.cit). On the other hand, $\mathrm{J}^H$-stability is irrelevant to singularities of $X$ and its boundary divisor.
 
We can also define $M_B^\mathrm{NA}$ and $H_B^\mathrm{NA}$ when $X$ is deminormal similarly. Note also that we can define $(\mathcal{J}^H)^\mathrm{NA}$-energy for any non-normal but integral polarized variety $(V,M)$. Since the $\mathrm{J}^H$-stability of $(V,M)$ coincides with one of the normalization of $(V,M)$, we have to discuss only about its normalization.
 
 Here, we remark that $I^\mathrm{NA}$ and $J^\mathrm{NA}$ are nonnegative functionals on $\mathcal{H}^\mathrm{NA}(L)$ that vanish only on the metric of the trivial test configuration and satisfy
 \[
 \frac{1}{n}J^\mathrm{NA}\le I^\mathrm{NA}-J^\mathrm{NA}\le nJ^\mathrm{NA}
 \]
 (cf. Proposition 7.8 of \cite{BHJ}). In other words, $ I^\mathrm{NA}-J^\mathrm{NA}, I^\mathrm{NA}$ and $J^\mathrm{NA}$ are equivalent norms on $\mathcal{H}^\mathrm{NA}(L)$. $E^\mathrm{NA},R_B^\mathrm{NA},I^\mathrm{NA},J^\mathrm{NA},(\mathcal{J}^H)^\mathrm{NA},H_B^\mathrm{NA}$ and $M_B^\mathrm{NA}$ are all homogeneous in $\mathcal{L}$. In other words, $E^\mathrm{NA}(\mathcal{X},t\mathcal{L})=tE^\mathrm{NA}(\mathcal{X},\mathcal{L})$ for $t>0$ for example. We can conclude that the stability of $(X,B;L)$ is equivalent to the stability of $(X,B;tL)$. $(\mathcal{J}^H)^\mathrm{NA}$ is also homogeneous in $H$.
 
 We also remark that $$(\mathcal{J}^{\lambda L})^\mathrm{NA}(\mathcal{X},\mathcal{L})=\lambda (I^\mathrm{NA}(\mathcal{X},\mathcal{L})-J^\mathrm{NA}(\mathcal{X},\mathcal{L}))$$ for $\lambda\in\mathbb{Q}$. It follows from an easy computation as in the proof of Lemma 7.25 of loc.cit. On the other hand,
 \[
 M_{B}^\mathrm{NA}-H_{B}^\mathrm{NA}=R_{B}^\mathrm{NA}+S(X,B;L)E^\mathrm{NA}=(\mathcal{J}^{K_{(X,B)}})^\mathrm{NA}.
 \]
 If $(X,B;L)$ is a klt polarized variety and there exists $\delta<\alpha(X,B;L)$ such that 
 \[
 (\mathcal{J}^{K_{(X,B)}})(\phi)^\mathrm{NA}\ge-\delta I^\mathrm{NA}(\phi)
 \]
 for all $\phi\in \mathcal{H}^\mathrm{NA}(L)$, this polarized pair is uniformly K-stable by Proposition 9.16 of loc.cit. Furthermore, it follows from \cite[Theorem 6.10]{Li} that if $X$ is smooth, $B=0$ and there exists $\delta<\alpha(X,L)$ such that 
 \[
  (\mathcal{J}^{K_{X}})(\phi)^\mathrm{NA}\ge-\delta I^\mathrm{NA}(\phi)
 \]
 for all $\phi\in \mathcal{H}^\mathrm{NA}(L)$, then there exists a cscK metric on $(X,L)$.
 
 Next, we give the definition of flag ideals in this section.
 
 \begin{de}[Flag Ideals (cf. Definition 3.1 of \cite{Od})]\label{fldef}
 Let $(X,L)$ be an $n$-dimensional polarized variety. A coherent sheaf of ideals $\mathfrak{a}$ of $X\times \mathbb{A}^1$ is called a {\it flag ideal} if 
 \[
 \mathfrak{a}=I_0+I_1t+\cdots +I_{r-1}t^{r-1}+(t^r),
 \]
 where $t$ is the coordinate function of $\mathbb{A}^1$ and $I_0\subseteq I_1\subseteq \cdots \subseteq I_{r-1}\subseteq \mathcal{O}_X$ are coherent ideals. Note that it is equivalent to that $\mathfrak{a}$ is $\mathbb{G}_m$-invariant under the natural action of $\mathbb{G}_m$ on $X\times \mathbb{A}^1$.
 \end{de}
 
 According to Proposition 3.10 of \cite{Od} or the proof of Theorem 2.9 of \cite{M} (cf. \cite{RT07}), we can see that for any ample test configuration $(\mathcal{X},\mathcal{L})$, there exists a flag ideal $\mathfrak{a}$ such that $(\mathrm{Bl}_{\mathfrak{a}}(X\times \mathbb{A}^1), L_{\mathbb{A}^1}-sE)$ is a semiample test configuration, where $s\in\mathbb{Q}_{\ge0}$ and $E$ is the exceptional divisor on $\mathrm{Bl}_{\mathfrak{a}}(X\times \mathbb{A}^1)$, which is the blow up of $X\times \mathbb{A}^1$ along $\mathfrak{a}$, and a pullback of $(\mathcal{X},\mathcal{L})$. If $(\mathcal{X},\mathcal{L})$ is semiample, there exists a sufficiently divisible integer $k$ such that $k\mathcal{L}$ is globally generated. Then, we can see that the ample model $(\mathcal{P}roj_{\mathbb{A}^1}(\bigoplus_{m=0}^\infty\pi_*\mathcal{O}_{\mathcal{X}}(km\mathcal{L})),\mathcal{O}(1))$ is an ample test configuration equivalent to $(\mathcal{X},k\mathcal{L})$ (cf. Proposition 2.17 of \cite{BHJ}). Therefore, there exists a flag ideal $\mathfrak{a}$ such that $(\mathrm{Bl}_{\mathfrak{a}}(X\times \mathbb{A}^1), L_{\mathbb{A}^1}-sE)$ is a semiample test configuration equivalent to $(\mathcal{X},\mathcal{L})$. 
 
 We can easily see that the deformations to the normal cone along $\mathfrak{a}$ and $\overline{\mathfrak{a}}$ define the same positive metric.
 
  Note that multiplying $\mathfrak{a}$ by $t^m$ coincides with replacing $(\mathcal{X},\mathcal{L})$ by $(\mathcal{X},\mathcal{L}-m\mathcal{X}_0)$. On the other hand, $(\mathcal{J}^H)^\mathrm{NA}$, $J^\mathrm{NA}$, $I^\mathrm{NA}$ and $M^\mathrm{NA}_{B}$ are {\it translation invariant}. In other words, $(\mathcal{J}^H)^\mathrm{NA}(\mathcal{X},\mathcal{L})=(\mathcal{J}^H)^\mathrm{NA}(\mathcal{X},\mathcal{L}+m\mathcal{X}_0)$ for $m\in \mathbb{Q}$. To study $\mathrm{J}^H$-stability and K-stability, we may assume that $I_0\ne 0$ for $\mathfrak{a}$.
  
  Finally, note that if $(X,L)$ is reducible, $(X_1,L_1,H_1)$ is an irreducible component of $(X,L,H)$, $(\mathcal{X}_1,\mathcal{L}_1)$ is the strict transformation of $X_1\times\mathbb{A}^1$ in $(\mathcal{X},\mathcal{L})$, then  
  \[
  (\mathcal{J}^H)^\mathrm{NA}(\mathcal{X},\mathcal{L})\ne(\mathcal{J}^H)^\mathrm{NA}(\mathcal{X},\mathcal{L}+m\mathcal{X}_{1,0})
  \]
  in general for $m\in \mathbb{Q}$, where $\mathcal{X}_{1,0}$ is the central fibre of $\mathcal{X}$ unless $\frac{L_1^{n-1}\cdot H_1}{L_1^n}=\frac{H\cdot L^{n-1}}{L^n}$.
  \begin{de}\label{denden}
  If a polarized reducible scheme $(X,L)$ with a $\mathbb{Q}$-divisor $H$ satisfies that 
  \[
  \frac{L^{n-1}\cdot H}{L^n}=\frac{(L|_{X_i})^{n-1}\cdot H|_{X_i}}{(L|_{X_i})^{n}}
  \]
   for any irreducible component $X_i$, then we say that {\it all irreducible components of} $(X,L)$ {\it have the same average scalar curvature with respect to} $H$.
  \end{de}
 
 \section{Mixed Multiplicity}\label{Mixed}
In this section $X$ is a variety. First, recall the definition of mixed multiplicities, which is used to express our decomposition formula, Theorem \ref{impl}.

\begin{de}\label{mazemaze}
Let $X$ be an $n$-dimensional variety, $L$ be a $\mathbb{Q}$-line bundle on $X$ and $\mathscr{I}_1,\cdots,\mathscr{I}_n$ be coherent ideals of $\mathcal{O}_X$ such that $\mathrm{supp}(\mathcal{O}_X/\mathscr{I}_i)$ is proper or $\emptyset$. Choose a compactification $\overline{X}$ of $X$ on which $L$ extends to a line bundle $\overline{L}$ and let $\pi:\overline{B}\to \overline{X}$ be the blowing up of $\overline{X}$ along $\prod \mathscr{I}_i$ so that $\pi^{-1}(\mathscr{I}_i)=\mathcal{O}_{\overline{B}}(-E_i)$. Let $\pi^*L=\mathcal{O}_{\overline{B}}(D)$. We define
\[
e_L(\mathscr{I}_1,\cdots,\mathscr{I}_n)=(D^n)-((D-E_1)\cdot \cdots \cdot(D-E_n)).
\]
$e_L(\mathscr{I}_1,\cdots,\mathscr{I}_n)$ is independent of the choice of $\overline{X}$ and $\overline{L}$. 
\end{de}

For the proof of the main theorem, we prepare the following modification of Proposition 4.10 of \cite{M}.

\begin{thm}\label{mu}
Given an $n$-dimensional proper variety $X$, a line bundle $L$ on $X$ and an flag ideal $\mathfrak{a}\subset \mathcal{O}_{X\times \mathbb{A}^1}$. Suppose that $\mathfrak{a}$ satisfies that the following condition (*):
\begin{itemize}
\item[(*)]\[
\mathfrak{a}=\mathscr{I}_{D_0}+\mathscr{I}_{D_1}t+\cdots +\mathscr{I}_{D_{r-1}}t^{r-1}+t^r,
\]
where each $\mathscr{I}_{D_i}$ is a coherent ideal sheaf corresponding to a Cartier divisor $D_i$ of $X$. Furthermore, for each $m\in \mathbb{Z}_{\ge0}$, $$\biggl(\overline{(\pi\times\mathrm{id}_{\mathbb{A}^1})^{-1}\mathfrak{a}}\biggr)^m=\sum_{k=0}^{mr} t^k\mathscr{I}_{m,k},$$ where $\mathscr{I}_{m,k}=\mathscr{I}_{D_j}^{m-i}\cdot \mathscr{I}_{D_{j+1}}^{i}$ for $j=\lfloor \frac{k}{m}\rfloor$ and $i=k-mj$.
\end{itemize}
Suppose also that $L$ is semiample and $\mathfrak{a}L_{\mathbb{A}^1}$ and all the $\mathscr{I}_{D_k}L$ are nef. Then,
\[
e_{L_{\mathbb{A}^1}}(\mathfrak{a})=\sum_{k=0}^{r-1}\sum_{j=0}^ne_L(\mathscr{I}_{D_k}^{[j]},\mathscr{I}_{D_{k+1}}^{[n-j]}),
\]
where $\mathscr{I}_{D_r}=\mathcal{O}_X$, $\mathscr{I}_{D_k}^{[j]}$ indicates that $\mathscr{I}_{D_k}$ appears $r_i$ times and $e_{L_{\mathbb{A}^1}}(\mathfrak{a})=e_{L_{\mathbb{A}^1}}(\mathfrak{a}^{[n+1]})$.
\end{thm}

In Theorem \ref{mu}, let $\mathcal{X}$ be the blow up of $X\times \mathbb{A}^1$ along $\mathfrak{a}$ with the exceptional divisor and $\mathcal{L}=L_{\mathbb{A}^1}-E$. Then it is easy to see that 
\[
e_{L_{\mathbb{A}^1}}(\mathfrak{a})=-\mathcal{L}^{n+1}.
\]
 Here, $\mathscr{I}L$ is nef (resp. ample) if $\pi :B\to X$ is the blow up along $\mathscr{I}$, $E$ is the exceptional divisor corresponding $\pi^{-1}\mathscr{I}$ and $\pi^*L-E$ is nef (resp. ample) on $B$. Note that the theorem holds even when $\mathfrak{a}L$ is not necessarily globally generated. 

\begin{ex}\label{33}
Suppose that $(X,L)$ is an $n$-dimensional polarized proper variety, and $\mathscr{I}$ is a coherent ideal on $X$. Let $\mathcal{X}$ be the deformation to the normal cone along $\mathscr{I}$ and $E$ be the exceptional divisor. Suppose that $\mathcal{L}=p^*L-cE$ is a semiample line bundle where $p:\mathcal{X}\to X$ is the canonical projection for $c>0$. Then, the fiber of $0\in \mathbb{P}^1$, $$\mathcal{X}_0=\hat{X}+E,$$ where $\pi:\hat{X}\to X$ is the blow up along $\mathscr{I}$. Let $D$ be the exceptional divisor on $\hat{X}$ and then $D=\hat{X}\cap E$ scheme-theoretically. Therefore,
\begin{align*}
\mathcal{L}^{n+1}&=(p^*L-cE)^{n+1} \\
&=-c\sum_{i=0}^nE\cdot (p^*L^i\cdot (p^*L-cE)^{n-i}) \\
&=c\sum_{i=0}^n(\hat{X}-\mathcal{X}_0)\cdot (p^*L^i\cdot (p^*L-cE)^{n-i})\\
&=c\sum_{i=0}^n \pi^*L^i\cdot (\pi^*L-cD)^{n-i}-c(n+1)L^n.
\end{align*} 
 Therefore, if $t$ is a parameter of $\mathbb{A}^1$,
\[
e_{L_{\mathbb{A}^1}}(\mathscr{I}+(t))=\sum_{j=0}^ne_L(\mathscr{I}^{[j]},\mathcal{O}_X^{[n-j]}).
\]
 On the other hand, if $H$ is a line bundle on $X$, it is easy to see that
\begin{align*}
V(L)(\mathcal{J}^H)^\mathrm{NA}(\mathcal{X},\mathcal{L})&=c\Biggl(\pi^*H\cdot \left(\sum_{i=0}^{n-1} (\pi^*L- cD)^i\cdot \pi^*L^{n-i-1}\right) \\
&- \frac{nH\cdot L^{n-1}}{(n+1)L^n}\sum_{i=0}^n (\pi^*L- cD)^i \cdot \pi^*L^{n-i}\Biggr). 
\end{align*}
We call $(\mathcal{J}^H)^\mathrm{NA}$-energy of a semiample deformation to the normal cone a $(\mathcal{J}^H)^\mathrm{NA}$-slope. We also remark that $e_{L_{\mathbb{A}^1}}(\mathscr{I}+(t))=(p^*L)^{n+1}-\mathcal{L}^{n+1}=-\mathcal{L}^{n+1}$ if $c=1$.

Finally, if $n=2$ and $\mathscr{I}$ is invertible, then
\[
V(L)(\mathcal{J}^H)^\mathrm{NA}(\mathcal{X},\mathcal{L})=c^2\left(2\frac{H\cdot L}{L^2}L-H\right)\cdot D-c^3\frac{2H\cdot L}{3L^2}D^2.
\]
We show that this is nonnegative when $2\frac{H\cdot L}{L^2}L-H$ is nef and $H$ is pseudoeffective in Proposition \ref{deft}.
\end{ex}

In the condition (*), the fact that the right hand side is included in the left hand side is trivial. However, the opposite inclusion does not necessarily hold for an arbitrary flag ideal $\mathfrak{a}$ and for large $m$. Therefore, the inequality in Proposition 4.10 of \cite{M} is not necessarily an equality for an arbitrary flag ideal $\mathfrak{a}$ either. However, if (*) holds, we show the inequality is an equality. Theorem \ref{mu} follows from Propositions 4.3, 4.8, 4.9 and 4.10 of loc.cit. However, we need the following modification of Proposition 4.9 of loc.cit.

\begin{prop}\label{4.9}
Let $V^n$ be a proper variety over $k$ and $L$ be a line bundle on $V$.
\begin{itemize}
\item[(1)]  Let $\mathscr{I}_1,\mathscr{I}_2,\cdots ,\mathscr{I}_r$ be coherent ideals in $\mathcal{O}_V$. If $L,\mathscr{I}_1L,\cdots ,\mathscr{I}_rL$ are nef and all the $\mathscr{I}_j$ is invertible, then
\begin{align*}
|\chi (V,L^{\sum_{j=1}^r m_j}/(\prod_{j=1}^r \mathscr{I}_j^{m_j})L^{\otimes\sum_{j=1}^r m_j})&-\mathrm{dim}(\Gamma(V,L^{\sum_{j=1}^r m_j})/\Gamma(V,(\prod_{j=1}^r \mathscr{I}_j^{m_j})L^{\otimes\sum_{j=1}^r m_j}))|\\
&=O((\sum_{j=1}^r m_j)^{n-1}).
\end{align*}
\item[(2)] Let $\mathfrak{a}$ be a flag ideal in $\mathcal{O}_{V\times \mathbb{A}^1}$ and hence $\mathrm{supp}(\mathcal{O}_{V\times \mathbb{A}^1}/\mathfrak{a})$ is proper. If $L$ is semiample and $\mathfrak{a}L_{\mathbb{A}^1}$ is nef, then
\[
\left|\chi \left(V\times \mathbb{A}^1,L^{\otimes m}_{\mathbb{A}^1}/( \mathfrak{a}^{m})L^{\otimes m}_{\mathbb{A}^1}\right)-\mathrm{dim}\left(\Gamma(V\times \mathbb{A}^1,L^{\otimes m}_{\mathbb{A}^1})/\Gamma(V\times \mathbb{A}^1,\mathfrak{a}^{m}L^{\otimes m}_{\mathbb{A}^1})\right)\right|=O(m^{n}).
\]
\end{itemize} 
\end{prop}

\begin{proof}
(1): It follows from the fact we will prove that 
\[
h^i(V,L^{\otimes\sum m_j})=O((\sum m_j)^{n-1})
\]
 and 
 \[
 h^i(V,(\prod_j \mathscr{I}_j^{m_j})L^{\otimes\sum m_j})=O((\sum m_j)^{n-1})
 \]
  for $i>0$ as \cite[Proposition 4.9]{M}. The former is easier and we will only show the latter in this proof. More generally, we prove $h^i(V,(\prod_j \mathscr{I}_j^{m_j})(L^{\otimes\sum m_j}\otimes \mathscr{F}))=O((\sum m_j)^{n-i})$ for any coherent sheaf $\mathscr{F}$ by the induction on $n$. Let $\mathcal{O}_V(-E_i)=\mathscr{I}_i$. We can prove that there exists $C'>0$ such that
\[
h^i(V,\mathscr{F}\otimes L^{\otimes\sum m_j}(-\sum m_jE_j))<C'(\sum m_j)^{n-i}
\]
by the fact that $L-E_j$s are nef, using Fujita vanishing theorem (cf. Theorem 1.4.35 of \cite{Laz}) and the induction on $n$. In fact, if $i=0$ we can take an ample divisor $H$ that does not pass through any associated point of $\mathscr{F}$ such that 
$$
H^i(\mathscr{F}\otimes L^{\otimes \sum m_j}(H-\sum m_jE_j))=0
$$ for $i>0$ and any $m_j\ge 0$ by Fujita's theorem and
\begin{eqnarray*}
H^0( \mathscr{F}\otimes L^{\otimes \sum m_j}(-\sum m_jE_j)) \hookrightarrow H^0( \mathscr{F}\otimes L^{\otimes \sum m_j}(H-\sum m_jE_j))\\
h^0( \mathscr{F}\otimes L^{\otimes \sum m_j}(H-\sum m_jE_j))=\chi (\mathscr{F}\otimes L^{\otimes \sum m_j}(H-\sum m_jE_j))=O((\sum m_j)^n)
\end{eqnarray*}
by Snapper's theorem (\cite{B} Theorem 1.1). Otherwise, take an ample and integral divisor $H$ does not pass through any associated point of $\mathscr{F}$ and consider the following exact sequence:
\[
0\to \mathscr{F}\otimes L^{\otimes \sum m_j}(-\sum m_jE_j) \to \mathscr{F}\otimes L^{\otimes \sum m_j}(H-\sum m_jE_j) \to \mathscr{F}(H)_{|H}\otimes L^{\otimes \sum m_j}(-\sum m_jE_j) \to 0.
\]
Thanks to Fujita's theorem, we can choose $H$ so that $H^i(\mathscr{F}\otimes L^{\otimes \sum m_j}(H-\sum m_jE_j))=0$ for $i>0$ and any $m_j\ge 0$. By the long exact sequence and the induction hypothesis of $n$, we have
 \[
 h^i( \mathscr{F}\otimes L^{\otimes \sum m_j}(-\sum m_jE_j) )\le h^{i-1}(\mathscr{F}(H)_{|H}\otimes L^{\otimes \sum m_j}(-\sum m_jE_j))=O((\sum m_j)^{(n-1)-(i-1)}).
 \]
 
(2): There exists an exact sequence
\[
0\to H^0(V\times \mathbb{A}^1,\mathfrak{a}^m\cdot L_{\mathbb{A}^1}^{\otimes m})\to H^0(V\times \mathbb{A}^1, L_{\mathbb{A}^1}^{\otimes m}) \to H^0(V\times \mathbb{A}^1, L_{\mathbb{A}^1}^{\otimes m}/\mathfrak{a}^m\cdot L_{\mathbb{A}^1}^{\otimes m}) .
\]
We want to show that 
\[
\mathrm{dim}(\mathrm{Coker}(H^0(V\times \mathbb{A}^1, L_{\mathbb{A}^1}^{\otimes m}) \to H^0(V\times \mathbb{A}^1, L_{\mathbb{A}^1}^{\otimes m}/\mathfrak{a}^m\cdot L_{\mathbb{A}^1}^{\otimes m})))=O(m^{n})
\]
and 
\[
h^i(V\times \mathbb{A}^1, L_{\mathbb{A}^1}^{\otimes m}/\mathfrak{a}^m\cdot L_{\mathbb{A}^1}^{\otimes m})=O(m^n)
\]
for $i>0$. There exists $N'>0$ such that $L_{\mathbb{A}^1}^{\otimes m}$ is globally generated for $m\ge N'$. Then for $k>0$ and $m\ge N'$, since $H^0(X\times\mathbb{A}^1,L_{\mathbb{A}^1}^{\otimes m})$ is generated by 
\begin{align*}
H^0(V,L^m)\subset &H^0(V\times \mathbb{P}^1,L_{\mathbb{P}^1}^{\otimes m}(km(V\times\{0\})))\\
f\mapsto & f\cdot t^{-km},
\end{align*}
 we have only to prove that $$
\mathrm{dim}(\mathrm{Coker}(H^0(V\times \mathbb{P}^1, L_{\mathbb{P}^1}^{\otimes m}(mk(V\times\{0\}))) \to H^0(V\times \mathbb{P}^1, L_{\mathbb{P}^1}^{\otimes m}/\mathfrak{a}^m\cdot L_{\mathbb{P}^1}^{\otimes m})))=O(m^{n})
$$
where $L_{\mathbb{P}^1}=L\otimes \mathcal{O}_{\mathbb{P}^1}$. Note that $L_{\mathbb{P}^1}^{\otimes m}/\mathfrak{a}^m\cdot L_{\mathbb{P}^1}^{\otimes m}=(L_{\mathbb{P}^1}^{\otimes m}/\mathfrak{a}^m\cdot L_{\mathbb{P}^1}^{\otimes m})\otimes \mathcal{O}_{V\times \mathbb{P}^1}(mk(V\times \{0\}))$ since $\mathrm{supp}(\mathcal{O}_{V\times \mathbb{A}^1}/\mathfrak{a})$ is proper. Let $\pi:B\to V\times \mathbb{P}^1$ be the blow up of $V\times \mathbb{P}^1$ along $\mathfrak{a}$ and $E$ be the exceptional divisor. As in the proof of \cite[Proposition 4.8]{M}, there is a large integer $N$ such that
\begin{itemize}
\item[a)] $R^i\pi_*(\mathcal{O}_B(- mE))=0,\, i>0$ 
\item[b)] $\pi_*(\pi^*\mathcal{O}_B(- mE))=\mathfrak{a}^{m}$
\end{itemize}
when $m\ge N$. Then $h^i(V_{\mathbb{P}^1}, \mathfrak{a}^m\cdot L_{\mathbb{P}^1}^{\otimes m})=h^i(B,\pi^*L_{\mathbb{P}^1}^{\otimes m}(- mE))$ by Leray spectral sequence. By the assumption, since $\pi^*L_{ \mathbb{P}^1}-E$ is $\mathbb{P}^1$-nef, there exists sufficiently large $k$ that $\pi^*L_{ \mathbb{P}^1}(k(V\times\{0\})-E)$ is nef. In fact, we can prove that it is nef when we choose $k$ so large that $k\pi^*(V\times\{0\})-E$ is effective. Let $C$ be an integral curve on $B$ and if $C\cdot (\pi^*L-E)< 0$, $C$ is not contained in any fibre over $\mathbb{P}^1$. Since $\pi^*L$ is nef and $C\not \subset k\pi^*(V\times\{0\})-E$, $$C\cdot (\pi^*L_{ \mathbb{P}^1}(k(V\times\{0\}))-E)\ge 0.$$ Hence, we can apply Fujita's theorem again to obtain
\[
h^i(B,\pi^*L_{\mathbb{P}^1}^{\otimes m}(mk(V\times\{0\})-mE))=O(m^{n+1-i})
\] 
and 
\[
h^i(V\times\mathbb{P}^1,L_{\mathbb{P}^1}^{\otimes m}(mk(V\times \{0\}))=O(m^{n+1-i}).
\]
We can easily see that 
\begin{align*}
\mathrm{dim}(\mathrm{Coker}(H^0(V\times \mathbb{P}^1, L_{\mathbb{P}^1}^{\otimes m}(mk(V\times\{0\}))) &\to H^0(V\times \mathbb{P}^1, L_{\mathbb{P}^1}^{\otimes m}/\mathfrak{a}^m\cdot L_{\mathbb{P}^1}^{\otimes m})))\\
&\le h^1(B,\pi^*L_{\mathbb{P}^1}^{\otimes m}(mk(V\times\{0\})-mE))
\end{align*}
and 
\[
h^i(V\times \mathbb{A}^1, L_{\mathbb{A}^1}^{\otimes m}/\mathfrak{a}^m\cdot L_{\mathbb{A}^1}^{\otimes m})=O(m^{n})
\]
for $i>0$. We complete the proof.
\end{proof}

\begin{proof}[Proof of Theorem \ref{mu}]
By the assumption, $\mathfrak{a}$ satisfies the condition (*). For each $m$, we denote $\mathfrak{a}^m=\bigoplus_{k=0}^\infty t^k\mathscr{I}_{m,k}$, where $\mathscr{I}_{m,mk+i}=\mathscr{I}_{D_k}^{m-i}\cdot \mathscr{I}_{D_{k+1}}^{i}$. Since
\[
H^0(X\times \mathbb{A}^1,\mathfrak{a}^mL_{\mathbb{A}^1}^{\otimes m})=\bigoplus _{k=0}^\infty H^0(X,\mathscr{I}_{m,k}L^{\otimes m})\cdot t^k,
\]
it follows that
\[
\mathrm{dim}(H^0(X\times \mathbb{A}^1,L_{\mathbb{A}^1}^{\otimes m})/H^0(X\times \mathbb{A}^1,\mathfrak{a}^mL_{\mathbb{A}^1}^{\otimes m}))=\sum^\infty_{k=0}\mathrm{dim}(H^0(X,L^{\otimes m})/H^0(X,\mathscr{I}_{m,k}L^{\otimes m})).
\]
The rest of the proof follows from Proposition 4.10 of \cite{M} immediately. For the reader's convenience, we prove as follows. Apply Proposition \ref{4.9} and we get the estimates 
\[
\mathrm{dim}[H^0(X\times \mathbb{A}^1,L_{\mathbb{A}^1}^{\otimes m})/H^0(X\times \mathbb{A}^1,\mathfrak{a}^mL^{\otimes m}_{\mathbb{A}^1})].
\]
for $\chi(L_{\mathbb{A}^1}^{\otimes m}/\mathfrak{a}^mL_{\mathbb{A}^1}^{\otimes m})$ and
\[
\mathrm{dim}(H^0(X,L^{\otimes m})/H^0(X,\mathscr{I}_{D_k}^{m-i}\cdot \mathscr{I}_{D_{k+1}}^{i}\cdot L^{\otimes m}))
\]
for $\chi(L^{\otimes m}/\mathscr{I}_{D_k}^{m-i}\cdot \mathscr{I}_{D_{k+1}}^{i}\cdot L^{\otimes m})$. By the weak form of Riemann-Roch Theorem (\cite[Theorem 1.36]{KM}),
\begin{align*}
\chi(L_{\mathbb{A}^1}^{\otimes m}/\mathfrak{a}^{m} L_{\mathbb{A}^1}^{\otimes m})&=\chi(L_{\mathbb{P}^1}^{\otimes m}/\mathfrak{a}^{m} L_{\mathbb{P}^1}^{\otimes m}) \\
&=\frac{1}{(n+1)!}e_{L_{\mathbb{A}^1}}(\mathfrak{a})m^{n+1} +O(m^{n}).
\end{align*}
and
\begin{align*}
\chi(L^{\otimes m}/\mathscr{I}_{D_k}^{m-i}\cdot \mathscr{I}_{D_{k+1}}^{i}\cdot L^{\otimes m})&= \chi(L^{\otimes m})-\chi(L^{\otimes m}(-(m-i)D_k-iD_{k+1})) \\
&=\sum_{j=0}^n\frac{1}{j!(n-j)!}e_L(\mathscr{I}_{D_k}^{[n-j]},\mathscr{I}_{D_{k+1}}^{[j]})(m-i)^{n-j}i^j +O(m^{n-1}).
\end{align*}
 Therefore,
\begin{align*}
&\sum^\infty_{k=0}\mathrm{dim}(H^0(X,L^{\otimes m})/H^0(X,\mathscr{I}_{m,k}L^{\otimes m}))\\
&=\sum^{r-1}_{k=0}\sum^{m-1}_{i=0} \mathrm{dim}(H^0(X,L^{\otimes m})/H^0(X,\mathscr{I}_{D_k}^{m-i}\cdot \mathscr{I}_{D_{k+1}}^{i}\cdot L^{\otimes m}))\\
&=\sum^{r-1}_{k=0}\sum^{m-1}_{i=0}\left[\sum_{j=0}^n\frac{1}{j!(n-j)!}e_L(\mathscr{I}_{D_k}^{[n-j]},\mathscr{I}_{D_{k+1}}^{[j]})(m-i)^{n-j}i^j +R_i\right],
\end{align*}
where $R_i=O(m^{n-1})$. As in the proof of \cite[Proposition 4.8]{M}, $\sum_{i=0}^{m-1} R_i=O(m^n)$. Then \cite[Lemma 4.5]{M} gives:
\[
\frac{1}{(n+1)!}m^{n+1}=\frac{1}{j!(n-j)!}\sum^{m-1}_{i=0}(m-i)^{n-j}i^j +O(m^n)
\]
and we see the theorem holds.
\end{proof}

We have the following the extention of Theorem \ref{mu} to the case when $L$ is a $\mathbb{Q}$-line bundle.

\begin{prop}\label{muc}
Given an $n$-dimensional variety $X$, an ample $\mathbb{Q}$-line bundle $L$ on $X$ and an flag ideal $\mathfrak{a}\subset \mathcal{O}_{X\times \mathbb{A}^1}$. If $\mathfrak{a}$ satisfies (*), then we have
\[
e_{sL_{\mathbb{A}^1}}(\mathfrak{a})=\sum_{k=0}^{r-1}\sum_{j=0}^ne_{sL}(\mathscr{I}_{D_k}^{[j]},\mathscr{I}_{D_{k+1}}^{[n-j]})
\] 
for any $s\in \mathbb{Q}$.
\end{prop}

\begin{proof}
In fact, if $\mathrm{Bl}_{\mathfrak{a}}(X\times \mathbb{P}^1)$ is the blow up of $X\times \mathbb{P}^1$ along $\mathfrak{a}$ and $E$ is the exceptional divisor,
\[
e_{sL_{\mathbb{A}^1}}(\mathfrak{a})=(sL_{\mathbb{P}^1})^{n+1}-(sL_{\mathbb{P}^1}-E)^{n+1}
\] 
is a polynomial in $s$ whose degree is at most $n$. On the other hand, 
\[
\sum_{k=0}^{r-1}\sum_{j=0}^ne_{sL}(\mathscr{I}_{D_k}^{[j]},\mathscr{I}_{D_{k+1}}^{[n-j]})=\sum_{k=0}^{r-1}\sum_{j=0}^n((sL)^n-(sL-D_k)^j\cdot(sL-D_{k+1})^{n-j})
\]
is a polynomial in $s$ whose degree is at most $n-1$. Since $L$ is an ample $\mathbb{Q}$-line bundle, there exists a sufficiently divisible integer $m$ such that $mL$ is a line bundle and $\mathfrak{a}(mL_{\mathbb{A}^1})$ is ample. Then we have all the $\mathscr{I}_{D_k}(mL)$ is nef by Corollary 5.8 of \cite{RT07}. Then, we can see that $mL,2mL,\cdots ,(n+1)mL$ satisfy the assumption of Theorem \ref{mu}. Therefore, we can apply Theorem \ref{mu} and we have
\[
e_{sL_{\mathbb{A}^1}}(\mathfrak{a})=\sum_{k=0}^{r-1}\sum_{j=0}^ne_{sL}(\mathscr{I}_{D_k}^{[j]},\mathscr{I}_{D_{k+1}}^{[n-j]})
\]
for $s=m,2m,\cdots,(n+1)m$. There exists $n+1$ zeroes of the polynomial in $s$ $$e_{sL_{\mathbb{A}^1}}(\mathfrak{a})-\sum_{k=0}^{r-1}\sum_{j=0}^ne_{sL}(\mathscr{I}_{D_k}^{[j]},\mathscr{I}_{D_{k+1}}^{[n-j]})$$ and hence we have
\[
e_{sL_{\mathbb{A}^1}}(\mathfrak{a})=\sum_{k=0}^{r-1}\sum_{j=0}^ne_{sL}(\mathscr{I}_{D_k}^{[j]},\mathscr{I}_{D_{k+1}}^{[n-j]})
\]
for any $s$.
\end{proof}

\begin{prop}\label{slpro}
Given an $n$-dimensional polarized normal variety $(X,L)$. Suppose that $n\ge 2$, $H$ is an ample line bundle on $X$ and a flag ideal $\mathfrak{a}\subset\mathcal{O}_{X\times\mathbb{A}^1}$ satisfies (*) in Theorem \ref{mu}. Then for general normal and connected divisor $D\in |lH|$ for $l\gg 0$, the restriction of $\mathfrak{a}$ to $D$ satisfies (*). Furthermore, let $\Pi:\mathrm{Bl}_{\mathfrak{a}}(X\times \mathbb{A}^1)\to X\times \mathbb{A}^1$. Then, $\Pi^*(D\times \mathbb{A}^1)$ is a prime divisor for general $D$ and is isomorphic to the blow up of $D\times \mathbb{A}^1$ along $\mathfrak{a}|_{D\times \mathbb{A}^1}$.
\end{prop}

\begin{proof}
Let 
\[
\mathfrak{a}=\sum_{k=0}^r \mathscr{I}_kt^k,
\]
where each $\mathscr{I}_k$ is invertible as in Theorem \ref{mu}. For $l\gg 0$, note that $|lH|$ is base point free on $X$ and $H^0(X\times \mathbb{P}^1, lH_{\mathbb{P}^1})=H^0(X,lH)$. Thus, we may assume that $D$ and $D\times \mathbb{A}^1$ are normal and connected, and do not pass through the associated points of all the $\mathscr{I}_k$ and $\mathfrak{a}$ due to Bertini's theorem for base point free linear systems. We remark that the connectedness follws from the assumption of $\mathrm{dim}\,X$. Therefore, $\iota^{-1}\mathscr{I}_k$ is an invertible ideal sheaf if $\iota:D\hookrightarrow X$ is the canonical inclusion of $D$. By the assumption, $\mathfrak{a}$ satisfies (*). In other words, for $m\in \mathbb{Z}_{\ge0}$, $$\mathfrak{a}^m=\sum_{k=0}^\infty t^k\mathscr{I}_{m,k}$$ on $X\times \mathbb{A}^1$, where $\mathscr{I}_{m,mj+i}=\mathscr{I}_{j}^{m-i}\cdot \mathscr{I}_{j+1}^{i}$. Then for each $m\in \mathbb{Z}_{\ge0}$, $$[\mathfrak{a}|_{D\times\mathbb{A}^1}]^m=\sum_{k=0}^\infty t^k(\iota^{-1}\mathscr{I}_{m,k})$$ on $D\times \mathbb{A}^1$. Note that $\iota^{-1}(\mathscr{I}_{m,mj+i})=\iota^{-1}(\mathscr{I}_{j})^{m-i}\cdot \iota^{-1}(\mathscr{I}_{j+1})^{i}$. Hence, $\mathfrak{a}|_{D\times\mathbb{A}^1}$ also satisfies (*). 
 
 To prove the last assertion, we take $D$ so general that $\Pi^*(D\times \mathbb{A}^1)$ is also integral by Bertini's theorem. $\Pi^*(D\times \mathbb{A}^1)$ is birational to $D\times \mathbb{A}^1$ since it is not contained in the image of the exceptional locus of $\Pi$. On the other hand, let $\mathcal{D}$ be the blow up of $D\times\mathbb{A}^1$ along $\mathfrak{a}|_{D\times\mathbb{A}^1}$ and a closed immersion $\mathcal{D}\to\mathrm{Bl}_{\mathfrak{a}}(X\times \mathbb{A}^1)$ \cite[II Corollary 7.15]{Ha} factors through
 \[
 \varphi :\mathcal{D}\to \Pi^*(D\times \mathbb{A}^1).
 \] 
 Since $\Pi^*(D\times \mathbb{A}^1)$ is integral, $\varphi$ is a dominant closed immersion. Therefore, $\varphi$ is an isomorphism.
\end{proof}

\begin{de}
If $H$ and $D$ are as in Proposition \ref{slpro}, 
\[
e_{L_{|H}}(\mathscr{I}_{k|H}^{[j]},\mathscr{I}_{k+1|H}^{[n-1-j]})=l^{-1}e_{L_{|D}}(\mathscr{I}_{k|D}^{[j]},\mathscr{I}_{k+1|D}^{[n-1-j]}).
\]
If $H$ is a $\mathbb{Q}$-Cartier divisor, we define as follows:
\[
e_{L_{|H}}(\mathscr{I}_{k|H}^{[j]},\mathscr{I}_{k+1|H}^{[n-1-j]})=H\cdot (L^{n-1}-(L-D_k)^j\cdot (L-D_{k+1})^{n-1-j})
\]
where $D_k$ corresponds to $\mathscr{I}_k$. We can check that the definition coincides with the one we gave before if $H$ is as in Proposition \ref{slpro}.
\end{de}

Thanks to Proposition \ref{muc} and Proposition \ref{slpro}, we can calculate $(\mathcal{J}^H)^\mathrm{NA}$ of $(\mathrm{Bl}_{\mathfrak{a}}(X\times \mathbb{A}^1),L_{\mathbb{A}^1}-E)$ as follows, where $E$ is the exceptional divisor and $L_{\mathbb{A}^1}-E$ is ample. 

\begin{thm}\label{sl}
Suppose that $(X,L)$ is an $n$-dimensional variety with a $\mathbb{Q}$-line bundle $H$ and $n\ge2$. If $$\mathfrak{a}=\sum_{k=0}^r \mathscr{I}_kt^k$$ is a flag ideal that satisfies the condition (*) and $(\mathrm{Bl}_{\mathfrak{a}}(X\times \mathbb{A}^1),L_{\mathbb{A}^1}-E)$ is an ample test configuration over $(X,L)$, then
\begin{align*}
V(L)(\mathcal{J}^{H})^\mathrm{NA}(\mathrm{Bl}_{\mathfrak{a}}(X\times \mathbb{A}^1),L_{\mathbb{A}^1}-E)=&\frac{nH\cdot L^{n-1}}{(n+1)L^n}\sum _{k=0}^{r-1}\sum_{j=0}^{n}e_{L}(\mathscr{I}_{k}^{[j]},\mathscr{I}_{k+1}^{[n-j]}) \\
-&\sum _{k=0}^{r-1}\sum_{j=0}^{n-1}e_{L_{|H}}(\mathscr{I}_{k|H}^{[j]},\mathscr{I}_{k+1|H}^{[n-1-j]}).
\end{align*}
\end{thm}

\begin{proof}
By the assumption, we can apply Proposition \ref{muc} and we have
 \begin{align*}
 -\mathcal{L}^{n+1}&=e_{L_{\mathbb{A}^1}}(\mathfrak{a})\\
 &=\sum _{k=0}^r\sum_{j=0}^{n}e_{L}(\mathscr{I}_{D_k}^{[j]},\mathscr{I}_{D_{k+1}}^{[n-j]}).
 \end{align*}
 On the other hand, the left hand side and the right hand side of the equation we want to prove is linear in $H$. Therefore, we may assume that $H$ is ample. Since $L_{|D}$ is ample and $\mathfrak{a}_{|D}$ satisfies (*) for sufficiently general $D\in |lH|$ for $l\gg 0$ due to Proposition \ref{slpro}, we can also apply Proposition \ref{muc} to the computation of the mixed multiplicity of $D$:
  \begin{align*}
 -p^*D\cdot \mathcal{L}^{n}&=(\mathcal{L}|_{\mathcal{D}})^{n}\\
 &=e_{L_{\mathbb{A}^1}|_{D\times\mathbb{A}^1}}(\mathfrak{a}|_{D\times\mathbb{A}^1})\\
 &=\sum _{k=0}^r\sum_{j=0}^{n-1}e_{L|_D}(\mathscr{I}_{D_k|_D}^{[j]},\mathscr{I}_{D_{k+1}|_D}^{[n-1-j]}),
 \end{align*}
where $p:\mathcal{X}\to X$ is the canonical projection and $\mathcal{D}$ is the blow up of $D\times\mathbb{A}^1$ along $\mathfrak{a}|_{D\times\mathbb{A}^1}$ also by Proposition \ref{slpro}. The proof is complete.
\end{proof}

\begin{rem}
If $n=1$, we can define $e_{L_{|H}}(\mathscr{I}_{k|H}^{[j]},\mathscr{I}_{k+1|H}^{[n-1-j]})=0$ and it is easy to see that Theorem \ref{sl} also holds when $n=1$.
\end{rem}

Then we can see any $(\mathcal{J}^H)^\mathrm{NA}$-energy can be decomposed to values like $(\mathcal{J}^H)^\mathrm{NA}$-slopes when the flag ideal $\mathfrak{a}$ satisfies (*).

G. Chen \cite{G} proved the follwing uniform version of Lejmi-Sz{\' e}kelyhidi conjecture.

\begin{thm}[Theorem 1.1 of \cite{G}]\label{modLSconj}
Notations as in \cite{G}. Given a K{\" a}hler manifold $M^n$ with K{\" a}hler metrics $\chi $ and $\omega_0$. Let $c_0$ be the positive constant such that
\begin{eqnarray*}
\int_M \chi \wedge \frac{\omega_0^{n-1}}{(n-1)!}=c_0\int_M\frac{\omega_0^n}{n!}.
\end{eqnarray*}
Then the following are equivalent:
\begin{itemize}
\item[(1)] There exists a smooth function $\varphi$ such that
\[
\omega_{\varphi}=\omega_0+\sqrt{-1}\partial\bar{\partial}\varphi>0
\]
 satisfies the J-equation
\[
\mathrm{tr}_{\omega_{\varphi}}\chi= c_0.
\]
Moreover, such $\varphi$ is unique up to a constant;
\item[(2)] There exists a smooth function $\varphi$ such that $\varphi$ is the critical point of the $\mathcal{J}_\chi$ functional. Moreover, such $\varphi$ is unique up to a constant;
\item[(3)] The $\mathcal{J}_\chi$ functional is coercive. In other words, there exist a positive constant $\epsilon$ and
another constant $C$ such that $\mathcal{J}_\chi(\varphi) \ge \epsilon\mathcal{J}_{\omega_0}(\varphi)-C$;
\item[(4)] $(M,[\omega_0],[\chi])$ is uniformly J-stable. In other words, there exists a positive constant $\epsilon$
such that for any K{\" a}hler test configuration $(X,\Omega)$ (cf. \cite[Definition 2.10]{DR}, \cite[Definitions 3.2 and 3.4]{SjP}), the invariant $\mathcal{J}_{[\chi]}(X,\Omega)$ (cf. \cite[Definition 6.3]{DR}) satisfies
$\mathcal{J}_{[\chi]}(X,\Omega)\ge \epsilon \mathcal{J}_{[\omega_0]}(X,\Omega)$;
\item[(5)] $(M,[\omega_0],[\chi])$ is uniformly slope J-stable. In other words, there exists a positive constant $\epsilon$
such that for any subvariety $V$ of $M$, the deformation to normal cone of $V$
$(X,\Omega)$ (cf. \cite[Example 2.11 (ii)]{DR} satisfies $\mathcal{J}_{[\chi]}(X,\Omega)\ge \epsilon \mathcal{J}_{[\omega_0]}(X,\Omega)$;
\item[(6)]There exists a positive constant $\epsilon$ such that
\[
\int_V (c_0-(n-p)\epsilon)\omega^p_0-p\chi\wedge\omega_0^{p-1}\ge 0
\]
for any $p$-dimensional subvariety $V$ with $p=1,2,\cdots,n$.
\end{itemize}
\end{thm}

\begin{rem}\label{useful}
Note that uniform ``J-stability'' in (4) in Theorem \ref{modLSconj} is ``analytic'' in the sense of \cite{DR}, \cite{SjP}. However, we remark that for any polarized smooth variety $(X,L)$ with an ample divisor $H$, uniform $\mathrm{J}^H$-stability of $(X,L)$ is equivalent to the following condition:
\begin{itemize}
\item[(6)']There exists $\epsilon>0$ such that
\[
\left((n\frac{H\cdot L^{n-1}}{L^n}-(n-p)\epsilon)L-p\,H\right)\cdot L^{p-1}\ge 0
\]
for any $p$-dimensional subvariety $V$ with $p=1,2,\cdots,n$.
\end{itemize}
 In fact, if $(X,L)$ is uniformly $\mathrm{J}^H$-stable, it is easy to see that $(X,L)$ is uniformly slope $\mathrm{J}^H$-stable. By \cite[Proposition 13]{LS} (see Lemma \ref{unst} below), we have
\[
\int_V (c_0-(n-p)\epsilon)\mathrm{c}_1(L)^p-p\,\mathrm{c}_1(H)\cdot \mathrm{c}_1(L)^{p-1}\ge 0
\]
for all $p$-dimensional (algebraic) subvarieties $V$ with $p=1,2,\cdots,n$, where $c_0=n\frac{H\cdot L^{n-1}}{L^n}$. Therefore, (6)' holds.

To show the converse, suppose that (6)' holds. Take K\"{a}hler forms $\omega_0\in\mathrm{c}_1(L)$ and $\chi\in\mathrm{c}_1(H)$. Then, $(X,[\omega_0],[\chi])$ is uniformly J-stable in the sense of \cite{DR} by Theorem \ref{modLSconj}. In particular, $(X,L)$ is uniformly $\mathrm{J}^H$-stable.

We also remark that if the solution to the J-equation exists, it is unique up to a constant by Proposition 2 of \cite{xC} (see \cite[Remark 1.4]{G}).
\end{rem}

Due to Theorem \ref{modLSconj}, uniform slope $\mathrm{J}^H$-stability implies uniform $\mathrm{J}^H$-stability in the case when $X$ is smooth and $H$ is ample. On the other hand, Ross and Thomas proved Theorems 6.1 and 6.4 of \cite{RT07}. Thanks to the theorems, we can decompose $(\mathcal{J}^H)^\mathrm{NA}(\mathcal{X},\mathcal{L})$ into finitely many $(\mathcal{J}^H)^\mathrm{NA}$-energy of semiample deformations to the normal cone (we say that $(\mathcal{J}^H)^\mathrm{NA}$-slopes) under certain conditions. Unfortunately, the assumptions of Theorem 6.1 and 6.4 do not hold in general and we do not know that we can always decompose $(\mathcal{J}^H)^\mathrm{NA}$-energy into a finite number of $(\mathcal{J}^H)^\mathrm{NA}$-slopes. However, we can decompose $(\mathcal{J}^H)^\mathrm{NA}(\mathcal{X},\mathcal{L})$ into the mixed multiplicities instead of $(\mathcal{J}^H)^\mathrm{NA}$-slopes in general by taking an alternation as we show in \S \ref{ThMain}. 

\section{Newton Polyhedron and Toroidal Embeddings}\label{NewTro}

In this section, we prepare the notion of Newton polyhedron to prove our decomposition formula. Here, recall the definitions and some facts of \cite[Chapter II]{KKMS}.

\begin{de}[Toroidal embeddings]\label{toroidal}
Suppose that $Z$ is a $p$-dimensional normal variety, $U$ is a smooth open subvariety of $Z$ and $U\hookrightarrow Z$ is a {\it toroidal embedding without intersection} in the sense of \cite{KKMS}. That is, for any closed point $z$ of $Z$, there exist an affine toric variety $X_{\sigma}$, its closed point $t$ and an isomorphism $\widehat{\mathcal{O}_{Z,z}}\simeq \widehat{\mathcal{O}_{X_\sigma,t}}$ mapping the ideal corresponding to $Z-U$ onto the ideal corresponding to $X_{\sigma}-T$, where $T$ is the $p$-dimensional torus acting on $X_{\sigma}$, and if $E_i$ is any irreducible component of $Z-U$, then it is a normal Weil divisor on $Z$. By replacing $X_\sigma$ by its open toric subvariety, we may assume that $t$ is closed in $X_{\sigma}$ and then we call this a {\it local model} of $(Z,z)$. Then we define as following:
\begin{itemize}
\item $Y$ is a {\it stratum} of $Z$ if it is a locally closed subset that is an irreducible component of $\bigcap _{i\in I}E_i-\bigcup_{j\not\in I}E_j$. It is well-known that $Z$ is the disjoint union of all of its strata. Now fix a stratum $Y$. 
\item A {\it star} of $Y$ if it is an open subset of $Z$ that is the union of all the stratum $Y'$ whose closure contains $Y$. Let us denote it by $\mathrm{Star}\, Y$.
\item $M^{\mathrm{Star}\, Y}$, which we call a lattice with respect to $\mathrm{Star}\, Y$, is a group of Cartier divisors supported on $\mathrm{Star}\, Y-U$. It is a subgroup of the free abelian group of Weil divisors supported on $\mathrm{Star}\, Y-U$. If there is no confusion, we will denote it by $M^Y$ or $M$. We denote $M^{\mathrm{Star}\, Y}_{\mathbb{R}}=M^{\mathrm{Star}\, Y}\otimes_{\mathbb{Z}}\mathbb{R}$. $M^Y_+$ is a subsemigroup of non-negative Cartier divisors and we define $M^Y_{\mathbb{R},+}$ similarly. We call $M^Y_+$ the {\it positive cone} of $M^Y$. We remark that $M_+^{Y}+ (-M_+^{Y})=M^{Y}$, where $M_+^{Y}+ (-M_+^{Y})$ denotes the Minkowski sum.
\item $N^Y=\mathrm{Hom}(M^Y,\mathbb{Z})$ is a dual lattice and $\sigma^Y=\{ n\in N^Y_{\mathbb{R}}=N^Y\otimes\mathbb{R};\forall m\in M_+,\langle n,m\rangle \ge 0 \}$ is a polyhedral cone corresponding to $\mathrm{Star}\, Y$. If $D$ is a Cartier divisor supported on $\mathrm{Star}\, Y-U$, we denote $\mathrm{ord}_D\in M^Y$. We define that $m\le m'$ if $m'-m\in M_{\mathbb{R},+}$ for $m,m'\in M_{\mathbb{R}}$. It is easy to see that $m\le m'$ iff $m'-m$ is a nonnegative function on $\sigma^Y$. Since $\sigma^Y$ is not contained in any hyperplane, $M^Y_{\mathbb{R},+}\cap (-M^Y_{\mathbb{R},+})=0$ (cf. Corollary 1 p.61 loc.cit).
\item $\mathfrak{b}$ is a {\it toroidal fractional ideal} on $Z$ if it is a coherent sheaf of fractional ideals invariant under an isomorphism $\alpha:\widehat{\mathcal{O}_{Z,z_1}}\simeq\widehat{\mathcal{O}_{Z,z_2}}$ preserving strata (i.e. if $Y\subset \overline{Y^*}$ for some stratum $Y^*$, then $\alpha$ maps the ideal sheaf corresponding to $\overline{Y^*}$ isomorphically onto the one corresponding to $\overline{Y^*}$ cf. p.73 of loc.cit.) where $z_1$ and $z_2$ are closed points of $Z$ in the same stratum $Y$. Here, $\widehat{\mathcal{O}_{Z,z_1}}$ means the completion of the local ring at $z_1$. It is well-known that the restriction of $\mathfrak{b}$ to $\mathrm{Star}\, Y$ is a finite sum, $$\sum_{i=1}^l\mathcal{O}_{\mathrm{Star}\, Y}(-D^Y_i)$$ where $D^Y_i$ is a Cartier divisor supported on $\mathrm{Star}\, Y-U$ (cf. \cite[p.83]{KKMS} Lemma 3). Then we say that $\mathfrak{b}$ is generated by $D^Y_i$ or $\mathrm{ord}_{D^Y_i}\in M^Y$. Here, $\mathrm{ord}_{\mathfrak{b}} $ is {\it the order function} of $\mathfrak{b}$ if its restriction to each $\sigma^Y$ is a convex function $\min_{1\le i\le l}\mathrm{ord}_{D^Y_i}$. Then we remark that $\overline{\mathfrak{b}}$, the integral closure of $\mathfrak{b}$, is also toroidal by Theorem $9^*$ of \cite{KKMS}. In fact, 
\[
\overline{\mathfrak{b}}=\sum_{\mathrm{ord}_{D}\ge \mathrm{ord}_{\mathfrak{b}}\,\mathrm{on}\,\sigma^Y}\mathcal{O}_{\mathrm{Star}\, Y}(-D)
\] 
on a star of each stratum $Y$ (cf. Remark \ref{retro}). 
\item Consider a normal variety $V$, an affine birational morphism $\varphi:V\to \mathrm{Star}\, Y$ and the following commutative diagram:
$$
\xymatrix{
& V \ar[dd]^{\varphi} \\
U \ar@{^{(}->}[ur] \ar@{^{(}->}[dr] & \\
& \mathrm{Star}\, Y  ,
}
$$
where $U\hookrightarrow V$ is an open immersion. Then $\varphi$ is an {\it affine toroidal morphism} if for any isomorphism $\alpha:\widehat{\mathcal{O}_{Z,z_1}}\to\widehat{\mathcal{O}_{Z,z_2}}$ preserving strata for closed points $z_1,z_2$ of $Z$ in $Y$, $\alpha$ lifts to: 
\[\xymatrix{
V\times_Z\mathrm{Spec}\,\widehat{\mathcal{O}_{Z,z_2}} \ar[d] \ar[r]^{\cong} \ar@{}[dr]|\circlearrowleft & V\times_Z\mathrm{Spec}\,\widehat{\mathcal{O}_{Z,z_1}} \ar[d] \\
\mathrm{Spec}\,\widehat{\mathcal{O}_{Z,z_2}}  \ar[r]^{\mathrm{Spec}\,\alpha} & \mathrm{Spec}\,\widehat{\mathcal{O}_{Z,z_1}}. \\
}\]
By Theorem $1^*$ of loc.cit, it is known that there is a 1-1 correspondence between the set of $(V,\varphi)$ and the set of rational polyhedral cones $\tau\subset\sigma^Y$ given by
\[
\tau\longmapsto V_\tau=\mathcal{S}pec_{\mathrm{Star}\,Y}\mathscr{A}_{\tau},
\]
where $\mathscr{A}_{\tau}=$ subsheaf $\sum_{D\in \tau^{\vee}\cap M^Y}\mathcal{O}_{\mathrm{Star}\,Y}(-D)$ of the rational function field $K(Z)$.
Here, $\tau^{\vee}=\{m\in M^Y_{\mathbb{R}};\forall n\in \tau,\,\langle n,m\rangle\ge 0 \}$.
\end{itemize}
\end{de}

The definitions were given in \cite{KKMS} and the facts stated here were proved there.
\begin{rem}\label{retro}
We remark about \cite{KKMS}. Notations as in loc.cit.
\begin{itemize}
\item We point out a small error in the proof of Theorem $1^*$ of \cite[p.81]{KKMS}. To be precise, $M^{\tilde{Y}}\cong M^{Y}$ is not always true but only the surjectivity $M^{Y}\twoheadrightarrow M^{\tilde{Y}}$ holds in general. For example, consider the blow up $\mathrm{Bl}_0(\mathbb{A}^2)$ of an affine plane $\mathbb{A}^2\cong \mathrm{Spec}\,k[x,y]$ at $(0,0)$. Let $F$ be the strict transformation of $(xy=0)$, $\tilde{Y}=\mathrm{Bl}_0(\mathbb{A}^2)\setminus F$ and $Y=\mathbb{A}^2$. When we consider $\tilde{Y}$ and $Y$ as toric varieties, we have $M^{\tilde{Y}}\cong M^Y$ (cf. \cite[Chapter I]{KKMS}). However, $\tilde{Y}$ corresponds to a ray $\mathbb{R}_{\ge0}\cdot (1,1)$ if we identify $M^{Y}=\mathbb{Z}\cdot (x)\oplus\mathbb{Z}\cdot (y)$, and hence $M^{\tilde{Y}}=\mathbb{R}\cdot (1,1)$ when we consider $\mathbb{A}^2\setminus\{xy=0\}\hookrightarrow\tilde{Y}$ as a toroidal embedding.

 The mistake does not affect their discussions in \cite[Chapter II]{KKMS} so. In fact, we can easily see $M^{\tilde{Y}}_+$ coincides with the image of $M^Y\cap\tau^\vee$ and $\sigma^{\tilde{Y}}\cong \tau$ via the canonical inclusion $N^Y_{\mathbb{R}}\hookrightarrow N^{\tilde{Y}}_{\mathbb{R}}$.
\item For the sake of the completeness, we explain the proof of \cite[Chapter II, Theorem $9^*$]{KKMS}. The proof of all the assertion works as in toric varieties except the assertion of I that $\mathscr{F}_f$ is integrally closed, where the notations are as in Theorem $9^*$ of loc.cit. Since the conclusion is locally, we may assume that $U\hookrightarrow Z=\mathrm{Star}\, Y$ is a toroidal embedding without self intersection, $f:N^Y_{\mathbb{R}}\to \mathbb{R}$ is a convex and piecewise-linear function such that $f(N^{Y})\subset \mathbb{Z}$ and $f(\lambda x)=\lambda f(x)$ for $\lambda\in\mathbb{R}$, and 
\[
\mathscr{F}_f=\sum_{\mathrm{ord}_{D}\ge f} \mathcal{O}_{\mathrm{Star}\, Y}(-D)
\]
on $Z$. We can easily see that $\mathscr{F}_f$ is a coherent sheaf of fractional ideals. $\mathscr{F}_f$ is integrally closed if and only if its Rees algebra $\bigoplus_{n\ge 0} \mathscr{F}_f^n$ is normal. To prove the latter, it is easy to see that $\bigoplus_{n\ge 0} \mathscr{F}_f^n\otimes\mathcal{O}_{Z,z}$ is normal for any closed point $z\in Z$. By replacing $Y$, we may assume that $z\in Y$. Since $\mathcal{O}_{Z,z}$ is excellent, $\beta: \bigoplus_{n\ge 0} \mathscr{F}_f^n\otimes\mathcal{O}_{Z,z}\to\bigoplus_{n\ge 0} \mathscr{F}_f^n\otimes\widehat{\mathcal{O}_{Z,z}}$ is a regular homomorphism (cf. Lemma 4 of \cite[p.253]{Ma}). Then, $\mathrm{Spec}\,\beta$ has regular fibre and hence we have $\bigoplus_{n\ge 0} \mathscr{F}_f^n\otimes\mathcal{O}_{Z,z}$ is normal if and only if so is $\bigoplus_{n\ge 0} \mathscr{F}_f^n\otimes\widehat{\mathcal{O}_{Z,z}}$. On the other hand, $\mathscr{F}_f$ analytically corresponds to an integrally closed torus-invariant ideal on a normal toric variety. Therefore, $\bigoplus_{n\ge 0} \mathscr{F}_f^n\otimes\widehat{\mathcal{O}_{Z,z}}$ is normal.
\end{itemize}
\end{rem}
 Here, we prepare the following definition motivated by the notion of Newton polygon used in \cite[Theorem 6.4]{RT07}: 

\begin{de}[Newton polyhedron]\label{NewPoly}
Suppose that $U\hookrightarrow Z$ is a toroidal embedding without self intersection, $V=\mathrm{Star}\, Y$ where $Y$ is a stratum of $Z$ and $\mathfrak{b}$ is a toroidal fractional ideal sheaf on $V$. Then we define in this paper as follows:
\begin{itemize}
\item We call a conical polyhedron $P$ in $M_{V,\mathbb{R}}$ the {\it Newton polyhedron associated with} $\mathfrak{b}$ {\it on} $V$ denoted by $\mathrm{NP}^{V}_{\mathfrak{b}}$ if $\mathfrak{b}=\sum_{i=1}^n\mathcal{O}_V(-D_i)$ and $$P=\mathrm{Conv}\langle \mathrm{ord}_{D_1},\cdots ,\mathrm{ord}_{D_n}\rangle +M^V_{\mathbb{R},+},$$ where $\mathrm{Conv}\langle \mathrm{ord}_{D_1},\cdots ,\mathrm{ord}_{D_n}\rangle$ means the convex hull of $\{ \mathrm{ord}_{D_1},\cdots ,\mathrm{ord}_{D_n}\}$. Here, $P$ is independent of the choice of $D_1,\cdots ,D_n$.
\item If $C$ is a compact polyhedron that has finite vertices and satisfies that
\[
P=C+M^V_{\mathbb{R},+},
\]
we call it a {\it generating convex set} of $P$ and then say that the vertices of $C$ {\it generates} $P$. 
\item $F$ is {\it the bounded part of faces} of $P$ if it is the union of all of the bounded faces of $P$. Since $M^V_{\mathbb{R},+}\cap(-M^V_{\mathbb{R},+})=0$, we have $F\ne \emptyset$.
\end{itemize}
\end{de}

\begin{prop}\label{prevpro}
Notations as in Definition \ref{NewPoly}. Then the followings are equivalent:
\begin{enumerate}[(1)]
\item $\mathfrak{b}$ is integrally closed.
\item For each $D \in M$, 
\[
D \in \mathrm{NP}^V_{\mathfrak{b}}\Leftrightarrow \mathcal{O}_V(-D)\subset \mathfrak{b}.
\]
\end{enumerate}
\end{prop}
\begin{proof}
Let $\mathfrak{b}=\sum \mathcal{O}(-D_i)$. We only have to show the latter condition is equivalent to the following:
\begin{itemize}
\item[(3)] For each $D \in M$, 
\[
\mathrm{ord}_D\ge \mathrm{ord}_{\mathfrak{b}}\quad \mathrm{on}\quad \sigma^{Y} \Longleftrightarrow \mathcal{O}_V(-D)\subset \mathfrak{b}.
\]
\end{itemize}
In fact, (3) is equivalent to (1) by Theorem $9^*$ of \cite[Chapter II]{KKMS}. Therefore, we only have to prove that
\begin{eqnarray}\label{preprop}
\{ h\in M_{\mathbb{R}}; h\ge \mathrm{ord}_{\mathfrak{b}}\} =\mathrm{NP}^V_{\mathfrak{b}}.
\end{eqnarray}
$\supset$: If $h\in \mathrm{NP}^V_{\mathfrak{b}}$, we have by the definition $\mathrm{ord}_{\mathfrak{b}}=\min \mathrm{ord}_{D_i}$ on $\sigma^Y$ and $$h\ge \sum_{a_i\ge0;\sum a_i=1}a_i\mathrm{ord}_{D_i}\ge\min \mathrm{ord}_{D_i}.$$
$\subset$: Suppose that $h\in M_{\mathbb{R}}$ satisfies $h\ge \mathrm{ord}_{\mathfrak{b}}$. Equivalently, $h\ge \min \mathrm{ord}_{D_i}$ on $\sigma^Y$. Take finite elements $\rho_j\in M_+\setminus\{0\}$ defining $\sigma^Y$. Then,
\[
\min_{i,j}\{ \mathrm{ord}_{D_i}-h, \rho_j\}\le0.
\]
on $N_{\mathbb{R}}$. This means that the polyhedral cone in $N_{\mathbb{R}}$ defined by inequalities $\mathrm{ord}_{D_i}-h\ge0, \rho_j\ge0$ for all $i,j$ is contained in a hyperplane. Hence, the cone generated by $\mathrm{ord}_{D_i}-h, \rho_j$ in $M_{\mathbb{R}}$ contains a certain linear subspace and there exist $m_i,n_j\ge0$ but some $m_i$ or $n_j$ is not $0$ such that 
\[
\sum m_i(\mathrm{ord}_{D_i}-h)+\sum n_j \rho_j=0.
\]
If all $m_i=0$, then $\sum n_j \rho_j$ does not vanish on the relative interior of $\sigma^Y$. Then, there exists $m_i\ne 0$ and 
\[
h\in \frac{\sum m_i\mathrm{ord}_{D_i}}{\sum m_i}+M_{\mathbb{R},+}. 
\]
Therefore, $h\in \mathrm{NP}^V_{\mathfrak{b}}$.
\end{proof}

\begin{rem}\label{bdface}
Notations as in Definition \ref{NewPoly}. We can easily show that the bounded part of faces $F$ equals to $F'$, which is the union of faces of $C$ that are also of $P$. We can prove this by the following.
\begin{claim}
 If $x\in P$, then there exists $y\in F'$ and $x-y\in M_{\mathbb{R},+}$. 
\end{claim}
In fact, it is easy to see that $F'\subset F$. Assume that $H$ is a bounded face of $P$ but not contained in $C$. Let $h_1,\cdots,h_k$ be all the vertex of $H$. By the assumption, some $h_j$ is not contained in $C$. Then there exists $y\in F'$ such that $h_j-y\in M_{\mathbb{R},+}$ by the claim. Since $h_j$ is also a vertex of $P$, $2h_j-y\in P$ and $h_j=\frac{1}{2}((2h_j-y)+y)$, we have $h_j=y\in C$. It is a contradiction. Therefore, $F\subset F'$.

\begin{proof}[Proof of Claim]
Let $C=\mathrm{Conv}\langle y_1,\cdots ,y_n\rangle$ be a generating convex set of $P$. We may assume that $P\subset M_{\mathbb{R},+}$ by translation by a sufficiently large $w\in M_{+}$. Assume that the claim does not hold for $x\in P$. As we saw in the last part of the proof of Proposition \ref{prevpro}, there exists $z\in C$ such that $x\ge z$. Therefore we may assume $x\in C$. Let $H$ be the minimal face of $P$ contains $x$. Then $H\cap C$ is a face of $C$. Assume that the face $H\cap C$ of $C$ was not a face of $P$. We see that there exist $\eta\in H\setminus C$ and $w_i\in M_{\mathbb{R},+}$ such that
\[
\eta=\sum_{a_i\ge 0; \, \sum a_i=1} a_i(y_i+w_i).
\]
If $H$ is a face defined by a linear function $f$ and $f(\sum a_iy_i)>0$, then $f(\sum a_iw_i)<0$ and it contradicts to the fact that $f(\sum a_iy_i+R\sum a_iw_i)\ge 0$ for any $R>0$. Therefore, $\sum a_iy_i\in H\cap C$ and $\xi=\eta-\sum a_iy_i\in M_{\mathbb{R},+}\setminus \{0\}$. Since $x$ is contained in the relative interior of $H$ and $M_{\mathbb{R},+}\cap(-M_{\mathbb{R},+})=0$, there exists $s>0$ such that $x-s\xi\in H$ but $x-(s+\epsilon)\xi\not\in P$ for $\epsilon>0$. Then $x-s\xi$ is contained in a proper face $H'$ of $H$ and it follows from the similar argument as above that there exists $\gamma\in H'\cap C$ such that $x-s\xi\ge \gamma$. Since $\mathrm{dim}(H'\cap C)<\mathrm{dim}(H\cap C)$, we can iterate the above arguments by replacing $x$ and $H$ by $\gamma$ and $H'$ to obtain a face $H''$ of $P$ that is also a face of $C$ and $y\in H''$ such that $x\ge y$. It is a contradiction.
\end{proof}
 
 We also remark that vertices of $F$ coincides with vertices of $P$ and hence generate $P$.
\end{rem}

\begin{rem}
In Proposition \ref{prevpro}, if $\mathfrak{b}$ is integrally closed, then the following are equivalent.
\begin{itemize}
\item[(1)] $\mathfrak{b}$ is invertible;
\item[(2)] $\mathrm{NP}_{\mathfrak{b}}^V$ has only one vertex.
\end{itemize}
In fact, $(2)\Rightarrow (1)$ is easy and if $\mathfrak{b}$ is invertible, there exists a Cartier divisor $D$ supported on $V-U$ such that $\mathfrak{b}=\mathcal{O}_V(-D)$
due to Lemma 3 of \cite[p.83]{KKMS}.
\end{rem}

\begin{lem}\label{3rd}
Notations as in Proposition \ref{prevpro}. Suppose that $\mathrm{NP}^V_{\mathfrak{b}}$ has integral vertices $x_1,x_2,\cdots ,x_k$ and $g:V'\to V$ is an affine toroidal morphism corresponding to a polyhedral cone $\tau \subset \sigma^Y$ (cf. Definition \ref{toroidal}). Then $\mathrm{NP}^{V'}_{g^{-1}\mathfrak{b}}$ is generated by $x_1,x_2,\cdots ,x_k$. If $g^{-1}\mathfrak{b}$ is invertible, then one of $y_1,y_2,\cdots ,y_k$, which are the images of $x_1,x_2,\cdots ,x_k$ under the canonical map $M_V\to M_{V'}$, is the unique vertex of $\mathrm{NP}^{V'}_{g^{-1}\mathfrak{b}}$.
\end{lem}

\begin{proof}
 By the equation (\ref{preprop}) in the proof of Proposition \ref{prevpro} and Theorem $9^*$ of \cite{KKMS}, we have $\mathrm{NP}^V_{\mathfrak{b}}=\mathrm{NP}^V_{\overline{\mathfrak{b}}}$ if $\overline{\mathfrak{b}}$ is the integral closure of $\mathfrak{b}$. Let $\mathfrak{c}$ be an ideal on $V$ generated by $x_1,x_2,\cdots ,x_k$. Then $\overline{\mathfrak{b}}=\overline{\mathfrak{c}}$ and hence $\overline{g^{-1}\mathfrak{b}}=\overline{g^{-1}\mathfrak{c}}$. Hence, $\mathrm{NP}^{V'}_{g^{-1}\mathfrak{b}}=\mathrm{NP}^{V'}_{g^{-1}\mathfrak{c}}$. We can see $\mathrm{NP}^{V'}_{g^{-1}\mathfrak{c}}$ is generated by $y_1,y_2,\cdots ,y_k$. For the second assertion, $\mathrm{NP}^{V'}_{g^{-1}\mathfrak{b}}$ has the unique vertex by the previous remark. Since $g^{-1}\mathfrak{b}$ is invertible, $g$ factors through the normalized blow up $W$ along $\mathfrak{b}$, $W$ is the normalized blow up along $\mathfrak{c}$ due to Lemma 1.8 of \cite{BHJ}. Hence, $g^{-1}\mathfrak{b}=g^{-1}\mathfrak{c}$ and is generated by $y_1,y_2,\cdots ,y_k$. Then $g^{-1}\mathfrak{b}$ is generated by one of $y_1,y_2,\cdots ,y_k$ since $g^{-1}\mathfrak{b}$ is invertible and whose bounded part of faces is one point. Thus, one of $y_1,y_2,\cdots ,y_k$ is the unique vertex.
\end{proof}

\begin{ex}\label{yokuwakaru}
We illustrate the proof of Theorems \ref{slope} and \ref{impl} briefly by a few examples. Let $U\hookrightarrow X$ be a toroidal embedding and take a flag ideal $$\mathfrak{a}=\sum_{k=0}^r\mathcal{O}_X(-D_k)t^k,$$ where each $D_k$ is a Cartier divisor supported on $X-U$. It is necessary for (*) in Theorem \ref{mu} to seek out what is a generator of $\mathfrak{a}^n$. Then we can consider the Newton polyhedron of $\mathfrak{a}$. Note that $\mathrm{NP}_{\mathfrak{a}^k}=k\mathrm{NP}_{\mathfrak{a}}$. As we saw, the Newton polyhedron tells us what is a generator of $\mathfrak{a}^n$ for $n\ge 0$ if $\mathfrak{a}^n$ is integrally closed. First, as in \cite[Theorem 6.4]{RT07}, let us consider when each $D_k$ is some multiple of one Cartier divisor $D$. When
\[
\mathfrak{a}=t^4+\mathcal{O}(-3D),
\]
$\overline{\mathfrak{a}}$ contains $\mathcal{O}(-2D)t^{2}$, which is not on the bounded part of faces $F$. Therefore, we can conclude that the integral points of $F$ do not necessarily generate $\mathfrak{a}$. However, if there exists $D'$ such that $D=2D'$, the integral points of $nF$ generate $\mathfrak{a}^n$. In general, we will prove in Proposition \ref{cov} that there exists $\pi:X'\to X$, where $\pi$ is a composition of its 2-cyclic coverings and its Bloch-Gieseker covering \cite{BG} such that there exists $D'$ that satisfies $\pi^*D=2D'$. Here, we also remark that $F$ is one-dimensional then. 

Next, let us consider 
\[
\mathfrak{a}=t^2+\mathcal{O}(-E_1)t+\mathcal{O}(-E_1-E_2).
\]
Then $$\mathfrak{a}^2=t^4+\mathcal{O}(-E_1)t^3+(\mathcal{O}(-E_1)+\mathcal{O}(-E_2))\mathcal{O}(-E_1)t^2+\mathcal{O}(-2E_1-E_2)t+\mathcal{O}(-2E_1-2E_2)$$ and hence $F$ is not one-dimensional. We can easily see that if $\mathfrak{a}$ does not satisfy (*). If $\frac{1}{2}E_1$ and $\frac{1}{2}E_2$ are also Cartier divisors, 
\[
\overline{\mathfrak{a}}=t^2+\left(\mathcal{O}(-\frac{1}{2}E_1)+\mathcal{O}(-\frac{1}{2}E_2)\right)\mathcal{O}(-\frac{1}{2}E_1)t+\mathcal{O}(-E_1-E_2).
\]
Then the intersection of $F$ and the hyperplane defined by $t^1$ is generated by $E_1$ and $\frac{1}{2}(E_1+E_2)$. Assume that $f_1,f_2$ are local generators of $\frac{1}{2}E_1$ and $\frac{1}{2}E_2$ respectively. Take the blow up $\pi:X'\to X$ along $\mathcal{O}(-\frac{1}{2}E_1)+\mathcal{O}(-\frac{1}{2}E_2)$ with the exceptional divisor $\tilde{E}$. Assume that $F_1$ and $F_2$ are the strict transformations of $\frac{1}{2}E_1$ and $\frac{1}{2}E_2$ respectively. Then
\[
(\pi\times\mathrm{id}_{\mathbb{A}^1})^{-1}\overline{\mathfrak{a}}=t^2+\mathcal{O}(-2\tilde{E}-F_1)t+\mathcal{O}(-4\tilde{E}-2F_1-2F_2).
\]
We can easily see that the bounded part of faces $F$ is one-dimensional and $(\pi\times\mathrm{id}_{\mathbb{A}^1})^{-1}\overline{\mathfrak{a}}$ satisfies that (*). However, if $L$ is a polarization of $X$ and $L'=\pi^*L$, $L'$ is big and semiample but not ample and we can not apply Theorem \ref{sl} immediately. On the other hand, $L'$ is an accumulation point of ample $\mathbb{Q}$-line bundles to which we can apply Theorem \ref{sl}, and hence we can calculate $(\mathcal{J}^H)^\mathrm{NA}$-energy by using the mixed multiplicity of Cartier divisors on $X'$. 
\end{ex}

 \section{Construction of the Alternation}\label{ThMain}
In this section we prove Theorems \ref{a} and \ref{b} that we stated in \S \ref{Intro}. We want to show the following theorem motivated by Theorems 6.1 and 6.4 of \cite{RT07}:

\begin{thm}\label{slope}
Given an $n$-dimensional polarized variety $(X,L)$, and an flag ideal $\mathfrak{a}=\sum\mathfrak{a}_it^i\subset \mathcal{O}_{X\times \mathbb{A}^1}$ such that $\mathfrak{a}_0\ne0$. Then there exists an alternation $\pi:X'\to X$ (i.e. $\pi$ is a generically finite and proper morphism) such that $X'$ is smooth and irreducible, $D_0$ is an snc divisor corresponding to $\pi^{-1}\mathfrak{a}_0$ and the integral closure $\overline{(\pi\times\mathrm{id}_{\mathbb{A}^1})^{-1}\mathfrak{a}}$ of the inverse image of $\mathfrak{a}$ to $X'\times\mathbb{A}^1$ satisfies (*) in Theorem \ref{mu}. Moreover, if $(\mathrm{Bl}_{\mathfrak{a}}(X\times \mathbb{A}^1),L_{\mathbb{A}^1}-E)$ is a semiample test configuration, where $E=\Pi^{-1}(\mathfrak{a})$ and $\Pi:\mathrm{Bl}_{\mathfrak{a}}(X\times \mathbb{A}^1)\to X\times \mathbb{A}^1$ is the blow up, $(\pi\times\mathrm{id}_{\mathbb{A}^1})^{-1}\mathfrak{a}\cdot (\pi\times\mathrm{id}_{\mathbb{A}^1})^*L_{\mathbb{A}^1}$ is semiample and $\pi^*L-D_0$ is nef.
\end{thm}

The proof consists of three steps. First, we replace $X$ by a resolution of singularities $X_1$ of $X$. Second, we take a finite covering $X_2$ of $X_1$. Finally, we construct $X'$ as a toroidal blow up of $X_2$. Abusing the notaion, we denote $\pi\times\mathrm{id}_{\mathbb{A}^1}$ by $\pi$.

 First, we can take a log resolution of $X$ by the theorem of Hironaka \cite{Hi} and by the assumption that $\mathrm{char}\,k=0$, and may assume that each $Z_i$ the closed subscheme corresponding to $\mathfrak{a}_i$ is a simple normal crossing divisor (snc for short). We can easily see if $(Y,D)$ is a log smooth variety, which means $Y$ is smooth and $D$ is an snc divisor in this paper, and $U=Y-D$, $U\hookrightarrow Y$ is a toroidal embeddings without self intersection (cf. \S \ref{NewTro}). In fact, we can take $(\mathbb{A}^k\times(\mathbb{A}^1-\{0\})^{n-k},(\mathbf{0},\mathbf{1}))$ as a local model. Hence, if $U=X-Z_0$, $U\hookrightarrow X$ and $U\times (\mathbb{A}^1-0)\hookrightarrow X\times \mathbb{A}^1$ are toroidal embeddings without self intersection. Furthermore, we can easily check $\mathfrak{a}_i$ and $\mathfrak{a}$ are toroidal ideal sheaves in Definition \ref{toroidal}. Fix $r>0$ such that $\mathfrak{a}_r=\mathcal{O}_X$. Moreover, the integral closure $\overline{\mathfrak{a}}$ of $\mathfrak{a}$ is also a toroidal and flag ideal. However, $\mathscr{I}_m$ is not necessarily invertible on $X$ but toroidal, where
\[
\overline{\mathfrak{a}}=\sum_{m=0}^r\mathscr{I}_mt^m.
\]

Suppose that $Y_i$ is any stratum of $X$ and $X=\coprod Y_i$. We can see the open sets $\mathrm{Star}\, Y_i\times \mathbb{A}^1$'s are stars of strata $Y_i\times\{0\}$ in $X\times \mathbb{A}^1$ and 
\[
X\times\mathbb{A}^1= \bigcup\mathrm{Star}\, Y_i\times \mathbb{A}^1.
\]
 Moreover, $M_+^{\mathrm{Star}\, Y_i\times \mathbb{A}^1}=M_+^{\mathrm{Star}\, Y_i}\times \mathbb{Z}_{\ge 0}(\mathrm{ord}_{X\times \{0\}})$. Let $P_i=\mathrm{NP}_{\mathfrak{a}}^{\mathrm{Star}\, Y_i\times \mathbb{A}^1}$ and $(\mathrm{ord}_{X\times \{0\}})^*$ be the linear function that takes $1$ at $\mathrm{ord}_{X\times \{0\}}$ and vanishes on $M^{\mathrm{Star}\, Y_i}$. Let also $F_i$ be the bounded part of faces of $P_i$ and $C_i=\mathrm{Conv}(F_i)$. Then the following holds:
\begin{lem} \label{sectm}
If the notations are as above, then $P_i\cap \{ (\mathrm{ord}_{X\times \{0\}})^*=m\}$ is a rational conical polyhedron of $M_{\mathbb{R},+}^{\mathrm{Star}\, Y_i}$ such that
 \[
 P_i\cap \{ (\mathrm{ord}_{X\times \{0\}})^*=m\}=(C_i\cap \{ (\mathrm{ord}_{X\times \{0\}})^*=m\})+M_{\mathbb{R},+}^{\mathrm{Star}\, Y_i}
  \]
  and hence has finite rational vertices for any $0\le m\le r$.
\end{lem}

\begin{proof}
 Note that $C_i$ must contain $(0,r)(=(0,r\mathrm{ord}_{X\times \{0\}}))$. If $(x_i,t_i)$ are vertices of $C_i$, there exists $(y_i,s_i)\in M^{\mathrm{Star}\, Y_i\times \mathbb{A}^1}_{\mathbb{R},+}$ for $(z,m)\in P_i$ such that
\[
(z,m)=\sum_{a_i\ge 0;\sum a_i=1} a_i(x_i+y_i,t_i+s_i).
\]
Note that $s_i,t_i\ge 0$ and let $u=\sum a_it_i\le m$. Then
\[
\left(\frac{r-m}{r-u}\sum a_ix_i,m\right)=\frac{r-m}{r-u}\left(\sum a_ix_i,u\right)+\frac{m-u}{r-u}(0,r) \in C_i\cap \{ (\mathrm{ord}_{X\times \{0\}})^*=m\}
\]
and
\[
z\ge \sum a_ix_i \ge \frac{r-m}{r-u}\sum a_ix_i
\]
since $\mathfrak{a}\subset \mathcal{O}_{X\times \mathbb{A}^1}$ and $C_i$ contains $(0,r)$. The rest of the assertion follows from the fact that $C_i\cap \{ (\mathrm{ord}_{X\times \{0\}})^*=m\}$ is a compact rational polyhedron.
\end{proof}

Next, we want to multiply $M_+^{\mathrm{Star}\, Y_i}$ by a sufficiently divisible integer $l$ that $C_i\cap \{ (\mathrm{ord}_{X\times \{0\}})^*=m\}$ has the integral vertices for $m\in\mathbb{Z}_{\ge0}$ on any $\mathrm{Star}\,Y_i$. For this, we need the following proposition.

\begin{prop}\label{cov}
Suppose that $(V,E)$ is a log smooth variety and $E=\bigcup _{i =1}^s  E_i$ is the irreducible decomposition. We consider $V-E\hookrightarrow V$ to be a toroidal embedding without self-intersection. Let $l$ be an integer. Then there exists a finite covering $\tilde{\pi}:\tilde{V}\to V$ such that $(\tilde{V},\mathrm{red}(\tilde{\pi}^*E))$ is also a log smooth variety and 
\[
\tilde{\pi}^*E_i=l\, \mathrm{red}(\tilde{\pi}^*E_i).
\]
 In other words, if $Y$ is a stratum of $V$ contains a closed point $x$ of $V$ and a closed point $x'\in V$ satisfies $\tilde{\pi}(x')=x$, the irreducible component $Y'$ of $\tilde{\pi}^{-1}(Y)$ contains $x'$ is a stratum of $\tilde{V}$ and the canonical map gives an isomorphism
$$M_+^{Y'}\cong \frac{1}{l}M_+^Y. $$ 
 Moreover, if $\mathfrak{b}$ is a toroidal fractional ideal sheaf on $V$, $\mathrm{NP}^{\mathrm{Star}\, Y'}_{\tilde{\pi}^{-1}\mathfrak{b}}=\mathrm{NP}^{\mathrm{Star}\, Y}_{\mathfrak{b}}\subset M^Y\subset\frac{1}{l}M^Y$ via the isomorphism above.
\end{prop}  

We prove the following two lemmas for the proposition.

\begin{lem}\label{cov1}
Suppose that $(V,E)$ is a log smooth variety and $D$ is a smooth divisor suppoted on $E$ such that there exists a line bundle $L$ on $V$ that satisfies $D\in |lL|$. We remark that $D$ might be reducible. Let $\pi:V'\to V$ be an $l$-cyclic covering branched along $D$. Then $(V',\mathrm{red}(\pi^{*}E))$ is also a log smooth variety, 
\[
\pi^*D=l\, \mathrm{red}(\pi^*D)
\]
 and $\pi^*D'$ is also reduced for any irreducible component of $E$ not contained in $\mathrm{supp}\, D$. In other words, for any closed point $x\in V$ and any stratum $Y$ contains $x$, if $\pi(x')= x$, then $Y'$ the irreducible component of $\pi^{-1}(Y)$ contains $x'$ is the stratum of $V'$ and the canonical map gives
$$M_+^{Y'}\cong \left\{
\begin{split}
&M_+^{Y}+\mathbb{Z}_{\ge0}\frac{1}{l}D& \quad &(Y\subset D) \\
&M_+^{Y}& \quad &(Y\cap D=\emptyset).
\end{split} \right. $$
Moreover, if $\mathfrak{b}$ is a toroidal fractional ideal sheaf on $V$, $\mathrm{NP}^{\mathrm{Star}\, Y'}_{\pi^{-1}\mathfrak{b}}=\mathrm{NP}^{\mathrm{Star}\, Y}_{\mathfrak{b}}$ via the isomorphism above.
\end{lem}

\begin{proof}[Proof of Lemma \ref{cov1}]
First, we can easily show that $(V',(\bigcup_{i=1}^k \pi^{*}D_i))$ is log smooth where $D_i$'s are integral divisors on $V$ supported on $E$ and different from each other and irreducible components of $D$. In fact, this holds by the following:
\begin{claim}
 For any subset $J$ of $\{1,\cdots,k\}$, $\bigcap_{j\in J}\pi^*D_j$ is isomorphic to the $l$-cyclic covering of $\bigcap_{j\in J}D_j$ branched along a smooth divisor $D\cap\bigcap_{j\in J}D_j$ if the intersection is not void. Let $t$ be an indeterminate. If $D\cap\bigcap_{j\in J}D_j=\emptyset$, $\bigcap_{j\in J}\pi^*D_j$ is isomorphic to the $l$-cyclic unramified covering
\[
\mathcal{S}pec_{\bigcap_{j\in J}D_j} \left(\bigoplus_{i=0}^\infty (L^{-i}|_{\bigcap_{j\in J}D_j})t^i/(t^l-s_1)\right)
\] 
where $s_1$ is a nowhere-vanishing section of $L^l|_{\bigcap_{j\in J}D_j}$ (cf. Definition 2.49 of \cite{KM}). In particular, $\bigcap_{j\in J}\pi^*D_j$ is smooth.
\end{claim}
To prove the claim, $V'\cong \mathcal{S}pec_V (\bigoplus_{i=0}^\infty L^{-i}t^i/(t^l-s))$ where $t$ is an indeterminate and $s$ is the global section of $L^l$ corresponding to $D$ by the assumption (cf. Definition 2.50 of \cite{KM}). If $D\cap\bigcap_{j\in J}D_j\ne \emptyset$, $s|_{\bigcap_{j\in J}D_j}\in \Gamma (L^{l}|_{\bigcap_{j\in J}D_j})$ corresponds to this since the intersection is transversal. Therefore, by 
\[
\bigcap_{j\in J}\pi^*D_j\cong \mathcal{S}pec_{\bigcap_{j\in J}D_j} \left(\bigoplus_{i=0}^\infty( L^{-i}|_{\bigcap_{j\in J}D_j})t^i/(t^l-s|_{\bigcap_{j\in J}D_j})\right),
\]
it is isomorphic to the $l$-cyclic covering of $\bigcap_{j\in J}D_j$ branched along a smooth divisor $D\cap\bigcap_{j\in J}D_j$. Hence, it is smooth by Lemma 2.51 of \cite{KM}. On the other hand, if $D\cap\bigcap_{j\in J}D_j= \emptyset$, $s|_{\bigcap_{j\in J}D_j}$ is a nowhere-vanishing section. Therefore, $L^l|_{\bigcap_{j\in J}D_j}$ is trivial and $\bigcap_{j\in J}\pi^*D_j$ is isomorphic to the $l$-cyclic unramified covering defined by $s|_{\bigcap_{j\in J}D_j}$. Since $\bigcap_{j\in J}\pi^*D_j$ is {\' e}tale over a smooth variety, it is also smooth. We complete the proof of the claim.
 In particular, $V'$ is a smooth variety. In fact, let a closed point $x\in D$. Then $\mathcal{O}_{V,x}$ is a regular local ring and hence is a unique factorization domain (\cite[Theorem 48]{Ma}). The restriction of $s$ is an irreducible element of $\mathcal{O}_{V,x}$. Note that $\mathrm{Spec}\,\mathcal{O}_{V,x}\times_{V}V'\cong \mathrm{Spec}\,(\mathcal{O}_{V,x}[t]/(t^l-s))$ and $(t^l-s)$ is a prime ideal. Therefore, the generic fibre is irreducible and so is $V'$.
 
 Next, we prove the first assertion. We can easily see that $$\mathrm{red}(\pi^*D)\cap\bigcap_{j\in J}\pi^*D_j$$ is smooth where $D_j$ and $J$ as above if $\mathrm{red}(\pi^*D)\cap\bigcap_{j\in J}\pi^*D_j\ne \emptyset$ by Lemma 2.51 of loc.cit. In fact, $\pi^{-1}\bigcap_{j\in J}D_j\to \bigcap_{j\in J}D_j$ is an $l$-cyclic covering branched along $D_{|\bigcap_{j\in J}D_j}$ by the claim. Let $Z=\bigcap_{j\in J}D_j$ and $F=D_{|\bigcap_{j\in J}D_j}$. If $s\in \Gamma(Z,L^l|_{Z})$ is a global generator of $F$, $\pi^{-1}Z$ is isomorphic to $\mathrm{Spec}\, \mathcal{O}_Z[t]/(t^l-s)$ affine locally. Hence, $\mathrm{red}(\pi^*F)$ is isomorphic to $F$ and it is smooth. On the other hand, we can see $\pi^*F=l\mathrm{red}(\pi^*F)$ by $\pi^{-1}Z\cong \mathrm{Spec}\, \mathcal{O}_Z[t]/(t^l-s)$. In particular, we have $\pi^*D=l\mathrm{red}(\pi^*D)$ when $J=\emptyset$. Since 
 \begin{align*}
 (l\mathrm{red}(\pi^*D))\cap \pi^{-1}Z&= \pi^*D\cap \pi^{-1}Z \\
 &= \pi^*F\\
 &=  l\mathrm{red}(\pi^*F),
 \end{align*}
 we have $\mathrm{red}(\pi^*D)\cap \pi^{-1}Z=\mathrm{red}(\pi^*F)$. Hence, $\mathrm{red}(\pi^*D)\cap\bigcap_{j\in J}\pi^*D_j$ is smooth and we finish the proof of the first assertion. For the second assertion, if $Y'$ is an irreducible component of $\pi^{-1}Y$ contains $x'$, it is the unique stratum contains $x'$ defined $\mathrm{red}(\pi^*D)$ and the unique irreducible component of each $\pi^*D_i$ contains $x'$. For any toroidal fractional ideal sheaf $\mathfrak{b}$, there exists a finite number of Cartier divisors $F_j$ supported on $E$ such that 
\[
\mathfrak{b}=\sum\mathcal{O}_{\mathrm{Star}\, Y}(-F_j)
\]
locally. Then $\pi^{-1}\mathfrak{b}=\sum\mathcal{O}_{\mathrm{Star}\, Y'}(-\pi^*F_j)$. Since $\pi^*D=l\mathrm{red}(\pi^*D)$ and $\pi^*D'$ is reduced for any divisor different from irreducible components of $D$, the rest of the assertion follows immediately.
\end{proof}

\begin{lem}\label{cov2}
 Suppose that $(V,E)$ is log smooth and there exists a finite covering $\pi:V_1\to V$ such that $(V_1,\pi^*E)$ is also log smooth. For any closed point $x\in V$, take the stratum $Y$ contains $x$. Then $Y'$ that is the irreducible component of $\pi^{-1}(Y)$ contains $x_1$ such that $\pi(x_1)= x$ is the stratum of $V_1$ and
$$
M_+^{Y'}\cong M_+^{Y}.
$$
\end{lem} 

We can prove Lemma \ref{cov2} similarly to Lemma \ref{cov1}. 

\begin{proof}[Proof of Proposition \ref{cov}]
By Theorem 4.1.10 of \cite{Laz} (cf. \cite{BG}, \cite{KM}), we can take a smooth variety $V_1$ and a finite covering $p_1:V_1\to V$ such that there exists a line bundle $L_i$ such that $p_1^*E_i\in |lL_i|$ for any $E_i$, and $p_1^*E$ is an snc divisor. We emphasize that $V_1$ is irreducible due to the theorem of Fulton-Hansen \cite[Example 3.3.10]{Laz}.

Therefore, we can take $p_2:V_2\to V_1$ the $l$-cyclic covering of $V_1$ branched along $p_1^*E_1$. Due to Lemma \ref{cov1}, $p_2^*p_1^*E_2\in |l(p_2^*L_2)|$ is also smooth and hence we can construct $p_3:V_3\to V_2$ the $l$-cyclic covering of $V_2$ branched along $p_2^*p_1^*E_2$. Therefore, we can construct $p_{i+1}$ for $E_i$ inductively and let $\tilde{\pi}=p_1\circ \cdots \circ p_{s+1}:\tilde{V}=V_{s+1} \to V$. We can see $(\tilde{V},\mathrm{red}(\tilde{\pi}^*E))$ is also a log smooth variety by Lemma \ref{cov1}. Take any $\tilde{x}\in \tilde{V}$ and let $x_i=p_i\circ \cdots \circ p_{s+1}(\tilde{x})$ and $x=\tilde{\pi}(\tilde{x})$. Let $I\subset \{1,\cdots ,s\}$ and $J$ be its complement and suppose that the stratum $Y$ that contains $x$ is the irreducible component of $\bigcap _{i\in I}E_i-\bigcup _{j\in J}E_j$. If $Y_k$ is the irreducible component of $(p_{1}\circ \cdots \circ p_k)^{-1}(Y)$, which contains $x_k$, then $Y_k$ is a stratum of $V_k$ and
$$M^+_{Y_k}\cong M^+_{Y}+\sum_{t\in I\cap \{ 1,2,\cdots ,k-1\}}\mathbb{Z}_{\ge0}\frac{1}{l}E_t$$
by Lemmas \ref{cov1} and \ref{cov2}. Therefore, if $\tilde{Y}$ is the irreducible component of $\tilde{\pi}^{-1}(Y)$ contains $\tilde{x}$, then $\tilde{Y}$ is the stratum of $\tilde{V}$ and
$$
M^+_{\tilde{Y}}\cong lM^+_{Y}.
$$
The rest is also easy.
\end{proof}

According to Proposition \ref{cov}, by taking a finite covering $\tilde{\pi}:\tilde{X}\to X$ and replacing $X$ by $\tilde{X}$, we may assume that $C_i\cap \{ (\mathrm{ord}_{X\times \{0\}})^*=m\}$ has the integral vertices for $n\in\mathbb{Z}_{\ge0}$ on $\mathrm{Star}\, Y_i\times \mathbb{A}^1$. For the third step of the proof of Theorem \ref{slope}, we want the following.

\begin{lem}\label{fa}
Notations as above. If $P_i$ has a one-dimensional bounded part of faces $F_i$ and $F_i\cap \{ (\mathrm{ord}_{X\times \{0\}})^*=m\}$ consists of an integral point for each $m$ and $i$, the normalization of $\mathfrak{a}$ satisfies the condition (*).
\end{lem}

\begin{proof}
We may replace $\mathfrak{a}$ by its integral closure and may assume that $\mathfrak{a}=\sum \mathscr{I}_mt^m$ contains all the integral point of $F_i$ on $\mathrm{Star}\, Y_i\times \mathbb{A}^1$. It follows from the assumption $F_i$ is one-dimensional and the fact $\mathrm{NP}^{\mathrm{Star}\, Y_i}_{\mathfrak{a}^k}=k\mathrm{NP}^{\mathrm{Star}\, Y_i}_{\mathfrak{a}}$ that each $\mathscr{I}_m$ is invertible and $\mathfrak{a}^k$ also contains all the integral point of $kF_i$, which is the bounded part of faces of $\mathrm{NP}^{\mathrm{Star}\, Y_i}_{\mathfrak{a}^k}$ and is also one-dimensional. On the other hand, if $\mathfrak{a}^k$ is generated by the integral points of $kF_i$ for $k\ge 1$, $\mathfrak{a}$ satisfies (*). Therefore, we only have to prove that if $\mathfrak{a}$ contains all integral points of $F_i$ and $F_i$ is one-dimensional, then $\mathfrak{a}$ is generated by the integral points of $F_i$. It follows from the next claim immediately:
\begin{claim}
Suppose that $r'\le r$ and satisfies that $(0,r')\in F_i$, and let $\{(y_m,m)\}= F_i\cap\{ (\mathrm{ord}_{X\times \{0\}})^*=m\}$ for $0\le m\le r'$. Then, $x\ge y_m$ for $(x,m)\in P_i\cap M_{\mathrm{Star}\,Y_i}$ and there exists $0<\epsilon<1$ such that $(1-\epsilon)y_m-y_{m+1} \in M_{Y_i,\mathbb{R}}^+$ for $0\le m\le r'-1$. 
\end{claim}
We prove the claim by the induction on $m$. If $m=0$ and $(x,m)\in P_i\cap M_{\mathrm{Star}\,Y_i}$, there exists a point $(y,s)$ of $F_i$ such that $(x,0)\ge (y,s)$ by the claim of Remark 4.5. Since $s\ge 0$, $s=0$ and $y=y_0$. Moreover, $\frac{r'-1}{r'}y_0\in P_i$ and we can conclude that there exists a point $(y,s)$ of $F_i$ such that $0< s\le 1$ and $y\le \frac{r'-1}{r'}y_0$. Since $F_i$ is one-dimensional, $y=(1-s)y_0+sy_1$. Therefore,
\[
0\le y_1\le \frac{r's-1}{r's}y_0.
\] 
We see that if $y_1\ne 0$, we have $r's>1$ and hence the second assertion follows when $m=0$. Assume that $0\le m'\le r'-1$ satisfies that the claim holds for $m< m'$. For $(x,m')\in P_i\cap M_{\mathrm{Star}\, Y_i\times \mathbb{A}^1}$, there exists a point $(y,s)$ of $F_i$ such that $(x,m')-(y,s)\in M_{\mathrm{Star}\, Y_i\times \mathbb{A}^1,\mathbb{R}}^+$ by Remark \ref{bdface}. By the induction hypothesis and the assumption that $F_i$ is one-dimensional, we can see $y\ge y_{m'}$. The second assertion when $m=m'$ follows from the similar argument as when $m=0$. Therefore the claim holds also for $m'$.
\end{proof}

\begin{rem}
From the first of the proof of this lemma, $P_i\cap \{ (\mathrm{ord}_{X\times \{0\}})^*=0\}$ has the unique vertex $D_0$.
\end{rem}

\begin{proof}[Proof of Theorem \ref{slope}]
We may assume that all the $C_i\cap \{ (\mathrm{ord}_{X\times \{0\}})^*=m\}$ has finitely many integral vertices by Proposition \ref{cov}. Let $\mathscr{I}_m$ be an ideal sheaf globally defined on $X$ such that 
\[
\overline{\mathfrak{a}}=\sum_{m=0}^r \mathscr{I}_mt^m.
\]
Let also the integral vertices of $C_i\cap \{ (\mathrm{ord}_{X\times \{0\}})^*=m\}$ be $x^{(m)}_1,\cdots ,x^{(m)}_k$. Then each $\mathscr{I}_m$ is a toroidal ideal sheaf whose Newton polyhedron is generated by $\mathrm{Conv}\langle x^{(m)}_1,\cdots ,x^{(m)}_k\rangle$ on $\mathrm{Star}\, Y_i$ by Remark \ref{bdface} and Lemma \ref{sectm}. Since $X$ is toroidal, we can take its toroidal log resolution $\pi:X'\to X$ such that $(X',\pi^*D_0)$ is log smooth and the inverse images of $\mathscr{I}_m$ are line bundles by Theorems $10^*$ and $11^*$ of \cite[Chapter II]{KKMS}. Take $\mathrm{Star}\, Y'$ of $X'$ corresponds to a subcone $\sigma^{Y'}\subset \sigma^{Y_i}$. Then there exists a canonical surjection $M_{Y_i}\twoheadrightarrow M_{Y'}$ and $\mathrm{Image}\, M^+_{Y_i}\subset M^+_{Y'}$. $C_i$ is a generating convex set of $P_i$ and the image of $C_i$ also generates $\mathrm{NP}_{\pi^{-1}\mathfrak{a}}^{\mathrm{Star}\, Y'\times\mathbb{A}^1}$ by Lemma \ref{3rd}. Therefore, the bounded part of faces $F'$ on $\mathrm{Star}\, Y'$ is contained in the image of $F_i$ and the image of $C_i$, which we call $C'$. Furthermore, one of the images of $x^{(m)}_1,\cdots ,x^{(m)}_k$ is smaller than the others on $\sigma^{Y'}$ also by Lemma \ref{3rd}. Therefore, $F'$ meets $\{ (\mathrm{ord}_{X\times \{0\}})^*=m\}$ at one point $x^{(m)}_{\min}$. On the other hand, $C'\cap \{ m\le(\mathrm{ord}_{X\times \{0\}})^*\le m+1\}$ is the convex hull of all $x^{(m)}_j,x^{(m+1)}_j$ since $C'$ is generated by integral vertices. Therefore, $F'$ meets $\{ (\mathrm{ord}_{X\times \{0\}})^*=m+s\}$ only at one point $(1-s)x^{(m)}_{\min}+sx^{(m+1)}_{\min}$ for $0<s<1$ and hence $F'$ is one-dimensional. Then the first assertion follows from Lemma \ref{fa} and the rest follows from Corollary 5.8 of \cite{RT07} since $\mathscr{I}_0$ is a line bundle from the first. 
\end{proof}

We have the following consequence by Theorem \ref{slope}:

\begin{rem}\label{imp}
We can use Theorem \ref{sl} and Theorem \ref{slope} for computations of $(\mathcal{J}^H)^\mathrm{NA}$. In fact, for any $n(\ge2)$-dimensional polarized variety $(X,L)$ with a divisor $H$ and any class $\phi\in\mathcal{H}^\mathrm{NA}(L)$, there exists a flag ideal $\mathfrak{a}=\sum \mathfrak{a}_kt^k$ and $s\in\mathbb{Q}_{\ge0}$ such that a semiample test configuration $(\mathcal{X},\mathcal{L})=(\widetilde{\mathrm{Bl}_{\mathfrak{a}}(X\times \mathbb{A}^1)},L_{\mathbb{A}^1}-sE)$, where $\widetilde{(\cdot)}$ means the normalization, $s\ge 0$ and $E$ is the exceptional divisor. If $s=0$, $\phi$ is trivial. Hence, we may assume that $s>0$. Since $(\mathcal{J}^{H})^\mathrm{NA}$ is homogeneous in $\mathcal{L}$, 
\[
(\mathcal{J}^{H})^\mathrm{NA}(\mathcal{X},\mathcal{L})=s(\mathcal{J}^{H})^\mathrm{NA}(\mathcal{X},\frac{1}{s}L_{\mathbb{A}^1}-E).
\]
Then replace $L$ by $\frac{1}{s}L$ and we may assume that $s=1$. If necessary, replace $(\mathcal{X},L_{\mathbb{A}^1}-E)$ by $(\mathcal{X},L_{\mathbb{A}^1}-E+m\mathcal{X}_0)$ and we may assume that $\mathfrak{a}_0\ne 0$. Now, $\mathfrak{a}$ satisfies the assumption of Theorem \ref{slope}. Then there exists an alternation $\pi: X'\to X$ in Theorem \ref{slope} whose degree is $l$ over $X$. We can see 
\[
(\mathcal{J}^{\pi^*H})^\mathrm{NA}(\widetilde{\mathrm{Bl}_{\pi^{-1}\mathfrak{a}}(X'\times \mathbb{A}^1)},\pi^*L_{\mathbb{A}^1}-E')=(\mathcal{J}^H)^\mathrm{NA}(\mathcal{X},\mathcal{L}).
\]
where $E'$ is the exceptional divisor of $\widetilde{\mathrm{Bl}_{\pi^{-1}\mathfrak{a}}(X'\times \mathbb{A}^1)}$. In fact, $\mathrm{Bl}_{\pi^{-1}\mathfrak{a}}(X'\times \mathbb{A}^1)\to X\times \mathbb{A}^1$ factors through $\mathrm{Bl}_{\mathfrak{a}}(X\times \mathbb{A}^1)$ by the universal property of blowing up (cf. Proposition 7.14 of \cite{Ha}). Let $\pi':\widetilde{\mathrm{Bl}_{\pi^{-1}\mathfrak{a}}(X'\times \mathbb{A}^1)}\to \mathcal{X}$ and $E$ be the exceptional divisor of $\mathcal{X}$. Then $E'=\pi'^*E$ and the above equality follows from 
\[
\pi^*H_{\mathbb{A}^1}\cdot (\pi^*L_{\mathbb{A}^1}-E')^{n}-\frac{n\pi^*H\cdot (\pi^*L)^{n-1}}{(n+1)(\pi^*L)^n}(\pi^*L_{\mathbb{A}^1}-E')^{n+1}=l(H_{\mathbb{A}^1}\cdot (L_{\mathbb{A}^1}-E)^{n}-\frac{nH\cdot L^{n-1}}{(n+1)L^n}(L_{\mathbb{A}^1}-E)^{n+1})
\] (cf. \cite[Lemma 1.18]{B}) since $\pi'$ is also proper and a generically finite morphism whose degree is $l$. Let 
\[
\overline{\pi^{-1}\mathfrak{a}}=\sum_{k=0}^r\mathscr{I}_kt^k,
\]
where each $\mathscr{I}_k$ is an invertible ideal of $X'$. Now, we want to show the following.
\end{rem}

\begin{thm} \label{impl}
Notations as in Remark \ref{imp}. Then we have
\begin{align*}
V(\pi^*L)(\mathcal{J}^{\pi^*H})^\mathrm{NA}(\widetilde{\mathrm{Bl}_{\pi^{-1}\mathfrak{a}}(X'\times \mathbb{A}^1)},\pi^*L_{\mathbb{A}^1}-E')=&\frac{n\pi^*H\cdot \pi^*L^{n-1}}{(n+1)\pi^*L^n}\sum _{k=0}^{r-1}\sum_{j=0}^{n}e_{\pi^*L}(\mathscr{I}_{k}^{[j]},\mathscr{I}_{k+1}^{[n-j]}) \\
-&\sum _{k=0}^{r-1}\sum_{j=0}^{n-1}e_{\pi^*L_{|\pi^*H}}(\mathscr{I}_{k|\pi^*H}^{[j]},\mathscr{I}_{k+1|\pi^*H}^{[n-1-j]}).
\end{align*}
In other words, we can calculate $(\mathcal{J}^H)^\mathrm{NA}$-energy of $(\mathcal{X},\mathcal{L})=(\widetilde{\mathrm{Bl}_{\mathfrak{a}}(X\times \mathbb{A}^1)},L_{\mathbb{A}^1}-E)$ as follows:
 \begin{align*}
(\mathcal{J}^{H})^\mathrm{NA}(\mathcal{X},\mathcal{L})=&\frac{1}{lV(L)}\left(\frac{n}{n+1}\frac{H\cdot L^{n-1}}{L^n}\sum _{k=0}^{r-1}\sum_{j=0}^{n}e_{\pi^*L}(\mathscr{I}_{k}^{[j]},\mathscr{I}_{k+1}^{[n-j]})\right. \\
-&\left. \sum _{k=0}^{r-1}\sum_{j=0}^{n-1}e_{\pi^*L_{|\pi^*H}}(\mathscr{I}_{k|\pi^*H}^{[j]},\mathscr{I}_{k+1|\pi^*H}^{[n-1-j]})\right).
\end{align*}
Furthermore, there exists a sequence of $(\mathcal{J}^{\pi^*H})$-energy of semiample test configurations that converges to $(\mathcal{J}^{\pi^*H})^\mathrm{NA}(\widetilde{\mathrm{Bl}_{\pi^{-1}\mathfrak{a}}(X'\times \mathbb{A}^1)},\pi^*L_{\mathbb{A}^1}-E')$.
\end{thm}

\begin{proof}[Proof of Theorem \ref{impl}]
Let $(\mathcal{X}',\mathcal{L}')=(\widetilde{\mathrm{Bl}_{\pi^{-1}\mathfrak{a}}(X'\times \mathbb{A}^1)},\pi^*L_{\mathbb{A}^1}-E')$ be the normalized blow up. First, note $\mathcal{X}$ is the normalized blow up along $\mathfrak{a}$ and $\mathcal{X}'$ is the normalized blow up along $\pi^{-1}\mathfrak{a}$ in Remark \ref{imp}. They are normal test configurations of $X$ and $X'$ respectively. Now, we have the following commutative diagram.   $$
\begin{CD}
\mathcal{X}' @>{\pi'}>> \mathcal{X} \\
@V{\rho'}VV @V{\rho}VV \\
X'\times \mathbb{A}^1 @>{\pi}>> X\times\mathbb{A}^1
\end{CD}
$$ 
$\mathcal{X}'$ is isomorphic to the strict transformation of the normalization of $\mathcal{X}\times_XX'$. In fact, let the latter be $\overline{\mathcal{X}'}$ and we can find a birational morphism from $\mathcal{X}'$ to $\mathcal{X}\times_XX'$ induced by the diagram. The morphism induces a birational morphism $\alpha:\mathcal{X}'\to\overline{\mathcal{X}'}$. On the other hand, the inverse image of $\pi^{-1}\mathfrak{a}$ under $\rho'$ coincides the inverse image of $\rho^{-1}\mathfrak{a}$ under $\pi'$ on $\mathcal{X}\times_XX'$ and is hence an invertible sheaf on $\overline{\mathcal{X}'}$ since it is an inverse image of an invertible ideal sheaf and $\overline{\mathcal{X}'}$ is integral. Therefore, there exists a birational morphism $\beta:\overline{\mathcal{X}'}\to\mathcal{X}'$ by the universal property of blowing up and hence $\alpha$ and $\beta$ are the inverse of each other, i.e. there is a canonical isomorphism
\[
\mathcal{X}'\cong \overline{\mathcal{X}'}.
\]
 Let $F$ be a $\pi$-ample divisor on $X'$. Then $\rho'^*F$ is $\pi'$-ample. We can show that $\pi^*L+\epsilon F$ are ample for sufficiently small $\epsilon>0$. On the other hand, we can show that $\rho'^{-1}\pi^{-1}\mathfrak{a}=\pi'^*(\rho^{-1}\mathfrak{a})$ and hence 
 \[
 \rho'^*(\pi^*L_{\mathbb{A}^1}+sF_{\mathbb{A}^1})+s \pi'^*(\rho^{-1}\mathfrak{a})
 \]
  is $\mathbb{A}^1$-ample on $\mathcal{X}\times_XX'$ for sufficiently small positive rational number $s$. We can easily show that it holds also on $\mathcal{X}'$ by the isomorphism above. For sufficiently small positive rational number $\delta$, 
  \[
  \rho'^*\left(\pi^*L_{\mathbb{A}^1}+\frac{s\delta }{1+\delta}F_{\mathbb{A}^1}\right)+\frac{1+s\delta}{1+\delta}\pi'^*(\rho^{-1}\mathfrak{a})
  \]
   is also $\mathbb{A}^1$-ample since 
\begin{align*}
(1+\delta)\left(\rho'^*(\pi^*L_{\mathbb{A}^1}+\frac{s\delta }{1+\delta}F_{\mathbb{A}^1})+\frac{1+s\delta}{1+\delta}\pi'^*(\rho^{-1}\mathfrak{a})\right)=&\rho'^*(\pi^*L_{\mathbb{A}^1})+\pi'^*(\rho^{-1}\mathfrak{a})\\
+&\delta(\rho'^*(\pi^*L_{\mathbb{A}^1}+sF_{\mathbb{A}^1})+s \pi'^*(\rho^{-1}\mathfrak{a}))
\end{align*}
is the summation of a semiample divisor and ample one. Let $\epsilon=\frac{s\delta }{1+\delta}$ and note that $\lim_{\delta\to0}\epsilon=0$ and $s$ is independent of $\delta$. If we could prove the theorem for $(\mathcal{X}',\rho'^*(\pi^*L_{\mathbb{A}^1}+\frac{s\delta }{1+\delta}F_{\mathbb{A}^1})+\frac{1+s\delta}{1+\delta}\pi'^*(\rho^{-1}\mathfrak{a}))$, then we have
\begin{align}
&V(\pi^*L)(\mathcal{J}^{\pi^*H})^\mathrm{NA}(\mathcal{X}',\rho'^*(\pi^*L_{\mathbb{A}^1})+\pi'^*(\rho^{-1}\mathfrak{a}))\nonumber \\
=&\lim _{\delta \to 0}\left(\frac{1+s\delta}{1+\delta}\right)^{n+1}V\left(\frac{1+\delta}{1+s\delta}(\pi^*L+\epsilon F)\right)(\mathcal{J}^{\pi^*H})^\mathrm{NA}\left(\mathcal{X}',\frac{1+\delta}{1+s\delta}\rho'^*(\pi^*L_{\mathbb{A}^1}+\epsilon F_{\mathbb{A}^1})+\pi'^*(\rho^{-1}\mathfrak{a})\right)\label{eqn}  \\
=&\lim _{\delta \to 0}\left(\frac{1+s\delta}{1+\delta}\right)^{n+1}\left(\frac{n(\pi^*H)\cdot(\pi^*L+\epsilon F)^{n-1}}{(n+1)(\pi^*L+\epsilon F)^n} \sum _{k=0}^{r-1}\sum_{j=0}^{n}e_{\frac{1+\delta}{1+s\delta}(\pi^*L+\epsilon F)}(\mathscr{I}_{k}^{[j]},\mathscr{I}_{k+1}^{[n-j]})\right.  \nonumber \\
&\left.-\sum _{k=0}^{r-1}\sum_{j=0}^{n-1}e_{\frac{1+\delta}{1+s\delta}(\pi^*L+\epsilon F)_{|\pi^*H}}(\mathscr{I}_{k|\pi^*H}^{[j]},\mathscr{I}_{k+1|\pi^*H}^{[n-1-j]})\right) \nonumber \\
=&\frac{n\pi^*H\cdot \pi^*L^{n-1}}{(n+1)\pi^*L^n}\sum _{k=0}^{r-1}\sum_{j=0}^{n}e_{\pi^*L}(\mathscr{I}_{k}^{[j]},\mathscr{I}_{k+1}^{[n-j]})-\sum _{k=0}^{r-1}\sum_{j=0}^{n-1}e_{\pi^*L_{|\pi^*H}}(\mathscr{I}_{k|\pi^*H}^{[j]},\mathscr{I}_{k+1|\pi^*H}^{[n-1-j]}).\nonumber
\end{align}
 Therefore, we may assume that $X'\cong X$, $L$ is an ample $\mathbb{Q}$-line bundle and $\mathcal{X}$ is an ample test configuration. Take $\overline{\mathfrak{a}}$ the integral closure of $\mathfrak{a}$. This is an integrally closed flag ideal and $\mathcal{X}$ is the normalized blow up along this. Furthermore, $\overline{\mathfrak{a}}=\sum t^i\mathscr{I}_{D_i}$ satisfies (*) and $L$ is an ample $\mathbb{Q}$-line bundle. Therefore, we apply Theorem \ref{sl} and complete the proof. 
\end{proof}

 \section{J-Stability for Surfaces}\label{App1}
 
 In this section, we discuss about the applications of Theorems \ref{slope} and \ref{impl} to J-stability for surfaces.
 
 \subsection{J-stability for irreducible surfaces}\label{App11}
 
If $(X,L)$ is a polarized normal and irreducible surface with a $\mathbb{Q}$-Cartier divisor $H$, $(\mathcal{J}^H)^\mathrm{NA}$ is decomposed into non-negative values due to Theorems \ref{slope} and \ref{impl} as we will see in this section. First, we prepare the following:
 
 \begin{prop}\label{deft}
 Let $X$ be an integral surface, $L$ be a big and nef $\mathbb{Q}$-line bundle and $H$ be a $\mathbb{Q}$-Cartier divisor such that $L\cdot H\ge 0$ and 
 \[
 2\frac{L\cdot H}{L^2}L-H
 \]
 is nef. Let also $C$ be a pseudoeffective $\mathbb{Q}$-Cartier divisor such that $$(L-C)\cdot H\ge0.$$ Then,
 \[
2(C\cdot L)(L\cdot H)-(C\cdot H)(L^2)-(L\cdot H)(C^2)\ge0. 
 \]
 \end{prop}
 
 \begin{proof}
If $L\cdot H=0$, $-H$ is nef and
\[
2(C\cdot L)(L\cdot H)-(C\cdot H)(L^2)-(L\cdot H)(C^2)=-(C\cdot H)(L^2)\ge 0.
\]
Thus, we may assume that $L\cdot H>0$. Let $B=\frac{L\cdot H}{L^2}L-H$.  If $B\equiv 0$, where $\equiv$ means the numerical equivalence, let $F=C-\frac{C\cdot L}{L^2}L$. Then $F\cdot L=0$ and $F^2\le 0$ by the Hodge Index Theorem. Hence, we see the proposition is true since 
\begin{align*}
&2(C\cdot L)(L\cdot H)-(C\cdot H)(L^2)-(L\cdot H)(C^2)=(C\cdot L)(L\cdot H)-(L\cdot H)(C^2)\\
&\ge(C\cdot L)(L\cdot H)-(C\cdot L)^2 \frac{H\cdot L}{L^2} =\frac{(C\cdot L)(L\cdot H)}{L^2}(L\cdot (L-C)) \ge 0.
\end{align*}
 The last inequality holds by the assumption since $H\equiv \frac{L\cdot H}{L^2}L$. 
 
 Hence, we may assume that $B\not \equiv 0$. Note that $L\cdot B=0$ and let  
 \[
 E=C-\left( \frac{C\cdot L}{L^2}L+\frac{C\cdot B}{B^2}B\right).
 \]
 Then $E\cdot L=E\cdot B=0$. Now, $B^2<0$ and $E^2\le 0$ by Hodge Index Theorem. 
 Therefore, 
 \begin{align*}
 &2(C\cdot L)(L\cdot H)-(C\cdot H)(L^2)-(L\cdot H)(C^2)\\
 &\ge 2(C\cdot L)(L\cdot H)-(L^2)\left(\frac{(H\cdot L)(C\cdot L)}{L^2}-(C\cdot B)\right)-(L\cdot H)\left(\frac{(C\cdot L)^2}{L^2}+\frac{(C\cdot B)^2}{B^2}\right) \\
 &= (C\cdot L)(L\cdot H)+(L^2)(C\cdot B) -(L\cdot H)\left(\frac{(C\cdot L)^2}{L^2}+\frac{(C\cdot B)^2}{B^2}\right).
  \end{align*}
 Here, we used $-(L\cdot H)(E^2)\ge 0$. By assumption, we have that
  \[
 \frac{H\cdot L}{L^2}L+B=2\frac{L\cdot H}{L^2}L-H
 \]
  is nef and 
 \[
 \left(1-\frac{C\cdot L}{L^2}\right)L-\frac{C\cdot B}{B^2}B-E= L-C.
 \]
  Then, we have 
  \[
  B^2\ge -\frac{(H\cdot L)^2}{L^2}
  \]
   and 
  \[
  (H\cdot L)\left(1-\frac{C\cdot L}{L^2}\right)+(C\cdot B)\ge0
  \]
   by $(\frac{H\cdot L}{L^2}L+B)^2\ge0$ and $H\cdot(L-C)\ge0$ respectively, and 
  \begin{align*}
   &(C\cdot L)(L\cdot H)+(L^2)(C\cdot B)-(L\cdot H)\left(\frac{(C\cdot L)^2}{L^2}+\frac{(C\cdot B)^2}{B^2}\right)\\
   &\ge   (C\cdot L)(L\cdot H)+(L^2)(C\cdot B) -(L\cdot H)\left(\frac{(C\cdot L)^2}{L^2}-\frac{(C\cdot B)^2L^2}{(H\cdot L)^2}\right) \\
&=(L^2)(L\cdot H)\left(\frac{C\cdot L}{L^2}+\frac{C\cdot B}{H\cdot L}\right) \cdot \left(1-\frac{C\cdot L}{L^2}+\frac{C\cdot B}{H\cdot L}\right).
  \end{align*} 
 Here, we used $-\frac{(C\cdot B)^2}{B^2}\ge(C\cdot B)^2\frac{L^2}{(H\cdot L)^2}$. Then, we have 
 \[
 \frac{C\cdot L}{L^2}+\frac{C\cdot B}{H\cdot L}=(H\cdot L)^{-1}\left(2\frac{L\cdot H}{L^2}L-H\right)\cdot C\ge0.
 \]
  On the other hand,
  \[
  1-\frac{C\cdot L}{L^2}+\frac{C\cdot B}{H\cdot L}=(H\cdot L)^{-1}\left((H\cdot L)\left(1-\frac{C\cdot L}{L^2}\right)+(C\cdot B)\right)\ge0.
  \] We accomplish the proof.
   \end{proof}
 
 \begin{rem}\label{kh}
 The assumption that $L\cdot H\ge0$ and $(L-C)\cdot H\ge0$ is satisfied when one of $\{ L-C,H\}$ is nef and the other is pseudoeffective. Moreover, if $L\cdot H>0$ and $B\not\equiv0$, we have proved that
 \[
 2(C\cdot L)(L\cdot H)-(C\cdot H)(L^2)-(L\cdot H)(C^2)\ge L^2(L\cdot H)\left(\frac{C\cdot L}{L^2}+\frac{C\cdot B}{H\cdot L}\right)\cdot \left(1-\frac{C\cdot L}{L^2}+\frac{C\cdot B}{H\cdot L}\right) (\ge 0).
 \]
  in the above proof. We see that if the equality holds, we have 
  \[
  E\equiv 0
  \]
   and 
   \[
  \left(B^2 +\frac{(H\cdot L)^2}{L^2}\right)(C\cdot B)^2=0.
   \]
   We make use of these facts in the proof of Proposition \ref{p} below.
 \end{rem}
 
\begin{thm}\label{d2}
 Let $(X,L)$ be a polarized normal surface and $H$ be a pseudoeffective $\mathbb{Q}$-Cartier divisor on $X$ such that
 \[
 2\frac{L\cdot H}{L^2}L-H
 \]
 is nef. Then $(X,L)$ is $\mathrm{J}^H$-semistable.
\end{thm}

\begin{proof}
First, note that $(\mathcal{J}^H)^\mathrm{NA}$ is linear in $H$ and $H+a L$ is a big divisor such that
\[
2\frac{L\cdot (H+a L)}{L^2}L-(H+a L)=2\frac{L\cdot H}{L^2}L-H+a L
\] 
is ample for $a >0$. Therefore, we may assume that $H$ is big and $2\frac{L\cdot H}{L^2}L-H$ is ample. Furthermore, replace $H$ by $bH$ for some $b>0$ and we may also assume that $L\cdot H=L^2$. We must show that $(\mathcal{J}^H)^\mathrm{NA}(\phi)\ge0$ for any $\phi\in\mathcal{H}^\mathrm{NA}(L)$. By replacing $L$ by some multiple, we may assume as in Remark \ref{imp} that a flag ideal $\mathfrak{a}=\sum\mathfrak{a}_kt^k$ that satisfies $\mathfrak{a}_0\ne 0$ and induces a semiample test configuration $(\mathcal{X},\mathcal{L})=(\mathrm{Bl}_{\mathfrak{a}}(X\times \mathbb{A}^1),L_{\mathbb{A}^1}-E)$, where $E$ is the exceptional divisor, that is a representative of $\phi$. By Theorem \ref{impl}, there exists an alternation $\pi:X'\to X$ and
\[
\overline{\pi^{-1}\mathfrak{a}}=\mathscr{I}_{D_0}+\mathscr{I}_{D_1}t+\cdots +\mathscr{I}_{D_{r-1}}t^{r-1}+t^r,
\]
where each $D_i$ is an snc divisor, satisfies that the condition (*) and $(\pi^{-1}\mathfrak{a})\pi^*L_1$ is semiample. Next, we take $F,\epsilon,s$ and $\delta$ as in the proof of Theorem \ref{impl}. We can check that $\pi^*H+\epsilon F$ is big and $\pi^*L+\epsilon F$ and
\[
 2\frac{(\pi^*L+\epsilon F)\cdot (\pi^*H+\epsilon F)}{(\pi^*L+\epsilon F)^2}(\pi^*L+\epsilon F)-(\pi^*H+\epsilon F)=\pi^*(2L-H)+O(\epsilon)\pi^*L+(\epsilon +O(\epsilon^2))F
\]
are ample for sufficiently small $\epsilon$. Then if we could prove the theorem in the case where $\pi$ is isomorphic and $\mathfrak{a}$ induces an ample deformation to the normal cone, we have by the equation (\ref{eqn}) in the proof of Theorem \ref{impl}, 
\begin{align*}
(\mathcal{J}^{H})^\mathrm{NA}(\mathcal{X},\rho^*L_{\mathbb{A}^1}+\rho^{-1}(\mathfrak{a}))&=\lim _{\delta \to 0}(\mathcal{J}^{\pi^*H+\epsilon F})^\mathrm{NA}\left(\mathcal{X}',\rho'^*(\pi^*L_{\mathbb{A}^1}+\frac{s\delta }{1+\delta}F_{\mathbb{A}^1})+\frac{1+s\delta}{1+\delta}\pi'^*(\rho^{-1}\mathfrak{a})\right)\\
&\ge 0,
\end{align*}
where $\epsilon=\frac{s\delta }{1+\delta}$. Thus, we may also assume that $\pi$ is an isomorphism and $\mathfrak{a}$ induces an ample deformation to the normal cone. Then all the $L-D_i$ is nef by Corollary 5.8 of \cite{RT07}. Let us denote $D_r=0$. Now, it is easy to see by Theorem \ref{sl} that $V(L)(\mathcal{J}^H)^\mathrm{NA}(\mathcal{X},\mathcal{L})$ is as follows:
\begin{align*}
V(L)(\mathcal{J}^H)^\mathrm{NA}(\mathcal{X},\mathcal{L})&=p^*H\cdot \mathcal{L}^2-\frac{2H\cdot L}{3L^2}\mathcal{L}^3 \\
&=(3L^2)^{-1}\Biggl(\sum_{i=0}^{r-1}\biggl(  6((D_i+D_{i+1})\cdot L)(L\cdot H)\\
&-3((D_i+D_{i+1})\cdot H)(L^2)-2(L\cdot H)(D_i^2+D_i\cdot D_{i+1}+D_{i+1}^2)\biggr)\Biggr).
\end{align*}
We can prove that $$
6((D_i+D_{i+1})\cdot L)(L\cdot H)-3((D_i+D_{i+1})\cdot H)(L^2)-2(L\cdot H)(D_i^2+D_i\cdot D_{i+1}+D_{i+1}^2)\ge 0 $$
as follows. Note that $L-\frac{D_i+D_{i+1}}{2}=\frac{1}{2}((L-D_i)+(L-D_{i+1}))$ is nef. Then by Proposition \ref{deft}, 
\begin{align*}
 &4((D_i+D_{i+1})\cdot L)(L\cdot H)-2((D_i+D_{i+1})\cdot H)(L^2)-(L\cdot H)(D_i+D_{i+1})^2\\
 &=4\left(2\left(\frac{D_i+D_{i+1}}{2}\cdot L\right)(L\cdot H)-\left(\frac{D_i+D_{i+1}}{2}\cdot H\right)(L^2)-(L\cdot H)\left(\frac{D_i+D_{i+1}}{2}\right)^2\right)\\
 &\ge0.
\end{align*}
Therefore, we have that
\begin{align*}
 &6((D_i+D_{i+1})\cdot L)(L\cdot H)-3((D_i+D_{i+1})\cdot H)(L^2)
 -2(L\cdot H)(D_i^2+D_i\cdot D_{i+1}+D_{i+1}^2) \\
 &=4((D_i+D_{i+1})\cdot L)(L\cdot H)-2((D_i+D_{i+1})\cdot H)(L^2)-(L\cdot H)(D_i+D_{i+1})^2\\
 &+2(D_i\cdot L)(L\cdot H)-(D_i\cdot H)(L^2)-(L\cdot H)(D_i^2)\\
 &+2(D_{i+1}\cdot L)(L\cdot H)-(D_{i+1}\cdot H)(L^2)-(L\cdot H)(D_{i+1}^2)\\
 &\ge0.
\end{align*}
by Proposition \ref{deft}. Hence, 
\begin{align*}
(\mathcal{J}^H)^\mathrm{NA}(\mathcal{X},\mathcal{L})&=(3V(L)^2)^{-1}\Biggl(\sum_{i=0}^{r-1}\biggr(  6((D_i+D_{i+1})\cdot L)(L\cdot H)\\
&-3((D_i+D_{i+1})\cdot H)(L^2)-2(L\cdot H)(D_i^2+D_i\cdot D_{i+1}+D_{i+1}^2)\biggl)\Biggr).
\end{align*}
 is nonnegative.
\end{proof}

\begin{cor}\label{jstable}
For any polarized integral surface $(X,L)$ with a big (resp. pseudoeffective) $\mathbb{Q}$-Cartier divisor $H$, the following are equivalent. 
\begin{itemize}
\item[(1)] $(X,L)$ is uniformly $\mathrm{J}^H$-stable (resp. $\mathrm{J}^H$-semistable). In other words, there exists $\epsilon>0$ such that for any semiample test configuration $(\mathcal{X},\mathcal{L})$
\[
(\mathcal{J}^H)^\mathrm{NA}(\mathcal{X},\mathcal{L})\ge \epsilon J^\mathrm{NA}(\mathcal{X},\mathcal{L})\quad (\mathrm{resp}.\, \ge0).
\] 
\item[(2)] $(X,L)$ is uniformly slope $\mathrm{J}^H$-stable (resp. slope $\mathrm{J}^H$-semistable). In other words, there exists $\epsilon>0$ such that for any semiample deformation to the normal cone $(\mathcal{X},\mathcal{L})$ along any integral curve
\[
(\mathcal{J}^H)^\mathrm{NA}(\mathcal{X},\mathcal{L})\ge \epsilon J^\mathrm{NA}(\mathcal{X},\mathcal{L})\quad (\mathrm{resp}.\, \ge0).
\] 
\item[(3)] There exists $\epsilon>0$ such that for any integral curve $C$, 
\[
\left(2\frac{H\cdot L}{L^2}L-H\right)\cdot C\ge \epsilon L\cdot C\quad (\mathrm{resp}.\, \ge0).
\]
\end{itemize}
\end{cor}
\begin{proof}
$(1)\Rightarrow (2)$ is trivial. We have proved that $(3)\Rightarrow (1)$ when $2\frac{H\cdot L}{L^2}L-H$ is nef. If $(2\frac{H\cdot L}{L^2}L-H)\cdot C\ge \epsilon L\cdot C$, then
\[
\left(2\frac{(H-\epsilon L)\cdot L}{L^2}L-(H-\epsilon L)\right)\cdot C\ge 0.
\]
Therefore, take $\epsilon$ so small that $H-\epsilon L$ still be big and
\begin{align*}
(\mathcal{J}^H)^\mathrm{NA}&=(\mathcal{J}^{H-\epsilon L})^\mathrm{NA} +\epsilon (I^\mathrm{NA}-J^\mathrm{NA})\\
&\ge \epsilon (I^\mathrm{NA}-J^\mathrm{NA}).
\end{align*}
$(2)\Rightarrow (3)$ follows from the next generalized lemma of Lejmi-Sz\'{e}kelyhidi \cite{LS}. We reprove it for possibly singular varieties.
\end{proof}

\begin{lem}[cf. Proposition 13 \cite{LS}]\label{unst}
For any polarized $n$-dimensional variety $(X,L)$ with a $\mathbb{Q}$-Cartier divisor $H$, if there exists a $p$-dimensional subvariety $V$ such that $$\left(n\frac{H\cdot L^{n-1}}{L^n}L-pH\right)\cdot L^{p-1}\cdot V<0,$$ 
then $(X,L)$ is slope $\mathrm{J}^H$-unstable. Furthermore, if 
\[
\left((n\frac{H\cdot L^{n-1}}{L^n}-(n-p)\epsilon)L-pH\right)\cdot L^{p-1}\cdot V<0
\]
for any $\epsilon>0$, then $(X,L)$ is not uniformly slope $\mathrm{J}^H$-stable. 
\end{lem}

\begin{proof}
To prove the first assertion of this lemma, let $(\mathcal{X},L_{\mathbb{A}^1}-E)$ be a deformation to the normal cone of $V$ and $\hat{X}$ be the strict transformation of $X\times \{0\}$. $\hat{X}$ is isomorphic to the blow up of $X$ along $V$ and let $\pi:\hat{X}\to X$ be the canonical projection and $D$ be the exceptional divisor. Now, we can compute $(\mathcal{J}^H)^\mathrm{NA}(\mathcal{X},L_{\mathbb{A}^1}-\delta E)$. Note that $\pi_*(D^k)$ is a zero cycle when $k<n-p$. Then for $\delta>0$,
\begin{align*}
&V(L)(\mathcal{J}^H)^\mathrm{NA}(\mathcal{X},L_{\mathbb{A}^1}-\delta E)\\
&= \delta \pi^*H\cdot \left(\sum_{i=0}^{n-1} (\pi^*L-\delta D)^i\cdot \pi^*L^{n-i-1}\right) -\delta \frac{nH\cdot L^{n-1}}{(n+1)L^n}\sum_{i=0}^{n} (\pi^*L-\delta D)^i \cdot \pi^*L^{n-i} \\
&=\delta^{n-p+1} (-1)^{n-p-1} \frac{n!}{(n-p+1)!p!} D^{n-p}\cdot \left(n\frac{H\cdot L^{n-1}}{L^n}\pi^*L^p-p\pi^*H\cdot \pi^*L^{p-1}\right)+O(\delta^{n-p+2}). 
\end{align*}
It is well-known that $\pi_*((-1)^{n-p-1}D^{n-p})=e_VX[V]$ where $e_VX$ is the multiplicity of $X$ along $V$ (cf. \cite[\S 4.3]{F}) and $e_VX>0$. Therefore, we apply the projection formula to obtain
\begin{align*}
(-1)^{n-p-1}D^{n-p}\cdot \left(n\frac{H\cdot L^{n-1}}{L^n}\pi^*L^p-p\pi^*H\cdot \pi^*L^{p-1}\right)&=e_VX\left(n\frac{H\cdot L^{n-1}}{L^n}L^p-pH\cdot L^{p-1}\right)\cdot V\\
&<0
\end{align*}
by the assumption. Hence, for sufficiently small $\delta$, $(\mathcal{J}^H)^\mathrm{NA}(\mathcal{X},L-\delta E)<0$. The last assertion follows from the above argument and replacing $H$ by $H-\epsilon L$ for sufficiently small $\epsilon$.
\end{proof}

 \subsection{Remarks on reducible J-stable surfaces}\label{Demi}

\subsubsection{Stability of reducible schemes}

 J-stability for reducible schemes behave in a more complicated way as the following generalization of \cite[Theorem 7.1]{LW} shows:
 
 \begin{thm}\label{demichan}
Let $(X,\Delta;L)$ be an $n$-dimensional polarized reducible deminormal pair with a $\mathbb{Q}$-line bundle $H$ such that $X=\bigcup _{i=1}^l X_i$ be the irreducible decomposition. Let also $L_i=L|_{X_i}$ and $H_i=H|_{X_i}$. Suppose that 
\[
\frac{H\cdot L^{n-1}}{L^n}\ne\frac{H_i\cdot L_i^{n-1}}{L_i^n}
\]
for some $1\le i\le l$. Then $(X,L)$ is $\mathrm{J}^H$-unstable.

Furthermore, let $\nu:\coprod_{i=1}^l X_i^\nu\to X$ be the normalization, $\overline{L_i}=\nu^*L|_{X_i^\nu}$, $\overline{H_i}=\nu^*H|_{X^\nu_i}$ and $K_{(X_i^\nu,\Delta^\nu_i)}=\nu^*(K_{(X,\Delta)})|_{X^\nu_i}$. Suppose that 
\[
\frac{K_{(X,\Delta)}\cdot L^{n-1}}{L^n}\ne\frac{K_{(X_i^\nu,\Delta^\nu_i)}\cdot \overline{L_i}^{n-1}}{\overline{L_i}^n}
\]
for some $1\le i\le l$. Then $(X,\Delta;L)$ is K-unstable.
\end{thm}

\begin{rem}
We point out a small error in \cite{Fu}. Notations as in loc.cit. In the proof of \cite[Proposition 6.1]{Fu}, it is asserted that $\psi_L^{\cdot n-1}\cdot (\psi_N-\mu_N(L)\psi_L)\equiv 0$. To be precise, $\psi_L^{\cdot n-1}\cdot (\psi_N-\mu_N(L)\psi_L)\not\equiv 0$ if there exists an irreducible component $(X_i,L_i)$ of $(X,L)$ such that $\mu_{N|_{X_i}}(L_i)\ne \mu_N(L)$. Theorem 1.1 of loc.cit does not hold for reducible deminormal schemes.
\end{rem}

See the definition of deminormal schemes at \S \ref{Notat}. To prove Theorem \ref{demichan}, recall \cite[\S 9.3]{BHJ} briefly.
 
 \begin{de}[Definition 9.10 of \cite{BHJ}]
 Let $(X,L)$ be a polarized normal and irreducible variety, $\phi\in\mathcal{H}^\mathrm{NA}(L)$ be a positive metric and $(\mathcal{X},\mathcal{L})$ be a normal, semiample test configuration representing $\phi$. Suppose that $\mathcal{X}$ dominates $X_{\mathbb{A}^1}$. For each irreducible component $E$ of $\mathcal{X}_0$, let $Z_E\subset X$ be the closure of the center of $v_E$ on $X$, and let $r_E\coloneq \mathrm{codim}_XZ_E$. Then the canonical morphism $\mathcal{X}\to X\times\mathbb{A}^1$ maps $E$ onto $Z_E\times\{0\}$. Let $F_E$ be the general fibre of the induced morphism $E\to Z_E$. Then, set the local degree $\mathrm{deg}_E(\phi)$ as 
 \[
 \mathrm{deg}_E(\phi)\coloneq(F_E\cdot \mathcal{L}^{r_E}).
 \]
 It is independent of the choice of representatives of $\phi$.
 \end{de}
 
 Since $\mathcal{L}$ is semiample on $E\subset \mathcal{X}_0$, $\mathrm{deg}_E(\phi)\ge 0$. Furthermore, $\mathrm{deg}_E(\phi)>0$ iff $E$ is not contracted on the ample model of $(\mathcal{X},\mathcal{L})$. Here, we have the following:
 
 \begin{prop}[Lemma 9.11, Proposition 9.12 of \cite{BHJ}]
 Notations as above. Given $0\le j\le n$, $\mathbb{Q}$-line bundles $M_1,\cdots ,M_{n-j}$ on $X$ and a Cartier divisor $D$ on $\mathcal{X}$ such that $D=\mathcal{L}-L_{\mathbb{A}^1}$. Then, we have, for $0<\epsilon \ll 1$ rational:
 \begin{eqnarray*}
 &\left(E\cdot (L_{\mathbb{A}^1}-\epsilon D)^j\cdot (M_{1,\mathbb{A}^1}\cdot \cdots \cdot M_{n-j,\mathbb{A}^1})\right)\\
&= \left\{
\begin{split}
&\epsilon^{r_E}\left[\mathrm{deg}_E(\phi)\binom{j}{r_E}(Z_E\cdot L^{j-r_E}\cdot M_1\cdot\cdots\cdot M_{n-j})\right] +O(\epsilon^{r_E+1}) &\quad &\mathrm{for}\, j\ge r_E \\
&0 &\quad & \mathrm{for}\, j< r_E.
\end{split} \right. 
 \end{eqnarray*}
 \end{prop}
 The proof of this proposition is straightforward.
 
Recall the following facts of \cite[\S 7]{BHJ}. Let $H$ be a $\mathbb{Q}$-line bundle on $X$ and $(\mathcal{X},L_{\mathbb{A}^1}-\epsilon D)$ be a semiample test configuration, where $D=\sum D_i$. Then
 \begin{align*}
 (L_{\mathbb{P}^1}-\epsilon D)^{n+1}&=-\sum_{j=0}^n\left(\epsilon\sum_i D_i\cdot (L_{\mathbb{P}^1}-\epsilon D)^j\cdot L_{\mathbb{P}^1}^{n-j}\right)\\
 H_{\mathbb{P}^1}\cdot(L_{\mathbb{P}^1}-\epsilon D)^{n}&=-\sum_{j=0}^{n-1}\left(\epsilon\sum_i D_i\cdot (L_{\mathbb{P}^1}-\epsilon D)^j\cdot L_{\mathbb{P}^1}^{n-1-j}\right)\cdot H_{\mathbb{P}^1}\\
 H^\mathrm{NA}_B(\mathcal{X},L_{\mathbb{P}^1}-\epsilon D)&=V(L)^{-1}\sum_EA_{(X,B)}(v_E)b_E(E\cdot (L_{\mathbb{P}^1}-\epsilon D)^n),
 \end{align*}
where $E$ runs over the irreducible components of $\mathcal{X}_0$ and $b_E=\mathrm{ord}_E(\mathcal{X}_0)$, where $\mathcal{X}_0$ is the central fibre. As in \cite[Proposition 9.12]{BHJ}, if $r=\min\{ r_{D_i}\}$, then
\begin{align*}
 (L_{\mathbb{P}^1}-\epsilon D)^{n+1}&=O(\epsilon^{r+1});\\
 H_{\mathbb{P}^1}\cdot(L_{\mathbb{P}^1}-\epsilon D)^{n}&=O(\epsilon^{r+1}).
\end{align*}
On the other hand, we remark that
\[
V(L)^{-1}A_{(X,B)}(v_E)b_E(E\cdot (L_{\mathbb{P}^1}-\epsilon D)^n)=O(\epsilon^{r_E}).
\]

Finally, we recall the following notion to prove Theorem \ref{demichan}:

\begin{de}[Rees Valuation cf. Definition 1.9 of \cite{BHJ}]
Let $X$ be a normal variety and $Z$ be its closed subscheme. Let also $\hat{X}\to X$ be the normalized blow up along $Z$ with the exceptional divisor $E$. A divisorial valuation $v$ on $X$ is a {\it Rees valuation} with respect to $Z$ if there exists a prime divisor $F$ of $\hat{X}$ contained in $E$ such that 
\[
v=\frac{\mathrm{ord}_F}{\mathrm{ord}_F(E)}.
\]
We denote the set of all Rees valuations by $\mathrm{Rees}(Z)$. Here, we remark that thanks to \cite[Theorem 4.8]{BHJ}, the following holds. Let $\mathcal{X}$ be the normalization of the deformation to the normal cone of $X$ along $Z$. Then, $\mathrm{Rees}(Z)$ coincides with the set of the valuations $v_E$, where $E$ runs over the non-trivial irreducible components of $\mathcal{X}_0$. For the definition of $v_E$, see Definition \ref{dfn}.
\end{de}

\begin{proof}[Proof of Theorem \ref{demichan}]
Note that 
\[
\sum_{i=1}^l\left(H_i\cdot L_i^{n-1}-\frac{H\cdot L^{n-1}}{L^n}L_i^n\right)=0.
\]
Thus, we may assume that  
\[
\left(H_1\cdot L_1^{n-1}-\frac{H\cdot L^{n-1}}{L^n}L_1^n\right)<0.
\]
Let $Z=X_1\cap\bigcup_{i\ge 2} X_i$ and $\mathcal{X}$ be the blow up of $X_{\mathbb{A}^1}$ along $Z\times \{0\}$ with the exceptional divisor $E$. Note that $Z\times \{0\}=X_1\times\{0\} \cap\bigcup_{i\ge 2} X_i\times\mathbb{A}^1$ and hence the strict transformation $F$ of $X_1\times\{0\}$ is disjoint from the strict transformations of $X_i\times\mathbb{A}^1$ for $i\ge 2$. On the other hand, $F$ is a Cartier divisor on the strict transformation of $X_1\times\mathbb{A}^1$. Hence, $F$ is a Cartier divisor on $\mathcal{X}$. Choose $\eta>0$ such that $-E+\eta F$ is $X_{\mathbb{A}^1}$-ample. For sufficiently small $\epsilon>0$, we consider an ample test configuration
\[
(\mathcal{X},L_{\mathbb{A}^1}-\epsilon(E-\eta F))
\]
over $(X,L)$.

Take the normalization $(\overline{\mathcal{X}}=\coprod \mathcal{X}^\nu_i,\overline{L_{\mathbb{A}^1}-\epsilon(E-\eta F)})$ and $(\overline{X}=\coprod X^\nu_i,\overline{L})$ of $(\mathcal{X},L_{\mathbb{A}^1}-\epsilon(E-\eta F))$ and $(X,L)$ respectively. Then 
\begin{align*}
(\mathcal{J}^H)^\mathrm{NA}(\mathcal{X},L_{\mathbb{A}^1}-\epsilon(E-\eta F))&=(\mathcal{J}^{(\prod \overline{H}_i)})^\mathrm{NA}\left(\coprod\mathcal{X}^\nu_i,\overline{L_{\mathbb{A}^1}-\epsilon(E-\eta F)}\right)\\
&=V(L)^{-1}\sum_{i=1}^l\left(\overline{H}_{i,\mathbb{P}^1}\cdot (\mathcal{L}^\nu_i)^n-\frac{nH\cdot L^{n-1}}{(n+1)L^n}(\mathcal{L}^\nu_i)^{n+1}\right),
\end{align*}
where $\mathcal{L}^\nu_i=\overline{L_{\mathbb{A}^1}-\epsilon(E-\eta F)}|_{\mathcal{X}^\nu_i}$. Note that the image of each irreducible component of $E$ on $X\times\{0\}$ has codimension $\ge 1$. Then we apply \cite[Lemma 9.11]{BHJ} to each $\mathcal{X}^\nu_i$ and obtain
\begin{align*}
(\mathcal{J}^H)^\mathrm{NA}(\mathcal{X},L_{\mathbb{A}^1}-\epsilon(E-\eta F))&=(\mathcal{J}^{(\prod \overline{H}_i)})^\mathrm{NA}(\mathcal{X}^\nu,\overline{L}_{\mathbb{A}^1}+\epsilon\eta \mathcal{X}^\nu_{1,0})+O(\epsilon^2)\\
&=\epsilon\eta \frac{n}{V(L)}\left(H_1\cdot L_1^{n-1}-\frac{H\cdot L^{n-1}}{L^n}L_1^n\right)+O(\epsilon^2),
\end{align*}
where $\mathcal{X}^\nu_{1,0}$ is the central fibre of $\mathcal{X}^\nu_1$. Since $(H_1\cdot L_1^{n-1}-\frac{H\cdot L^{n-1}}{L^n}L_1^n)<0$, we have $(\mathcal{J}^H)^\mathrm{NA}(\mathcal{X},L_{\mathbb{A}^1}-\epsilon(E-\eta F))<0$ for $0<\epsilon\ll 1$.

For K-unstability, we may assume that $(X,L)$ is slc and then the second assertion immediately follows from \cite[Lemma 9.11]{BHJ}. Indeed, by the above argument and replacing $H$ by $K_{(X,\Delta)}$,
\[
(\mathcal{J}^{K_{(X,\Delta)}})^\mathrm{NA}(\mathcal{X},L_{\mathbb{A}^1}-\epsilon(E-\eta F))=\epsilon\eta \frac{n}{V(L)}\left(K_{(X^\nu_1,\Delta^\nu_1)}\cdot \overline{L}_1^{n-1}-\frac{K_{(X,\Delta)}\cdot L^{n-1}}{L^n}\overline{L}_1^n\right)+O(\epsilon^2).
\]
Assume that 
\[
\left(K_{(X^\nu_1,\Delta^\nu_1)}\cdot \overline{L}_1^{n-1}-\frac{K_{(X,\Delta)}\cdot L^{n-1}}{L^n}\overline{L}_1^n\right)<0
\]
 and then we have only to show
\[
V(L)H^\mathrm{NA}_{\Delta}(\widetilde{\mathcal{X}},\widetilde{\mathcal{L}})=\sum_{i=1}^lV(\overline{L_i})H^\mathrm{NA}_{\Delta^\nu_i}(\mathcal{X}^\nu_i,\mathcal{L}^\nu_i)=O(\epsilon^2),
\]
where $(\widetilde{\mathcal{X}},\widetilde{\mathcal{L}})$ is the partial normalization (see Definition \ref{genki}). The proof of the first equality is straightforward. To prove the second equality, note that $\mathcal{X}^\nu_i$ is the normalization of the deformation of $X_i^\nu$ to the normal cone of $Z_i$, where $Z_i$ is the inverse image of $Z$ under $X_i^\nu\to X$. Hence, we have only to prove that $\mathrm{Rees}(Z_i)$ has valuations whose log discrepancies are $0$ or whose center on $X\times \{0\}$ has codimension $r\ge 2$ thanks to \cite[Theorem 4.8]{BHJ}. Indeed, since the points of codimension 1 in $X_i\cap X_j$ for $i\ne j$ are nodes, $Z_i$ is the union of the restrictions of irreducible components of the conductor subscheme to $X^\nu_i$ and closed subschemes of $X^\nu_i$ of codimension more than 1 for each $i$. Note that the valuations of prime divisors of the conductor subscheme have the log discrepancies $0$. Therefore, by \cite[Proposition 9.12]{BHJ} we have
\[
H^\mathrm{NA}_{\Delta^\nu_i}(\mathcal{X}^\nu_i,\mathcal{L}_i^\nu)=O(\epsilon^2).
\]
We complete the proof.
\end{proof}

\begin{de}
Let $V$ be a reducible scheme and $V=\bigcup V_i$ is the irreducible decomposition. A $\mathbb{Q}$-line bundle $L$ on $V$ is ample (resp. nef, big, pseudoeffective) if so is each $L|_{V_i}$.
\end{de}

Reflecting Theorem \ref{demichan}, we have only to consider about polarized deminormal schemes $(X,L)$ with $\mathbb{Q}$-line bundle on $X$ such that all irreducible components of $(X,L)$ have the same average scalar curvature with respect to $H$ (see Definition \ref{denden}). For such $(X,L)$, $\mathrm{J}^H$-stability behave similarly as for normal varieties. Indeed, we have the following decomposition formula:
\begin{lem}\label{demibunkai}
Let $(X,L)$ be a polarized deminormal scheme with a $\mathbb{Q}$-line bundle $H$ such that all irreducible components $\{X_i\}_{i=1}^l$ of $X$ have the same average scalar curvature with respect to $H$. If the restrictions of $L$ and $H$ to each $X_i$ are $L_i$ and $H_i$ respectively, then for any semiample test configuration $(\mathcal{X},\mathcal{L})$,
\begin{align*}
V(L)(\mathcal{J}^H)^\mathrm{NA}(\mathcal{X},\mathcal{L})&=\sum_{i=1}^l V(L_i)(\mathcal{J}^{H_i})^\mathrm{NA}(\mathcal{X}_i,\mathcal{L}_i);\\
V(L)J^\mathrm{NA}(\mathcal{X},\mathcal{L})&=\sum_{i=1}^l V(L_i)J^\mathrm{NA}(\mathcal{X}_i,\mathcal{L}_i);\\
V(L)I^\mathrm{NA}(\mathcal{X},\mathcal{L})&=\sum_{i=1}^l V(L_i)I^\mathrm{NA}(\mathcal{X}_i,\mathcal{L}_i),
\end{align*}
where $\mathcal{X}_i$ is the strict transformation of $X_i\times\mathbb{A}^1$ and $\mathcal{L}_i$ is the restriction of $\mathcal{L}$ to $\mathcal{X}_i$. Moreover, let $(X,B)$ be a deminormal pair, $\nu:(X^\nu=\coprod_{i=1}^l X_i^\nu,L^\nu) \to (X,L)$ be the normalization and $\overline{D}$ be the conductor subscheme of $X^\nu$. Suppose that $\overline{B}$ is the divisorial part of $\nu^{-1}B$ and $X_i^\nu$ has the same average scalar curvature with respect to $K_{(X^\nu,\overline{B}+\overline{D})}$. If $(\mathcal{X}^\nu,\mathcal{L}^\nu)$ is the normalization of $(\mathcal{X},\mathcal{L})$, then
\begin{align*}
V(L)H_B^\mathrm{NA}(\mathcal{X},\mathcal{L})&=\sum_{i=1}^l V(L^\nu_i)H_{\overline{B}|_{X^\nu_i}+\overline{D}|_{X^\nu_i}}^\mathrm{NA}(\mathcal{X}^\nu_i,\mathcal{L}^\nu_i);\\
V(L)M_B^\mathrm{NA}(\mathcal{X},\mathcal{L})&=\sum_{i=1}^l V(L^\nu_i)M_{\overline{B}|_{X^\nu_i}+\overline{D}|_{X^\nu_i}}^\mathrm{NA}(\mathcal{X}^\nu_i,\mathcal{L}^\nu_i),
\end{align*}
where $\mathcal{X}^\nu_i$ is the strict transformation of $X^\nu_i\times\mathbb{A}^1$, $L^\nu_i$ is the restriction of $L^\nu$ to each $X^\nu_i$ and $\mathcal{L}^\nu_i$ is the restriction of $\mathcal{L}^\nu$ to $\mathcal{X}_i$.
\end{lem}
The proof is straightforward. Therefore we have the following generalization of Corollary \ref{jstable}:
\begin{cor}\label{ms}
For any polarized deminormal surface $(X,L)$ with a big (resp. pseudoeffective) $\mathbb{Q}$-line bundle $H$ such that all irreducible components of $(X,L)$ have the same average scalar curvature with respect to $H$ (see Remark \ref{denden}), the following are equivalent. 
\begin{itemize}
\item[(1)] $(X,L)$ is uniformly $\mathrm{J}^H$-stable (resp. $\mathrm{J}^H$-semistable). In other words, there exists $\epsilon>0$ such that for any semiample test configuration $(\mathcal{X},\mathcal{L})$
\[
(\mathcal{J}^H)^\mathrm{NA}(\mathcal{X},\mathcal{L})\ge \epsilon J^\mathrm{NA}(\mathcal{X},\mathcal{L})\quad (\mathrm{resp}.\, \ge0).
\] 
\item[(2)] $(X,L)$ is uniformly slope $\mathrm{J}^H$-stable (resp. slope $\mathrm{J}^H$-semistable). In other words, there exists $\epsilon>0$ such that for any semiample deformation to the normal cone $(\mathcal{X},\mathcal{L})$ along any integral curve
\[
(\mathcal{J}^H)^\mathrm{NA}(\mathcal{X},\mathcal{L})\ge \epsilon J^\mathrm{NA}(\mathcal{X},\mathcal{L})\quad (\mathrm{resp}.\, \ge0).
\] 
\item[(3)] There exists $\epsilon>0$ such that for any integral curve $C$, 
\[
\left(2\frac{H\cdot L}{L^2}L-H\right)\cdot C\ge \epsilon L\cdot C\quad (\mathrm{resp}.\, \ge0).
\]
\end{itemize}
\end{cor}

\begin{proof}
$(1)\Rightarrow (2)$ is trivial. On the other hand, $(3)\Rightarrow (1)$ holds by the following argument. For any semiample test configuration $(\mathcal{X},\mathcal{L})$, take the normalization $(\coprod\mathcal{X}_i,\prod\mathcal{L}_i)$, where $\mathcal{X}_i$ is the strict transform of each irreducible component $X_i\times\mathbb{A}^1$ of $X\times\mathbb{A}^1$ and $\mathcal{L}_i=\mathcal{L}|_{\mathcal{X}_i}$. Then we have
\[
V(L)(\mathcal{J}^H)^\mathrm{NA}(\mathcal{X},\mathcal{L})=\sum_iV(L_i)(\mathcal{J}^{H_i})^\mathrm{NA}(\mathcal{X}_i,\mathcal{L}_i)
\]
by Lemma \ref{demibunkai}, where $L_i=L|_{X_i}$ and $H_i=H|_{X_i}$. Similarly, we have 
\[
V(L)J^\mathrm{NA}(\mathcal{X},\mathcal{L})=\sum_iV(L_i)J^\mathrm{NA}(\mathcal{X}_i,\mathcal{L}_i).
\]
 Therefore, the assertion follows from Corollary \ref{jstable}.
 
 Finally, to prove $(2)\Rightarrow (3)$, we have to show the following generalization of Lemma \ref{unst}. Indeed, the corollary follows immediately from Lemma \ref{unst2} below.
\end{proof}

\begin{lem}\label{unst2}
For any polarized $n$-dimensional deminormal scheme $(X,L)$ with a $\mathbb{Q}$-line bundle $H$ such that all irreducible components of $(X,L)$ have the same average scalar curvature, if there exists a $p$-dimensional subvariety $V$ such that $$\left(n\frac{H\cdot L^{n-1}}{L^n}L-pH\right)\cdot L^{p-1}\cdot V<0,$$ 
then $(X,L)$ is slope $\mathrm{J}^H$-unstable. Furthermore, if 
\[
\left((n\frac{H\cdot L^{n-1}}{L^n}-(n-p)\epsilon)L-pH\right)\cdot L^{p-1}\cdot V<0
\]
for any $\epsilon>0$, then $(X,L)$ is not uniformly slope $\mathrm{J}^H$-stable. 
\end{lem}

\begin{proof}
Recall that Lemma \ref{unst} holds for nonnormal varieties and note that we have only to prove the lemma when $\epsilon=0$. Take $V$ as in the assumption. Let $X_i$'s be irreducible components of $X$. If $\mathcal{X}$ is the blow up of $X\times \mathbb{A}^1$ along $V\times\{0\}$, the strict transformation $\mathcal{X}_i$ of each $X_i\times\mathbb{A}^1$ is the blow up along $V_i\times\{0\}$, where $V_i$ is the restriction of $V$ to $X_i$. Note that if $V\subset X_i$,
\[
\left(n\frac{H_i\cdot L_i^{n-1}}{L_i^n}L_i-pH_i\right)\cdot L_i^{p-1}\cdot V_i=\left(n\frac{H\cdot L^{n-1}}{L^n}L-pH\right)\cdot L^{p-1}\cdot V<0,
\]
 where $L_i$ and $H_i$ are the restrictions of $L$ and $H$ to $X_i$ respectively. Moreover, if $E$ is the exceptional divisor of $\mathcal{X}$, then the restriction $E|_{\mathcal{X}_i}$ to $\mathcal{X}_i$ is the one of $\mathcal{X}_i$. Hence, it follows from the computation of the proof of Lemma \ref{unst} that for sufficiently small $\delta>0$,
\[
V(L_i)(\mathcal{J}^{H_i})^\mathrm{NA}(\mathcal{X}_i,L_{i,\mathbb{A}^1}-\delta E|_{\mathcal{X}_i})=\frac{n!\, e_{V_i}X_i\,\delta^{n-p+1} }{(n-p+1)!p!}\left(n\frac{H\cdot L^{n-1}}{L^n}L_i^p-pH_i\cdot L_i^{p-1}\right)\cdot V_i+O(\delta^{n-p+2})
\]
for $X_i$ such that $V\subset X_i$ and 
\[
V(L_j)(\mathcal{J}^{H_j})^\mathrm{NA}(\mathcal{X}_j,L_{j,\mathbb{A}^1}-\delta E|_{\mathcal{X}_j})=O(\delta^{n-p+2})
\]
for $X_j$ such that $V\not\subset X_j$. We complete the proof.
\end{proof}

\subsubsection{Counter examples of Proposition \ref{deft} for reducible schemes}
In fact, it is easy to see that Proposition \ref{deft} also hold when $(X,L)$ is deminormal surface with a pseudoeffective $\mathbb{Q}$-line bundle $H$ such that all irreducible components of $(X,L)$ have the same average scalar curvature. Indeed, we have only to take the normalization and apply Proposition \ref{deft} to each irreducible component. However, Hodge index theorem does not hold even for connected deminormal surfaces in general as the following example shows. We prove that Proposition \ref{deft} does not hold in general either.

\begin{ex}
Let $X$ be a projective, smooth and irreducible surface and $P$ be a closed point of $X$. Let $\pi:Y\to X$ be the blow up a closed point $P$ of $X$. Then glue two copies of $Y$ by \cite[Theorem 33]{K}, say $Y_1$ and $Y_2$, along the exceptional curve $E$, as a deminormal surface 
\[
Z=Y_1\cup_EY_2.
\] 
 The construction as follows:
\begin{lem}\label{glue}
Suppose that $V$ is a smooth irreducible variety and $F$ is an snc divisor on $V$. Let $V_1$ and $V_2$ be two copies of $V$. Then, there exists an slc scheme $W\cong V_1\cup_FV_2$.
\end{lem}
\begin{proof}
Take an ample $\mathbb{Q}$-divisor $\Delta$ on $V$ such that $(V,F+\Delta)$ is an lc pair and $K_V+\Delta+F$ is ample. For $i=1,2$, let $(V_i,F_i+\Delta_i)$ be copies of $(V,F+\Delta)$. We can easily take the canonical involution $\tau$ of $(F_1+F_2,\mathrm{Diff}_{F_1}\Delta_1+\mathrm{Diff}_{F_2}\Delta_2)$ by changing indices. By Koll{\' a}r's gluing theorem (Theorem 23 of \cite{K}), there exists an slc pair $(W,\overline{\Delta})$. In fact, $W=V_1\cup_FV_2$ set-theoretically. Finally, $W$ is slc since $\overline{\Delta}$ is effective.
\end{proof}

\begin{rem}\label{gluten}
We can also glue $Y_1$ and $Y_2$ directly as follows. Take an affine open neighborhood $U\cong \mathrm{Spec}\,A$ of $X$ contains $P$. Let $x,y\in A$ form a system of parameter around $P$. Then $$U'=\mathrm{Proj}\left((A\otimes A/(x_1x_2,y_1y_2,x_1y_2,y_1x_2))[u,v]/(x_1v-y_1u,x_2v-y_2u)\right)$$ is \'{e}tale equivalent to
$$\mathrm{Proj}\left((k[T_1,W_1]\otimes k[T_2,W_2]/(T_1T_2,W_1W_2,T_1W_2,W_1T_2))[u,v]/(T_1v-W_1u,T_2v-W_2u)\right),$$
 where $x_1=x\otimes 1,x_2=1\otimes x$. Hence, $U'$ is deminormal. We can construct $Z$ by gluing together $U'$ and $(Y_1-E_1)\sqcup (Y_2-E_2)$. Finally, we can easily check that $Z$, which we have constructed just now, coincides with the deminormal surface we obtain by Lemma \ref{glue} thanks to \cite[Proposition 5.3]{Ko}.
 \end{rem}

We can see that the following lemma.
\begin{lem}\label{grugru}
Suppose that $M_1=\pi^*F_1-\epsilon E$ and $M_2=\pi^*F_2-\delta E$ are line bundles on $Y_1$ and $Y_2$ respectively, where $F_1$ and $F_2$ are line bundles on $X$. If $\epsilon = \delta$, there exists $M$ a line bundle on $Z$ whose restriction to $Y_1$ and $Y_2$ are $M_1$ and $M_2$ respectively. 
\end{lem}
\begin{proof}
First, we prove the lemma in the case where $\epsilon =\delta =0$. We may assume that $F_1$ and $F_2$ are ample and there are general global sections $s_1$ and $s_2$ whose zero loci do not pass through $P$. If $D_i=(s_i)_0$, $\pi^*D_i$ is disjoint from $E_i$ and we can get $D$ by gluing $\pi^*D_1$ and $\pi^*D_2$. Hence, $D$ is a Cartier divisor on $Z$ and we obtain $F=\mathcal{O}_Z(D)$ such that $F_{|Z_i}=F_i $ for $i=1,2$. Next, we consider an affine neighborfood of $P$ of $X$, $\mathrm{Spec}(A)$. Let $x,y\in A$ form a system of parameter around $P$. As we saw in Remark \ref{gluten}, $Z$ is $$\mathrm{Proj}\left((A\otimes A/(x_1x_2,y_1y_2,x_1y_2,y_1x_2))[u,v]/(x_1v-y_1u,x_2v-y_2u)\right)$$ locally around $E$ where $x_1=x\otimes 1,x_2=1\otimes x$. On $(\mathrm{Spec}A\cup_P\mathrm{Spec}A)\setminus P$, $Z$ is isomorphic to this. $\mathcal{O}_\mathrm{Proj}(1)$ can be extended globally such that it is trivial on $Z\setminus E$ and the restriction to $Y_i$ is $\mathcal{O}(-E_i)$.
\end{proof}
  Then, Hodge index theorem does not hold on $Z$, which we have constructed. Take an ample line bundle $L$ on $X$, let two line bundles $L_1,L_2$ are two copies of $\pi^*L$ on $Y_1,Y_2$. We construct line bundles $M_1$ and $M_2$ on $Z$ by gluing $L_1$ and $\mathcal{O}_{Y_2}$ together, and $\mathcal{O}_{Y_1}$ and $L_2$ together respectively. Then $M_1^2,M_2^2>0$ but $M_1\cdot M_2=0$.
  
  Here by using $Z$ above, we construct a counter-example for Theorem \ref{d2} when $X$ is reducible and all irreducible components of $X$ do not have the same average scalar curvature. Due to Lemma \ref{grugru}, we can construct two line bundles $\tilde{L}$ and $\tilde{H}$ by gluing
  \[
  \delta L_1-\epsilon E_1, L_2-\epsilon E_2
  \]
  and 
  \[
  \delta L_1-\frac{1}{3}\epsilon E_1, \left(\frac{1}{2}+\eta\right)L_2-\frac{1}{3}\epsilon E_2
  \]
  respectively for $0<\epsilon\ll\delta\ll\eta\ll 1$. Then
  \[
  \frac{\tilde{L}\cdot \tilde{H}}{\tilde{L}^2}=\frac{1}{2}+\eta+O(\epsilon,\delta).
  \]
  Therefore,
  \begin{align*}
  \nu^*\left(2\frac{\tilde{L}\cdot \tilde{H}}{\tilde{L}^2}\tilde{L}-\tilde{H}\right)&=(1+2\eta+O(\epsilon,\delta))( \delta L_1-\epsilon E_1\times L_2-\epsilon E_2)\\
  &-\left(\delta L_1-\frac{1}{3}\epsilon E_1\times \left(\frac{1}{2}+\eta\right)L_2-\frac{1}{3}\epsilon E_2\right)\\
  &=\Biggl(\delta(2\eta+O(\epsilon,\delta)) L_1-\epsilon\left(\frac{2}{3}+ O(\epsilon,\delta,\eta)\right)E_1\\
  &\times \left(\frac{1}{2}+\eta+O(\epsilon,\delta)\right)L_2-\epsilon\left(\frac{2}{3}+ O(\epsilon,\delta,\eta)\right) E_2\Biggr)
  \end{align*}
  is ample for $0<\epsilon\ll\eta\delta\ll\delta\ll\eta\ll 1$, where $\nu$ is the normalization. Hence, $2\frac{\tilde{L}\cdot \tilde{H}}{\tilde{L}^2}\tilde{L}-\tilde{H}$, $\tilde{L}$ and $\tilde{H}$ are also ample. We may assume that $\eta\delta L-\epsilon E$ is ample and $\tilde{L}-C$ is ample, where
  \[
  \nu^*C=(\delta(1-\eta)L_1\times\mathcal{O}_{Y_2}).
  \]
  Then
  \[
  \left(2\frac{\tilde{L}\cdot \tilde{H}}{\tilde{L}^2}\tilde{L}-\tilde{H}\right)\cdot C-\frac{\tilde{L}\cdot \tilde{H}}{\tilde{L}^2}C^2=\frac{1}{2}\delta^2(-1+O(\epsilon,\delta,\eta))L^2<0.
  \]
 \end{ex}
   
\section{Stability versus Uniform Stability}\label{App2}

In this section, we construct polarized surfaces $(X,L)$ with ample divisors $H$ such that $(X,L)$ are $\mathrm{J}^H$-stable but not uniformly $\mathrm{J}^H$-stable. 

G. Chen proved the uniform version of Lejmi-Sz{\'e}kelyhidi conjecture in \cite{G}. The original conjecture (cf. \cite{LS}, \cite{CS}) was also solved recently as follows:
\begin{thm}[Theorem 1.2 of \cite{DP}, Corollary 1.2 of \cite{S}]
Notations as in Theorem \ref{modLSconj}. (1) is equivalent to the following:
\begin{itemize}
\item[(7)]
\[
\int_V c_0\omega^p_0-p\chi\wedge\omega_0^{p-1}> 0
\]
for all $p$-dimensional subvarieties $V$ with $p=1,2,\cdots,n-1$.
\end{itemize}
\end{thm}
It is proved in \cite[Corollary 1.3]{DP} when $M$ is projective and in \cite[Corollary 1.2]{S} for general K\"{a}hler manifolds. We can easily prove the original conjecture in projective smooth surfaces holds as follows. 
\begin{prop}\label{mat}
In Corollary \ref{jstable} (3), there exists $\epsilon>0$ such that for any integral curve $C$, 
\[
\left(2\frac{H\cdot L}{L^2}L-H\right)\cdot C\ge \epsilon L\cdot C
\]
iff
\[
\left(2\frac{H\cdot L}{L^2}L-H\right)\cdot C>0
\]
for any integral curve $C$.
\end{prop}

In fact, $(2\frac{H\cdot L}{L^2}L-H)^2=H^2>0$ and if $C\cdot (2\frac{H\cdot L}{L^2}L-H)>0$ for any integral curve, $2\frac{H\cdot L}{L^2}L-H$ is ample by the Nakai-Moishezon criterion. Therefore, there exists $\epsilon>0$ such that $(2\frac{H\cdot L}{L^2}-\epsilon)L-H$ is ample. Hence we have Proposition \ref{mat} by Corollary \ref{jstable}. For K\"{a}hler manifolds, the solvability of J-equation is equivalent to the class $2\frac{H\cdot L}{L^2}\mathrm{c}_1(L)-\mathrm{c}_1(H)$ is K{\" a}hler by \cite{SW}.

 However, J-stability and uniform J-stability is not equivalent as following. 

\begin{thm}\label{JJJ}
There are smooth polarized surfaces $(X,L)$ with ample divisors $H$ such that $(X,L)$ are $\mathrm{J}^H$-stable but not uniformly stable. In particular, if K{\" a}hler metrics $\omega_0$ and $\chi$ satisfy that $\omega_0\in \mathrm{c}_1(L)$ and $\chi\in\mathrm{c_1}(H)$, there is no smooth function $\varphi$ that satisfies the J-equation:
\[
\mathrm{tr}_{\omega_{\varphi}}\chi=\frac{H\cdot L}{L^2},
\]
where
\[
\omega_{\varphi}=\omega_0+\sqrt{-1}\partial \bar{\partial}\varphi
\]
is a positive $(1,1)$-form.
\end{thm}

\begin{proof}
The second assertion follows from the first assertion and Theorem 1.1 of \cite{G}. For the first assertion, we can construct an example such that $L,H$ is ample but there is an integral curve $C_0$ on $X=\mathbb{P}_{\mathbb{P}^1}(\mathcal{O}\oplus \mathcal{O}(e))$ for $e>0$, such that $(2\frac{L\cdot H}{L^2}L-H)\cdot C_0=0$ as follows:

  $C_0$ be the section of $\mathbb{P}^1$ satisfies that $C_0^2=-e$ and $f$ denote the numerical class of fibre. By Corollary 2.18 in \cite{Ha}, $aC_0+bf$ is ample $\Leftrightarrow b>ae>0$. Therefore, $aC_0+aef$ is nef for $ a>0$. Let $L=mC_0+nf$ and \begin{align*}
H&=\frac{2na}{2nm-m^2e}(mC_0+nf)-(aC_0+aef)\\
&=\frac{m^2ea}{2nm-m^2e}C_0+\frac{(2n(n-me)+m^2e^2)a}{2nm-m^2e}f 
\end{align*}
for $0<me<n$ and $a>0$. Here, we see that $H$ is ample. Then we can show that $2\frac{L\cdot H}{L^2}L-H=aC_0+aef$ by an easy computation. Note that $(aC_0+aef)\cdot C_0=0$ and $(X,L)$ is not uniformly $\mathrm{J}^H$-stable by Lemma \ref{unst}. Next, we prove the following:
\begin{prop}\label{p}
For any polarized deminormal surface $(X,L)$ with an ample $\mathbb{Q}$-line bundle $H$ such that all irreducible components of $(X,L)$ have the same average scalar curvature with respect to $H$ and
\[
2\frac{L\cdot H}{L^2}L-H
\]
is nef, $(X,L)$ is $\mathrm{J}^H$-stable.
\end{prop}

\begin{proof}[Proof of Proposition \ref{p}]
We can easily reduce to the case when $X$ is normal and irreducible by taking the normalization. Thus we may assume that $X$ is normal and irreducible. Suppose that $\mathfrak{a}$ is a flag ideal induces a nontrivial semiample test configuration. As in the proof of Theorem \ref{d2}, there exists an alternation $\pi:X'\to X$ due to Theorem \ref{slope} and
\[
\overline{\pi^{-1}\mathfrak{a}}=\mathscr{I}_{D_0}+\mathscr{I}_{D_1}t+\cdots +\mathscr{I}_{D_{r-1}}t^{r-1}+t^r,
\]
where each $D_i$ is a Cariter divisor, satisfies that the condition (*) and $(\pi^{-1}\mathfrak{a})\pi^*L_1$ is semiample and $\pi^*L-D_0$ is nef. Let $L'=\pi^*L$ and $H'=\pi^*H$. Now, we can see the following by Theorem \ref{impl},
\begin{align*}
p^*H'\cdot \mathcal{L}'^2-\frac{2H'\cdot L'}{3L'^2}\mathcal{L}'^3&=(3L'^2)^{-1}\Biggl(\sum_{i=0}^{r-2}\biggl(  6((D_i+D_{i+1})\cdot L')(L'\cdot H')\\
&-3((D_i+D_{i+1})\cdot H')(L'^2)-2(L'\cdot H')(D_i^2+D_i\cdot D_{i+1}+D_{i+1}^2)\biggr) \\
&+6(D_{r-1}\cdot L')(L'\cdot H')-3(D_{r-1}\cdot H')(L'^2)-2(L'\cdot H')(D_{r-1}^2))\Biggr).
\end{align*}
Note that $D_0\ge D_1\ge\cdots \ge D_{r-1}\ge 0$, where $D\ge D'$ if $D-D'$ is effective, and we may assume that $D_{r-1}\ne 0$. Therefore, we can prove that $$
6((D_i+D_{i+1})\cdot L')(L'\cdot H')-3((D_i+D_{i+1})\cdot H')(L'^2)-2(L'\cdot H')(D_i^2+D_i\cdot D_{i+1}+D_{i+1}^2)\ge 0 $$
by applying Proposition \ref{deft} as in the proof of Theorem \ref{d2} since $L'$, $H'$ and $2\frac{L'\cdot H'}{L'^2}L'-H'$ are nef and big and all the $L'-D_i$ and $L'-\frac{D_i+D_{i+1}}{2}$ are pseudoeffective (cf. Remark \ref{kh}). Hence, we have only to prove 
\begin{eqnarray}\label{value}
6(D_{r-1}\cdot L')(L'\cdot H')-3(D_{r-1}\cdot H')(L'^2)-2(L'\cdot H')(D_{r-1}^2)>0.
\end{eqnarray}
 As in the proof of Proposition \ref{deft}, let $B=\frac{L'\cdot H'}{L'^2}L'-H'\not \equiv 0$ and $E=D_{r-1}- \frac{D_{r-1}\cdot L'}{L'^2}L'+\frac{D_{r-1}\cdot B}{B^2}B$. Since $H'$ is nef and big, we have $H'^2>0$ and $B^2> -\frac{(H'\cdot L')^2}{L'^2}$. Hence, it is necessary for the value (\ref{value}) to be zero that $E^2=0$ and $D_{r-1}\cdot B=0$. Equivalently, there exists $t\in\mathbb{Q}$ such that
\[
D_{r-1}\equiv tL'. 
\]
However, if $D_{r-1}\equiv tL'$, then $0< t\le 1$ since $L'-D_{r-1}$ is pseudoeffective (if $t=0$, $\mathfrak{a}$ induces an almost trivial test configuration) and we see 
\[
6(D_{r-1}\cdot L')(L'\cdot H')-3(D_{r-1}\cdot H')(L'^2)-2(L'\cdot H')(D_{r-1}^2)\ge t(L'^2)(H'\cdot L')>0.
\]
We complete the proof of the proposition.
\end{proof}
 Therefore, we have constructed $\mathrm{J}^H$-stable smooth surfaces but they are not uniformly $\mathrm{J}^H$-stable. We finish the proof of Theorem \ref{JJJ}.
\end{proof}

We apply Theorem \ref{JJJ} to construct the following log pairs. We remark that they are not log Fano (cf. \cite[Theorem 1.5]{LXZ}).

\begin{cor}\label{K}
There exist polarized normal pairs $(X,\Delta;L)$ such that they are K-stable but not uniformly K-stable.
\end{cor}
\begin{proof}
In the proof of Theorem \ref{JJJ}, fix a polarization $L=mC_0+nf$. By Bertini's theorem, we can take a general ample $\mathbb{Q}$-divisor $\Delta$ and $K_{(X,\Delta+C_0)}\equiv H$ for sufficiently large $a$ so that $(X,\Delta+C_0)$ is lc and its lc center contains $C_0$. Since $(X,\Delta+C_0;L)$ is $(\mathcal{J}^{K_{(X,\Delta+C_0)}})^\mathrm{NA}$-stable as we saw and has only lc singularities, we have $H^\mathrm{NA}_{\Delta+C_0}\ge 0$. Therefore, $(X,\Delta+C_0;L)$ is K-stable. However, if $\mathcal{X}$ is the deformation to the normal cone along $C_0$, then $H^\mathrm{NA}_{\Delta+C_0}(\mathcal{X},L_{\mathbb{A}^1}-\delta E)=0$ for $\delta>0$, where $E$ is the exceptional divisor (cf. \cite[Theorem 4.8, Corollary 7.18]{BHJ}). For any $\epsilon>0$, there exists $\delta>0$ such that 
\[
(\mathcal{J}^{K_{(X,\Delta+C_0)}})^\mathrm{NA}(\mathcal{X},L_{\mathbb{A}^1}-\delta E)< \epsilon J^\mathrm{NA}(\mathcal{X},L_{\mathbb{A}^1}-\delta E)
\]
 due to Lemma \ref{unst} and hence $(X,\Delta+C_0;L)$ is not uniformly K-stable. Therefore we have the corollary.
\end{proof}

Hence, we propose the following conjecture:

\begin{conj}
There exists a polarized normal variety $(X,L)$ such that it is K-stable but not uniformly K-stable.
\end{conj}

On the other hand, if $X$ is deminormal, we have the following result:

\begin{cor}\label{DJ}
There exist polarized connected and deminormal surfaces $(Z,M)$ such that $K_Z$ are ample and $(Z,M)$ are $\mathrm{J}^{K_Z}$-stable and K-stable but not uniformly either.
\end{cor}

\begin{proof}
Take $(X,\Delta+C_0,L)$ as in the proof of Corollary \ref{K} such that $\Delta$ is an integral and ample $\mathbb{Z}$-divisor. By Lemma \ref{glue}, we obtain an slc scheme $Z\cong X\cup_{\Delta+C_0}X$ and let $M$ be the line bundle on $Z$ obtained by gluing $L_1$ and $L_2$ together i.e. $\nu^*M=L_1\times L_2$, where $\nu:X_1\sqcup X_2\to Z$ is the normalization and $(X_i,L_i)$ is a copy of $(X,L)$ for $i=1,2$. We have $\nu^*K_Z=K_{(X_1,\Delta+C_0)}\times K_{(X_2,\Delta+C_0)}$ and hence $K_Z$ is ample. Note also that all irreducible components of $(Z,M)$ have the same average scalar curvature with respect to $K_Z$. Recall from \cite[Remark 3.19]{BHJ} that $(Z,M)$ is K-stable (resp. uniformly K-stable) iff for any semiample test configuration $(\mathcal{Z},\mathcal{M})$ over $(Z,M)$, $M^\mathrm{NA}_{(\Delta+C_0)_1+(\Delta+C_0)_2}(\overline{\mathcal{Z}},\overline{\mathcal{M}})>0$ (resp. $M^\mathrm{NA}_{(\Delta+C_0)_1+(\Delta+C_0)_2}(\overline{\mathcal{Z}},\overline{\mathcal{M}})\ge\epsilon J^\mathrm{NA}(\overline{\mathcal{Z}},\overline{\mathcal{M}})$ for some $\epsilon>0$), where $(\overline{\mathcal{Z}},\overline{\mathcal{M}})$ is the normalization of $(\mathcal{Z},\mathcal{M})$. Since $(X,\Delta+C_0,L)$ is $\mathrm{J}^{K_{(X,\Delta+C_0)}}$-stable and K-stable, so is $(Z,M)$. On the other hand, let $\mathcal{Z}$ be the deformation to the normal cone of $Z$ along a closed subscheme $D_0$ contained in the node and $E$ be the exceptional divisor. Here, $D_0$ is the irreducible component of the conductor subscheme in $Z$, which is the image of $(C_0)_1\sqcup (C_0)_2$. It is easy to see that the inverse image $\nu^{-1}D_0\supset (C_0)_1\sqcup (C_0)_2$. To prove the converse, let $\mathscr{J}$ be the coherent ideal of $\mathcal{O}_Z$ corresponding to $D_0$ and $\mathscr{I}$ be the one of $\mathcal{O}_{X_1\sqcup X_2}$ corresponding to $(C_0)_1\sqcup (C_0)_2$. Note that $(C_0)_1\sqcup (C_0)_2$ is a part of the conductor and it is easy to see that $\nu^{-1}\mathscr{J}= \mathscr{I}$ holds for $\coprod_{k=1}^2(X_k\setminus(\Delta\cap C_0)_k)$. Since $\mathscr{I}$ is $S_2$, we have
\[
\mathscr{I}\subset \nu^{-1}\mathscr{J}.
\]
Therefore, we have that $\nu^{-1}D_0=(C_0)_1\sqcup (C_0)_2$ scheme-theoretically. Then $\overline{\mathcal{Z}}$ is the normalization of the deformation to the normal cone of $X_1\sqcup X_2$ along $(C_0)_1\sqcup (C_0)_2$ and the pullback $\overline{E}$ of $E$ is the exceptional divisor. Since for any $\epsilon>0$, there exists sufficiently small $\delta>0$ such that 
\[
M^\mathrm{NA}_{(\Delta+C_0)_1+(\Delta+C_0)_2}(\overline{\mathcal{Z}},L_1\times L_2-\delta\overline{E})<\epsilon J^\mathrm{NA}(\overline{\mathcal{Z}},L_1\times L_2-\delta\overline{E})
\]
and
\[
(\mathcal{J}^{K_{(X_1,\Delta+C_0)}\times K_{(X_2,\Delta+C_0)}})^\mathrm{NA}(\overline{\mathcal{Z}},L_1\times L_2-\delta\overline{E})<\epsilon J^\mathrm{NA}(\overline{\mathcal{Z}},L_1\times L_2-\delta\overline{E})
\]
due to Proposition \ref{p} and Corollary \ref{K}, $(Z,M)$ is neither uniformly $\mathrm{J}^{K_Z}$-stable nor uniformly K-stable.
\end{proof}

\begin{rem}
There is a more general way to construct such examples as in Theorem \ref{JJJ} by the application of the following consequences of Sj\"{o}str\"{o}m Dyrefelt \cite{Sj}. We extend them to singular projective surfaces:
\begin{thm}[cf. Corollary 7, Corollary 34 of \cite{Sj}]\label{sjo}
Let $X$ be a normal and irreducible surface, $H$ be a fixed ample $\mathbb{R}$-divisor on $X$, $\mathrm{Amp}_X$ be its ample cone and $\mathrm{Big}_X$ be the cone of $NS(X)\otimes\mathbb{R}$ that consists of big $\mathbb{R}$-divisors, where $NS(X)$ is N\'{e}ron-Severi group. Let the subsets $\mathrm{uJs}^H_X$ and $\mathrm{Js}^H_X$ of $\mathrm{Amp}_X$ be 
\[
\mathrm{uJs}^H_X=\{L\in \mathrm{Amp}_X;2\frac{L\cdot H}{L^2}L-H \, \mathrm{is}\, \mathrm{ample}\}
\]
and 
\[
\mathrm{Js}^H_X=\{L\in \mathrm{Amp}_X;2\frac{L\cdot H}{L^2}L-H \, \mathrm{is}\, \mathrm{nef}\}
\] 
respectively. Then $\mathrm{uJs}^H_X$ is open connected, and star convex, i.e. the line segment from $L$ to $H$ is contained in $\mathrm{uJs}^H_X$ for any $L\in \mathrm{uJs}^H_X$. Moreover, $\mathrm{Js}^H_X$ is the closure of $\mathrm{uJs}^H_X$ in $\mathrm{Amp}_X$.
\end{thm}

\begin{rem}
Abusing the notations, we denote also a numerical class of a $\mathbb{Q}$-line bundle $L$ by $L$. By Corollary \ref{jstable}, Proposition \ref{p} and the fact $H^2>0$, we have if $H$ is a $\mathbb{Q}$-line bundle, $\mathrm{uJs}^H_X\cap (NS(X)\otimes\mathbb{Q})$ is the set of polarizations of $X$ such that $(X,L)$ is uniformly $\mathrm{J}^H$-stable and $\mathrm{Js}^H_X\cap (NS(X)\otimes\mathbb{Q})$ is the set of $L$ such that $(X,L)$ is $\mathrm{J}^H$-stable.
\end{rem}

\begin{proof}[Proof of Theorem \ref{sjo}]
By the definition, it follows that $\mathrm{uJs}^H_X$ is open and $\mathrm{Js}^H_X$ is closed. Due to Proposition \ref{p}, we only have to show the following:
\begin{claim}
If $L\in \mathrm{Js}^H_X$, then $L+tH\in \mathrm{uJs}^H_X$ for $t>0$.
\end{claim}
Indeed, $2(L\cdot H)L-(L^2)H$ is nef by the assumption and
\begin{align}\label{mm}
2\frac{(L+tH)\cdot H}{(L+tH)^2}(L+tH)-H&=2\frac{L\cdot H+tH^2}{L^2+2tL\cdot H+t^2H^2}(L+tH)-H\\
&=\frac{1}{L^2+2tL\cdot H+t^2H^2}(2(L\cdot H)L-(L^2)H+2t(H^2)L+t^2(H^2)H)\nonumber
\end{align}
is ample.
\end{proof}

If $H$ is $\mathbb{Q}$-ample, it follows from Theorem \ref{sjo} that all rational points $L$ contained in the boundary $\partial(\mathrm{uJs}^H_X)$ in $\mathrm{Amp}_X$ satisfy that $(X,L)$ is $\mathrm{J}^H$-stable but not uniformly $\mathrm{J}^H$-stable. Next, we extend the following to the case when $X$ has singularities:

\begin{thm}[cf. Theorem 6 \cite{Sj}]\label{dyre}
Notations as in Theorem \ref{sjo}. If $\mathrm{Big}_X=\mathrm{Amp}_X$, 
\[
\mathrm{uJs}^H_X=\mathrm{Amp}_X.
\]
 Otherwise, let $M\in\partial\mathrm{Amp}_X$ be $M^2 =H^2$ and $L_t$ be any ample divisors on $X$ such that $L_t := (1-t)M+tH$ for $t \in (0, 1)$. Then 
 \[
 \mathrm{uJs}^H_X=\{\lambda L_t : \lambda > 0, t \in (\frac{1}{2}, 1]\} \subset \mathrm{Amp}_X.
 \]
\end{thm}

\begin{proof}
For the first assertion, it is easy to see that for any ample $\mathbb{R}$-divisor $L$, there are $t>0$ and a nef but not ample $\mathbb{R}$-divisor $N$ such that 
\[
L=tH+N.
\]
Note that $N$ is not big by the assumption. Then $N^2=0$. Therefore, as in the equation (\ref{mm}) in the proof of Theorem \ref{dyre},
\[
2\frac{(N+tH)\cdot H}{(N+tH)^2}(N+tH)-H=\frac{1}{2tN\cdot H+t^2H^2}(2(N\cdot H)N+2t(H^2)N+t^2(H^2)H)
\]
is ample.

For the second assertion, we only have to show that $L_t\in \mathrm{Js}^H_X\setminus \mathrm{uJs}^H_X$ when $t=\frac{1}{2}$ by Theorem \ref{sjo}. If $t=\frac{1}{2}$, $L=\frac{1}{2}(H+M)$ satisfies that 
\[
2\frac{L\cdot H}{L^2}L-H=M
\]
is nef but not ample. We complete the proof.
\end{proof}
Thus, if there exists $\mathbb{Q}$-divisors $H,M$ as in Theorem \ref{dyre}, we can construct examples as Theorem \ref{JJJ}. 

Finally, we remark that we can not extend Theorems \ref{sjo} and \ref{dyre} to the case when $X$ is deminormal in general. In fact, let $X$ be a deminormal surface and $X_i$'s are irreducible components of $X$. In fact, in Theorem \ref{dyre}, if $$\frac{M|_{X_i}^2}{H|_{X_i}^2}=\frac{M|_{X_j}^2}{H|_{X_j}^2}$$ and $$\frac{M|_{X_i}\cdot H|_{X_i}}{H|_{X_i}^2}=\frac{M|_{X_j}\cdot H|_{X_j}}{H|_{X_j}^2}$$ for $i,j$, then the theorem holds for $X$. To be precise, we have Proposition \ref{below} below. Otherwise, the theorem does not hold since there exists $t>0$ such that all irreducible components of $(X,H+tM)$ do not have the same average scalar curvature with respect to $H$.
\end{rem}

\begin{prop}\label{below}
Let $X$ be a projective, deminormal and reducible surface and $L,H$ be $\mathbb{Q}$-ample line bundles. Suppose that $L,H$ satisfy that 
\begin{align*}
\frac{L|_{X_i}^2}{H|_{X_i}^2}&=\frac{L|_{X_j}^2}{H|_{X_j}^2};\\
\frac{L|_{X_i}\cdot H|_{X_i}}{H|_{X_i}^2}&=\frac{L|_{X_j}\cdot H|_{X_j}}{H|_{X_j}^2}
\end{align*} 
for $i\ne j$, where $X=\coprod _{i=1}^lX_i$ is the irreducible decomposition. If $(X,L)$ is $\mathrm{J}^H$-stable, then $(X,L+tH)$ is uniformly $\mathrm{J}^H$-stable for $t>0$.
\end{prop}

\begin{proof}
Due to Theorem \ref{sjo}, we have only to prove that all irreducible components of $(X,L+tH)$ have the same average scalar curvature with respect to $H$. Indeed, 
\begin{align*}
\frac{(L+tH|_{X_i})\cdot H|_{X_i}}{(L+tH|_{X_i})^2}&=\frac{L|_{X_i}\cdot H|_{X_i}+t(H|_{X_i})^2}{t^2(H|_{X_i})^2+2t(L|_{X_i}\cdot H|_{X_i})+(L|_{X_i})^2}\\
&=\frac{t+\frac{L|_{X_i}\cdot H|_{X_i}}{H|_{X_i}^2}}{t^2+2t\left(\frac{L|_{X_i}\cdot H|_{X_i}}{H|_{X_i}^2}\right)+\frac{L|_{X_i}^2}{H|_{X_i}^2}}\\
&=\frac{(L+tH|_{X_j})\cdot H|_{X_j}}{(L+tH|_{X_j})^2}.
\end{align*}
\end{proof}

\section{J-stability for Higher Dimensional Varieties}\label{Fib}

In this section, we prove two criteria of J-stability for higher dimensional polarized deminormal schemes, Theorems \ref{fu} and \ref{SW}. We also explain applications of them and prove K-stability of log minimal models. Let $(X,L)$ be a deminormal polarized scheme and $H$ be a $\mathbb{Q}$-line bundle on $X$. Define that $(X,L)$ and $H$ satisfy the condition (**) if all irreducible components of $(X,L)$ have the same average scalar curvature with respect to $H$.

\subsection{Song-Weinkove criterion and its applications to stability threshold}

  First, recall the following criterion, which is proved by Song-Weinkove \cite{SW}. 
   
   \begin{thm}[Theorem 1.1 \cite{SW}]\label{SW2}
   Let $X$ be a K\"{a}hler manifold and $\chi$, $\omega$ be K\"{a}hler $(1,1)$-forms on $X$. Then the $(n-1,n-1)$-form 
   \[
   \left(n\frac{\int_X\chi \wedge \omega^{n-1}}{\int_X\omega^n}\omega-(n-1)\chi\right)\wedge \chi^{n-2} 
   \]
    is positive iff $\mathcal{J}^{\chi}$-energy is proper, where an $(n-1,n-1)$-form $\beta$ is positive if $\beta\wedge\alpha\wedge\bar{\alpha}>0$ for any point of $X$ and for any nonzero $(1,0)$-form $\alpha$.   
   \end{thm}
   
   From an algebro-geometric perspective, the next theorem follows from Theorems \ref{SW2} and \ref{modLSconj}:
   
    \begin{thm}\label{SW1}
   Let $(X,L)$ be a smooth polarized variety over $\mathbb{C}$ and $H$ be an ample $\mathbb{Q}$-line bundle on $X$. Suppose that $$n\frac{H \cdot L^{n-1}}{L^n}L-(n-1)H $$ is ample. Then there exists $\delta>0$ such that 
   \[
   (\mathcal{J}^{H})^{\mathrm{NA}}\ge\delta J^{\mathrm{NA}}
   \]
    on $\mathcal{H}^{\mathrm{NA}}(L)$.  
   \end{thm}
  We call this theorem Song-Weinkove criterion for J-stability. This is a weak form not only of Theorem \ref{SW2} but also of Theorem \ref{modLSconj}. 
  
  We generalize the following variant of Song-Weinkove criterion, which is partially known by \cite{HK} and \cite{Fu}, with a purely algebro-geometric proof:

  \begin{thm}[cf. Theorem 3 \cite{HK}, Theorem 6.5 \cite {Fu}]\label{fu}
   Let $(X,L)$ be an $n$-dimensional polarized deminormal scheme with a $\mathbb{Q}$-line bundle $H$ on $X$ such that $(X,L)$ and $H$ satisfy (**). Assume that $L^{n-1}\cdot H>0$ (resp. $\ge0$) and $$n^2\frac{H\cdot L^{n-1}}{L^n}L-(n^2-1)H$$ is ample (resp. nef). Then there exists $\epsilon>0$ (resp. $\ge 0$) such that $$(\mathcal{J}^{H})^\mathrm{NA}\ge \epsilon (I^\mathrm{NA}-J^\mathrm{NA}).$$
   \end{thm}
   
   \begin{proof}
We may assume that $X$ is normal and irreducible by taking the normalization due to Lemma \ref{demibunkai}. If $n=1$, $L$ is numerically equivalent to a positive (resp. nonnegative) multiple of $H$ by the assumption and hence the theorem immediately follows. Thus we may assume that $n\ge 2$. Let 
\[
A=\left(\frac{n^2}{n^2-1}\frac{H\cdot L^{n-1}}{L^n}-\delta\right)L-H
\]
 for $\delta\ge0$. The assertion when $L^{n-1}\cdot H>0$ and $\frac{n^2H\cdot L^{n-1}}{L^n}L-(n^2-1)H$ is ample follows from when $H\cdot L^{n-1}\ge0$ and $\frac{n^2H\cdot L^{n-1}}{L^n}L-(n^2-1)H$ is nef. In fact, since $\frac{n^2H\cdot L^{n-1}}{L^n}L-(n^2-1)H$ is ample and $H\cdot L^{n-1}>0$, $A$ is nef and $(H-(n^2-1)\delta L)\cdot L^{n-1}\ge0$ for sufficiently small $\delta>0$. Replace $H$ by $H-(n^2-1)\delta L$ for such $\delta$. Then, if the theorem when $H\cdot L^{n-1}\ge0$ and $\frac{n^2H\cdot L^{n-1}}{L^n}L-(n^2-1)H$ is nef would hold, $$(\mathcal{J}^{H})^\mathrm{NA}-(n^2-1)\delta(I^\mathrm{NA}-J^\mathrm{NA})=(\mathcal{J}^{H-(n^2-1)\delta L})^\mathrm{NA}\ge0.$$ 
 Therefore, we can take $\epsilon=(n^2-1)\delta$.
 
 Let $\delta=0$. Then we have
   \[
  \frac{A\cdot L^{n-1}}{L^n}=\frac{1}{n^2-1}\frac{H\cdot L^{n-1}}{L^n}\ge0
   \]
   and
   \[
   \frac{n^2}{n^2-1}\frac{H\cdot L^{n-1}}{L^n}(I^\mathrm{NA}-J^\mathrm{NA})=(\mathcal{J}^{A})^\mathrm{NA}+(\mathcal{J}^{H})^\mathrm{NA}.
   \]
   For any non-Archimedean metric $\phi$ and any representative $(\mathcal{X},\mathcal{L})$ of $\phi$ that dominates $X_{\mathbb{A}^1}$, we may assume that the support of $\mathcal{L}-L_{\mathbb{A}^1}$ does not contain the strict transformation of $X\times\{0\}$ by translation. Then we have $A_{\mathbb{P}^1}\cdot(\mathcal{L}-L_{\mathbb{P}^1})\cdot L_{\mathbb{P}^1}^{n-1}=0$. On the other hand,
   \[
   V(L)(\mathcal{J}^{A})^\mathrm{NA}=A_{\mathbb{P}^1}\cdot\mathcal{L}^n+nV(L)\frac{A\cdot L^{n-1}}{L^n}J^\mathrm{NA}.
   \]
    Since 
    \begin{align*}
   A_{\mathbb{P}^1}\cdot\mathcal{L}^n&=A_{\mathbb{P}^1}\cdot(\mathcal{L}^n-L_{\mathbb{P}^1}^n)\\
   &=\sum_{j=0}^{n-1} A_{\mathbb{P}^1}\cdot(\mathcal{L}-L_{\mathbb{P}^1})\cdot (\mathcal{L}^j\cdot L_{\mathbb{P}^1}^{n-1-j})\\
   &=\sum_{j=0}^{n-1} A_{\mathbb{P}^1}\cdot(\mathcal{L}-L_{\mathbb{P}^1})\cdot ((\mathcal{L}^j-L_{\mathbb{P}^1}^j)\cdot L_{\mathbb{P}^1}^{n-1-j})\\
   &=\sum_{j=0}^{n-1} A_{\mathbb{P}^1}\cdot(\mathcal{L}-L_{\mathbb{P}^1})^2\cdot \left(\sum_{i=0}^{j-1}\mathcal{L}^i\cdot L_{\mathbb{P}^1}^{n-2-i}\right)\\
   &\le0
   \end{align*} 
   by \cite[Lemma 1]{LX} and $I^\mathrm{NA}\ge\frac{n+1}{n}J^\mathrm{NA}$, we have
   \begin{align*}
   (\mathcal{J}^{H})^\mathrm{NA}&\ge\frac{n^2}{n^2-1}\frac{H\cdot L^{n-1}}{L^n}(I^\mathrm{NA}-J^\mathrm{NA})-n\frac{A\cdot L^{n-1}}{L^n}J^\mathrm{NA}\\
   &\ge\frac{n}{n^2-1}\frac{H\cdot L^{n-1}}{L^n}J^\mathrm{NA}-\frac{n}{n^2-1}\frac{H\cdot L^{n-1}}{L^n}J^\mathrm{NA}\\
   &=0 .
   \end{align*}
   \end{proof}
  
  \begin{rem}
  We do not assume that $H$ is nef or pseudoeffective in Theorem \ref{fu}. Therefore, Theorem \ref{fu} is not weaker than Theorem \ref{SW} below over $\mathbb{C}$.
  \end{rem}
  
   We apply Theorem \ref{fu} to extend the following for deminormal schemes:
   
     \begin{cor}[cf. Theorem 3.15 \cite{DR2}, Lemma 6.9 \cite{Li}]\label{Lil}
    For any polarized deminormal scheme $(X,L)$ and any $\mathbb{Q}$-line bundle $H$ on $X$ such that $(X,L)$ and $H$ satisfy (**), there exists a constant $\delta>0$ such that
   \[
   -\delta J^{\mathrm{NA}}\le (\mathcal{J}^{H})^\mathrm{NA}\le \delta J^{\mathrm{NA}}.
   \]
   \end{cor}
   
   \begin{rem}
   C. Li \cite{Li} proved Corollary \ref{Lil} for smooth projective manifolds and $M$ is their canonical divisor $K_X$ over $\mathbb{C}$ by applying Theorem \ref{SW1}. On the other hand, Dervan and Ross show Corollary \ref{Lil} for normal varieties in the proof of \cite[Theorem 3.15]{DR2}.
   \end{rem}
   
   \begin{proof}[Proof of Corollary \ref{Lil}]
  We may assume that $X$ is normal and irreducible. Because $J^{\mathrm{NA}}$ and $I^{\mathrm{NA}}-J^{\mathrm{NA}}$ is comparable, in the above inequality we can use the latter to compare. Let $M=H+\delta L$. Then $\frac{M\cdot L^{n-1}}{L^n}=\frac{H\cdot L^{n-1}}{L^n}+\delta$. For sufficiently large $\delta$, $M$ is ample and
   \begin{align*}
   n\frac{M\cdot L^{n-1}}{L^n}L-\frac{n^2-1}{n}M=&n\left(\frac{H\cdot L^{n-1}}{L^n}+\delta\right)L-\frac{n^2-1}{n}(H+\delta L) \\
   =&\left(n^{-1}\delta + n\frac{H\cdot L^{n-1}}{L^n}\right)L-\frac{n^2-1}{n}H
   \end{align*}
  is ample. In this case, we get $(\mathcal{J}^{H})^\mathrm{NA}+ \delta (I^{\mathrm{NA}}-J^{\mathrm{NA}})=(\mathcal{J}^{M})^\mathrm{NA}\ge 0$ by Proposition \ref{fu}.
  
  On the other hand, let $M=-H+\delta L$. For sufficiently large $\delta$, we can prove the latter part of the corollary similarly.
   \end{proof}

   Thanks to Corollary \ref{Lil}, we can introduce the notion of the stability threshold of Sj\"{o}str\"{o}m Dyrefelt \cite{Sj}.
   \begin{de}
   Let $(X,B;L)$ be a polarized deminormal pair. Then the {\it non-Archimedean log K-stability threshold} is
   \[
    \Delta_B(X,L)=\sup\{\delta\in \mathbb{R}; M^\mathrm{NA}_B\ge \delta (I^\mathrm{NA}-J^\mathrm{NA}) \}.
   \]
   On the other hand, let $H$ be a $\mathbb{Q}$-line bundle such that irreducible components of $(X,L)$ have the same average scalar curvature with respect to $H$. Then the {\it non-Archimedean $J^H$-stability threshold} is
   \[
   \Delta_H^{pp}(X,L)=\sup\{\delta\in \mathbb{R}; (\mathcal{J}^H)^\mathrm{NA}\ge \delta (I^\mathrm{NA}-J^\mathrm{NA}) \}.
   \]
   \end{de}
   By Corollary \ref{Lil}, we can easily show that both of $\Delta_B(X,L)$ and $\Delta_H^{pp}(X,L)$ are well-defined for slc pairs whose irreducible components have the same average scalar curvature. We can extend the basic properties of $\Delta_H^{pp}$ of \cite{Sj} for possibly singular varieties as follows:
   \begin{thm}[cf. Theorem 17 of \cite{Sj}]\label{thsjo}
   Suppose that $X$ is a deminormal projective surface. For any two $\mathbb{Q}$-line bundles $L,H$, where $L$ is ample and $H$ is pseudoeffective such that $(X,L)$ and $H$ satisfy (**) but $(X,L)$ is not uniformly $\mathrm{J}^H$-stable, then the stability threshold satisfies 
   \[
    \Delta_H^{pp}(X,L)=2\frac{H\cdot L}{L^2}-\inf\{\delta>0: \delta L-H \, \mathrm{is}\, \mathrm{ample}\}.
   \]
   \end{thm}
  We mimic the proof of \cite{Sj}. To prove the theorem, we prepare the followings.
   \begin{lem}[cf. Lemma 14 \cite{Sj}]\label{14}
  $X$ as in Theorem \ref{thsjo}. Suppose that $H$ is a $\mathbb{Q}$-line bundle on $X$ and $L$ is an ample $\mathbb{Q}$-line bundle. Suppose that $(X,L)$ and $H$ satisfy (**). Let $a,b\in \mathbb{Q}$ with $a\ge 0$. Then
   \[
   \Delta_{aH+bL}^{pp}(X,L)=a\Delta_{H}^{pp}(X,L)+b.
   \]
   \end{lem}
   The proof of Lemma \ref{14} is straightforward. Note that all irreducible components of $(X,L)$ have the same average scalar curvature with respect to $aH+bL$ for $a,b\in\mathbb{R}$ by the assumption.
   
   Let $H$ be a $\mathbb{Q}$-line bundle on $X$ and $L$ is an ample $\mathbb{Q}$-line bundle. Suppose that $(X,L)$ and $H$ satisfy (**). Decompose the two dimensional subspace $\mathrm{Span}(H,L)$ of $NS(X)\otimes\mathbb{R}$ spanned by $H$ and $L$ into two components
   \[
   \mathrm{Span}(H,L)=\mathrm{Span}(H,L)^+\cup\mathrm{Span}(H,L)^-,
   \]
   where
   \[
   \mathrm{Span}(H,L)^+\coloneq\{ aH+bL: a\ge 0, b\in \mathbb{R}\}
   \]
   and
   \[
    \mathrm{Span}(H,L)^-\coloneq\{ aH+bL: a\le 0, b\in \mathbb{R}\}.
    \]
    Then the following holds.
   \begin{lem}[cf. Lemma 15 \cite{Sj}]\label{15}
   Notations as above. Suppose that $M_0$ and $M_1$ are $\mathbb{Q}$-line bundles on $X$ such that either $(M_0,M_1)\in \mathrm{Span}(H,L)^+\times\mathrm{Span}(H,L)^+$ or $(M_0,M_1)\in\mathrm{Span}(H,L)^-\times\mathrm{Span}(H,L)^-$. Then
   \[
   \Delta^{pp}_{(1-t)M_0+tM_1}(X,L)=(1-t)\Delta^{pp}_{M_0}(X,L)+t\Delta^{pp}_{M_1}(X,L)
   \]
   for any $t\in [0,1]\cap\mathbb{Q}$.
   \end{lem}
   \begin{proof}
   First, we prove the lemma when $(M_0,M_1)\in \mathrm{Span}(H,L)^+\times\mathrm{Span}(H,L)^+$. Suppose that $M_0=a_0H+b_0L$ and $M_1=a_1H+b_1L$, where $a_0,a_1\in\mathbb{Q}_{\ge0}$ and $b_0,b_1\in\mathbb{Q}$. For each $t\in [0,1]\cap\mathbb{Q}$,
   \begin{align*}
   (\mathcal{J}^{(1-t)M_0+tM_1})^\mathrm{NA}&=(1-t)(\mathcal{J}^{M_0})^\mathrm{NA}+t(\mathcal{J}^{M_1})^\mathrm{NA}\\
   &=(a_0(1-t)+a_1t)(\mathcal{J}^{H})^\mathrm{NA}+(b_0(1-t)+b_1t)(I^\mathrm{NA}-J^\mathrm{NA}).
   \end{align*}
   Therefore, we have
   \begin{align*}
   \Delta^{pp}_{(1-t)M_0+tM_1}(X,L)&=(a_0(1-t)+a_1t)\Delta^{pp}_{H}(X,L)+(b_0(1-t)+b_1t)\\
   &=(1-t)\Delta^{pp}_{M_0}(X,L)+t\Delta^{pp}_{M_1}(X,L)
   \end{align*}
   by Lemma \ref{14}.
   
   On the other hand, if $(M_0,M_1)\in\mathrm{Span}(H,L)^-\times\mathrm{Span}(H,L)^-$, then $(M_0,M_1)\in \mathrm{Span}(-H,L)^+\times\mathrm{Span}(-H,L)^+$. Hence the lemma holds.
   \end{proof}
   
   \begin{proof}[Proof of Theorem \ref{thsjo}]
   Let $H_t=(1-t)H+tL$ for $t\in [0,1]\cap \mathbb{Q}$, $R(t)=\Delta_{H_t}^{pp}(X,L)$ and $L(t)=2\frac{H_t\cdot L}{L^2}-\inf\{\delta>0: \delta L-H_t \, \mathrm{is}\, \mathrm{ample}\}$. By Lemma \ref{15}, $R(t)$ is affine linear and extended to all of $[0,1]$. On the other hand, $L(t)$ is also affine linear and extended to all of $[0,1]$. It is easy to see that $R(1)=L(1)=1$. We have only to show that there exists $t_0<1$ such that $R(t_0)=L(t_0)$. This follows immediately from the following claim:
   \begin{claim}
   For $t\in [0,1]\cap \mathbb{Q}$, $L(t)<0$ if and only if $R(t)<0$.
   \end{claim}
   Indeed, by Corollary \ref{ms}, $R(t)\ge 0$ if and only if $2\frac{H_t\cdot L}{L^2}L-H_t$ is nef. The latter condition is equivalent to $L(t)\ge 0$.
   \end{proof}
   Finally, we remark that the following theorem:
   \begin{thm}[cf. Theorem 18 \cite{Sj}]
   Suppose that $(X,L)$ is a polarized deminormal surface. If $H$ is a $\mathbb{Q}$-line bundle on $X$ such that $(X,L)$ and $H$ satisfy (**) and $\Delta^{pp}_H(X,L)<\mathcal{T}(H,L)$, where
   \[
   \mathcal{T}(H,L)\coloneq \sup\{ \delta\in\mathbb{R} | H-\delta L\, \mathrm{is}\,\mathrm{pseudoeffective}\},
   \]
   then
   \begin{align}\label{thank}
   \Delta_H^{pp}(X,L)=2\frac{H\cdot L}{L^2}-\inf\{\delta>0: \delta L-H \, \mathrm{is}\, \mathrm{ample}\}.
   \end{align}
   \end{thm}
   \begin{proof}
  Note that $\mathcal{T}(H,L)\ge 0$ and the left and the right hand sides are positive homogeneous in $H$. Let $\delta_0\ge 0$ be a rational number such that $\Delta^{pp}_H(X,L)<\delta_0\le\mathcal{T}(H,L)$ and then we have $(X,L)$ is $\mathrm{J}^H$-unstable and $H-\delta_0 L$ is a pseudoeffective $\mathbb{Q}$-line bundle by the assumption. As in the proof of Theorem \ref{thsjo}, let $H_t=(1-t)(H-\delta_0L)+tL$, $R(t)=\Delta_{H_t}^{pp}(X,L)$ and $L(t)=2\frac{H_t\cdot L}{L^2}-\inf\{\delta>0: \delta L-H_t \, \mathrm{is}\, \mathrm{ample}\}$. Note that $L(t)$ and $R(t)$ is defined only on $t\in[0,1]$ such that $H_t$ is a $\mathbb{Q}$-line bundle. However, we can extend these to affine linear functions defined on $[0,1]$ similarly. Note that $L(1)=R(1)=1$. On the other hand, we can find $t_0<1$ such that $L(t_0)=R(t_0)$. Therefore, $L(t)=R(t)$ for any $t\in[0,1]$ and we complete the proof by $L(\frac{\delta_0}{1+\delta_0})=R(\frac{\delta_0}{1+\delta_0})$.
   \end{proof}

   \subsection{Generalized Lejmi-Sz\'{e}kelyhidi conjecture and K-stability of minimal models}
   In this subsection, we work over $\mathbb{C}$. We extend Lejmi-Sz\'{e}kelyhidi conjecture to the case when $X$ is deminormal and explain its applications to K-stability of klt minimal models. 
   
  \begin{thm}[Lejmi-Sz\'{e}kelyhidi conjecture for deminormal polarized schemes]\label{SW}
Let $(X,L)$ be a deminormal polarized scheme over $\mathbb{C}$ and $H$ be an ample (resp. nef) $\mathbb{Q}$-line bundle on $X$ such that $(X,L)$ and $H$ satisfy (**). Then the followings are equivalent.
\item[(1)] $(X,L)$ is uniformly $\mathrm{J}^H$-stable (resp. $\mathrm{J}^H$-semistable). In other words, there exists $\epsilon>0$ such that for any semiample test configuration $(\mathcal{X},\mathcal{L})$
\[
(\mathcal{J}^H)^\mathrm{NA}(\mathcal{X},\mathcal{L})\ge \epsilon J^\mathrm{NA}(\mathcal{X},\mathcal{L})\quad (\mathrm{resp}.\, \ge0).
\] 
\item[(2)] $(X,L)$ is uniformly slope $\mathrm{J}^H$-stable (resp. slope $\mathrm{J}^H$-semistable). In other words, there exists $\epsilon>0$ such that for any semiample deformation to the normal cone $(\mathcal{X},\mathcal{L})$ along any integral curve
\[
(\mathcal{J}^H)^\mathrm{NA}(\mathcal{X},\mathcal{L})\ge \epsilon J^\mathrm{NA}(\mathcal{X},\mathcal{L})\quad (\mathrm{resp}.\, \ge0).
\] 
\item[(3)] There exists $\epsilon>0$ such that for any $p$-dimensional subvariety $V$, 
\[
\left(n\frac{H\cdot L^{n-1}}{L^n}L-pH\right)\cdot L^{p-1}\cdot V\ge (n-p)\epsilon L^p\cdot V\quad (\mathrm{resp}.\, \ge0).
\]
\end{thm}

\begin{proof}
We have only to prove $(3) \Rightarrow (1)$ due to Lemma \ref{unst2}. First, we prove the assertion if $X$ is normal and irreducible and $\epsilon=0$. For sufficiently small $a>0$, $H+a L$ is ample and 
\[
\left(n\frac{(H+aL)\cdot L^{n-1}}{L^n}L-p(H+aL)\right)\cdot L^{p-1}\cdot V>0. 
\]
 Therefore, we may assume that $H$ and 
 \[
\left(n\frac{H\cdot L^{n-1}}{L^n}L-pH\right)\cdot L^{p-1}\cdot V>0. 
\]
 Take any semiample test configuration $(\mathcal{X},\mathcal{L})$ dominates $X_{\mathbb{A}^1}$ and a resolution of singularities $\widetilde{X}$ of $X$. Since we want to show $(\mathcal{J}^{H})^{\mathrm{NA}}$ is non-negative, we may assume that $\mathcal{L}$ is $\mathbb{A}^1$-ample by a small perturbation of $\mathcal{L}$. Let $E$ be an ample divisor on $\widetilde{X}$ and $\widetilde{\mathcal{X}}$ be the strict transformation of $\mathcal{X}$ in the normalization of $\mathcal{X}\times_X \widetilde{X}$. Now, we have the following commutative diagram.   $$
\begin{CD}
\widetilde{\mathcal{X}} @>{\widetilde{\pi}}>> \mathcal{X} \\
@V{\widetilde{\rho}}VV @V{\rho}VV \\
\widetilde{X}\times\mathbb{A}^1 @>{\pi\times \mathrm{id}_{\mathbb{A}^1}}>> X\times\mathbb{A}^1.
\end{CD}
$$
We can easily see that $(\widetilde{\mathcal{X}},\widetilde{\pi}^*\mathcal{L}+\epsilon \widetilde{\rho}^*E_{\mathbb{A}^1})$ is an ample test configuration of the polarized variety $(\widetilde{X},\pi^*L+\epsilon E)$ for any sufficiently small $\epsilon \in \mathbb{Q}_+$ since $\mathcal{L}$ is $\mathbb{A}^1$-ample and $\widetilde{\rho}^*E_{\mathbb{A}^1}$ is $\mathcal{X}$-ample. On the other hand, 
\[
\lim_{\delta,\epsilon\to 0}(\mathcal{J}^{(\pi^*H+\delta (\pi^*L+\epsilon E))})^\mathrm{NA}(\widetilde{\mathcal{X}},\widetilde{\pi}^*\mathcal{L}+\epsilon \widetilde{\rho}^*E_{\mathbb{A}^1}) =(\mathcal{J}^{H})^\mathrm{NA}(\mathcal{X},\mathcal{L}).
\]
Thus we only have to show that
\[
(\mathcal{J}^{(\pi^*H+\delta (\pi^*L+\epsilon E))})^\mathrm{NA}(\widetilde{\mathcal{X}},\widetilde{\pi}^*\mathcal{L}+\epsilon \widetilde{\rho}^*E_{\mathbb{A}^1}) \ge 0
\]
for rational $0<\epsilon\ll \delta\ll 1$ since $(\mathcal{J}^H)^\mathrm{NA}$-energy is linear in $H$. For any $p$-dimensional subvariety $V$ of $\widetilde{X}$, 
\begin{align*}
&\left(n\frac{\pi^*H\cdot (\pi^*L+\epsilon E)^{n-1}}{(\pi^*L+\epsilon E)^n}(\pi^*L+\epsilon E)-p\pi^*H\right)\cdot (\pi^*L+\epsilon E)^{p-1}\cdot V\\
&=\left((n\frac{H\cdot L^{n-1}}{L^n}+O(\epsilon))(\pi^*L+\epsilon E)-p\pi^*H\right)\cdot (\pi^*L+\epsilon E)^{p-1}\cdot V,
\end{align*}
where $O(\epsilon)$ is independent of $V$. On the other hand, we prove that
\[
\left(n\frac{H\cdot L^{n-1}}{L^n}(\pi^*L+\epsilon E)-p\pi^*H\right)\cdot (\pi^*L+\epsilon E)^{p-1}\cdot V\ge 0.
\]
Indeed, if $p-r>q$, the $\epsilon^r$-term of the left hand side of the above inequality is 0, where $q$ is the dimension of $\pi(V)$. Otherwise, we consider $\pi_*(E^r\cdot V)$ as a positive $(p-r)$-cycle of $X$ and the easy computation shows that the $\epsilon^r$-term is
\begin{align*}
&\epsilon^r\left(n\frac{H\cdot L^{n-1}}{L^n}\binom{p}{r}\pi^*L^{p-r}-p\binom{p-1}{r}(\pi^*H\cdot \pi^*L^{p-r-1})\right)\cdot (E^r\cdot V)\\
 &=\epsilon^r\binom{p}{r}\left(n\frac{H\cdot L^{n-1}}{L^n}L-(p-r)H\right)\cdot L^{p-r-1}\cdot \pi_*(E^r\cdot V)\ge 0
\end{align*}
by the assumption. Therefore, for any $\delta>0$, there exists $\epsilon_0>0$ such that 
\[
\left(n\frac{(\pi^*H+\delta(\pi^*L+\epsilon E))\cdot (\pi^*L+\epsilon E)^{n-1}}{(\pi^*L+\epsilon E)^n}(\pi^*L+\epsilon E)-p\pi^*H\right)\cdot (\pi^*L+\epsilon E)^{p-1}\cdot V> 0
\]
for any $0<\epsilon\le \epsilon_0$. Thanks to Theorem \ref{modLSconj}, 
\[
(\mathcal{J}^{(\pi^*H+\delta (\pi^*L+\epsilon E))})^\mathrm{NA}(\widetilde{\mathcal{X}},\widetilde{\pi}^*\mathcal{L}+\epsilon \widetilde{\rho}^*E_{\mathbb{A}^1}) \ge 0.
\]
Thus we prove $(3) \Rightarrow (1)$ when $\epsilon =0$ and $X$ is normal and irreducible.

 If $X$ is normal and irreducible, $H  $ is ample and $\epsilon>0$, uniform $\mathrm{J}^H$-stability follows immediately from the same argument of the proof of Theorem \ref{fu}. Finally, if $X$ is deminormal, the assertion follows immediately by taking its normalization.
\end{proof}
   
   \begin{rem}
   If $n=2$, Theorem \ref{SW} follows from Corollary \ref{jstable}. 
   \end{rem}
   
   \begin{rem}\label{remstable}
   The proof of Theorem \ref{SW} also shows that the following. If $(3)\Rightarrow (1)$ holds for all smooth polarized varieties $(X,L)$ over any algebraically closed field $k$ of characteristic 0 with nef $\mathbb{Q}$-Cartier divisors, then it also holds for all polarized deminormal schemes with nef $\mathbb{Q}$-line bundles that satisfy (**).
   \end{rem}
   
  Now, we have the following without assuming the log canonical divisor is semiample due to Theorem \ref{SW} with a more algebro-geometric approach than Sj\"{o}str\"{o}m Dyrefelt \cite{Zak} and J. Song \cite{S}. 
   
   \begin{thm}[cf. Theorem 1.1 of Jian-Shi-Song \cite{JSS}, Theorem 1.1 of Sj\"{o}str\"{o}m Dyrefelt \cite{Zak}, Corollary 1.3 of J. Song \cite{S}]\label{lastcor}
   Let $(X,\Delta ;L)$ be an $n$-dimensional polarized slc minimal model over $\mathbb{C}$ i.e. $K_{(X,\Delta)}$ is nef. Suppose that $(X,L)$ and $K_{(X,\Delta)}$ satisfy (**). If $K_{(X,\Delta)}$ is also big, $(X,\Delta ;K_{(X,\Delta)}+\epsilon L)$ is uniformly $\mathrm{J}^{K_{(X,\Delta)}}$-stable for sufficiently small $\epsilon>0$. Otherwise, suppose that $(X,\Delta)$ is klt. Then $(X,\Delta ;K_{(X,\Delta)}+\epsilon L)$ is uniformly K-stable for sufficiently small $\epsilon>0$. If $\Delta=0$ and $X$ is smooth, $(X,K_{X}+\epsilon L)$ also has a cscK metric.
   \end{thm}
   
   \begin{proof}
  We may assume that $X$ is normal and irreducible. If $K_{(X,\Delta)}^{n}\ne 0$,
   \[
   n\frac{K_{(X,\Delta)} \cdot (K_{(X,\Delta)}+\epsilon L)^{n-1}}{(K_{(X,\Delta)}+\epsilon L)^n}(K_{(X,\Delta)}+\epsilon L)-(n-1)K_{(X,\Delta)}= (1+O(\epsilon))K_{(X,\Delta)}+\epsilon(n+O(\epsilon))L 
   \]
   is ample for sufficiently small $\epsilon>0$. Therefore, the first assertion follows from Theorem \ref{SW}. 
   
  Next, suppose that $(X,\Delta)$ is klt and irreducible. Let $m=\nu(K_X)\le n$. Suppose that $V\subset X$ is a $p$-dimensional subvariety and $\nu(K_X|_V)=j\le \min\{m,p \}$. Then for sufficiently small $\epsilon>0$, since
   \begin{align*}
   n\frac{K_{(X,\Delta)}\cdot(K_{(X,\Delta)}+\epsilon L)^{n-1}}{(K_{(X,\Delta)}+\epsilon L)^n}=m+O(\epsilon),
   \end{align*}
   we have
   \begin{align*}
  &\left(n\frac{K_{(X,\Delta)}\cdot(K_{(X,\Delta)}+\epsilon L)^{n-1}}{(K_{(X,\Delta)}+\epsilon L)^n}(K_{(X,\Delta)}+\epsilon L)-pK_{(X,\Delta)}\right)\cdot (K_{(X,\Delta)}+\epsilon L)^{p-1}\cdot V\\
   &=((m-p)K_{(X,\Delta)}+m\epsilon L+O(\epsilon)(K_{(X,\Delta)}+\epsilon L))\cdot(K_{(X,\Delta)}+\epsilon L)^{p-1}\cdot V,
   \end{align*}
  where $O(\epsilon)$ is independent of $V$. Hence, there exists a constant $C>0$ independent of $V$ such that 
   \begin{align*}
  &\left(n\frac{K_{(X,\Delta)}\cdot(K_{(X,\Delta)}+\epsilon L)^{n-1}}{(K_{(X,\Delta)}+\epsilon L)^n}(K_{(X,\Delta)}+\epsilon L)-pK_{(X,\Delta)}\right)\cdot (K_{(X,\Delta)}+\epsilon L)^{p-1}\cdot V\\
   &>((m-p)K_{(X,\Delta)}+m\epsilon L-C\epsilon(n-p)(K_{(X,\Delta)}+\epsilon L))\cdot(K_{(X,\Delta)}+\epsilon L)^{p-1}\cdot V.
   \end{align*}
   Now, we want to prove the following:
   \begin{claim}
   $V,m,p$ and $j$ as above. Then
   \[
   ((m-p)K_{(X,\Delta)}+m\epsilon L)\cdot(K_{(X,\Delta)}+\epsilon L)^{p-1}\cdot V\ge 0
   \]
   for sufficiently small $\epsilon>0$.
   \end{claim}
   The theorem follows from this claim. In fact, it is easy to see that $(X,K_{(X,\Delta)}+\epsilon L)$ is $\mathrm{J}^{K_{(X,\Delta)}+\epsilon C(K_{(X,\Delta)}+\epsilon L) }$-semistable by Theorem \ref{SW} since $K_{(X,\Delta)}+\epsilon C(K_{(X,\Delta)}+\epsilon L)$ is ample. Then \[
   (\mathcal{J}^{K_{(X,\Delta)}})^\mathrm{NA}\ge -C\epsilon (I^\mathrm{NA}-J^\mathrm{NA})
   \]
  on $\mathcal{H}^\mathrm{NA}(L)$. Therefore, the first assertion follows from the fact that
   \[ 
   \alpha(X,\Delta;K_{(X,\Delta)}+\epsilon L)\ge \alpha(X,\Delta;K_{(X,\Delta)}+ L )>0
   \]
   for $0<\epsilon <1$. Indeed, it is easy to see the fact holds by the definition of the alpha invariant. Then we have only to choose $\epsilon$ so small that $C\epsilon <\alpha(X,\Delta;K_{(X,\Delta)}+ L )$. The second assertion follows from \cite[Theorem 6.10]{Li}.
   
   Now, we prove the claim. If $m\ge p$, the claim holds. Otherwise, $j\le m<p$ and hence the coefficient of $\epsilon^{p-i}$-term of $((m-p)K_{(X,\Delta)}+m\epsilon L)\cdot(K_{(X,\Delta)}+\epsilon L)^{p-1}\cdot V$ is
   \[
   \left((m-p)\binom{p-1}{p-i}+m\binom{p-1}{p-i-1}\right)K_{(X,\Delta)}^{i}\cdot L^{p-i}\cdot V=\binom{p}{i}(m-i)K_{(X,\Delta)}^{i}\cdot L^{p-i}\cdot V\ge0
   \]
   for $1\le i\le j$ or zero for $i>j$. On the other hand, we can easily see the term is positive when $i=0$. 
   \end{proof}
   
   \begin{rem}
   If $n=2$ in Theorem \ref{lastcor}, we can replace Theorem \ref{SW} by Theorem \ref{d2} in the above proof. On the other hand, there is a purely algebro-geometric proof of uniform K-stability of $(X,\Delta ;L)$ such that $K_{(X,\Delta)}\equiv 0$ for any polarization (cf. \cite{Od}, \cite{D}, \cite{BHJ}). Thus we have a purely algebro-geometric proof of uniform K-stability of klt log minimal surfaces with certain polarizations. 
   \end{rem}
   
   \begin{rem}
   Jian-Shi-Song \cite{JSS} prove Theorem \ref{lastcor} when $\Delta=0$ and $K_X$ is semiample. Theorem \ref{lastcor} is also proved without the assumption that $K_X$ is semiample by Z. Sj\"{o}str\"{o}m Dyrefelt \cite{Zak} and by J. Song \cite{S}.
   \end{rem}
   
   \begin{rem}
   There is another application. We extend \cite[Theorem 1.2]{JSS} to the case for deminormal varieties in \cite{Hat}. On the other hand, we prove in \cite{Hat} the following another extension of \cite{JSS}:
   \begin{thm}
   Let $(X,\Delta ;H)$ be an $n$-dimensional polarized klt variety over $\mathbb{C}$. Suppose that $X$ admits a flat fibration $f: X\to B$ over a polarized smooth curve $(B,L)$ and $H\equiv K_{(X,\Delta)}+f^*L_0$, where $L_0$ is a line bundle on $B$. If \[
   n\frac{K_{(X,\Delta)}\cdot H^{n-1}}{H^n}-(n-1)\ge 0
   \]
     for normal and irreducible fibre $(X_b,\Delta_b;H_b)$ over general point $b$ of $B$, then $(X,\Delta;f^*L+\delta H)$ is uniformly K-stable for sufficiently small $\delta>0$. Furthermore, if $X$ is smooth and $\Delta =0$, $(X,f^*L+\delta H)$ has a cscK metric. 
   \end{thm}
   \end{rem}

 \end{document}